\documentclass[11pt]{article}
\usepackage{geometry}                % See geometry.pdf to learn the layout options. There are lots.
\usepackage{graphicx}
\usepackage{amssymb}    
\usepackage{epstopdf}
\usepackage{hyperref}
\usepackage[T1]{fontenc}
\usepackage[ngerman, english]{babel}
\usepackage{mathtools}
\usepackage{mathabx}
\usepackage{bbm}
\usepackage{mathrsfs}
\usepackage{amsmath,amscd}
\usepackage{amsthm}
\usepackage{tikz}
\usepackage{tikz-cd}
\usepackage{babel}
\usetikzlibrary{babel}
\usetikzlibrary{arrows, matrix}
\usetikzlibrary{cd}
\usepackage{stmaryrd}
\usepackage[arrow, matrix, curve]{xy}

\newtheorem{innercustomthm}{Theorem}
\newenvironment{customthm}[1]
  {\renewcommand\theinnercustomthm{#1}\innercustomthm}
  {\endinnercustomthm}

 \usepackage{enumerate}
\newtheorem{theorem}{Theorem}[section]
\newtheorem{lemma}[theorem]{Lemma}
\newtheorem{proposition}[theorem]{Proposition}
\newtheorem{corollary}[theorem]{Corollary}

\newtheorem{definition}[theorem]{Definition}

\DeclareMathOperator{\Gal}{\operatorname{Gal}}
\DeclareMathOperator{\Q}{\mathbf{Q}}
\DeclareMathOperator{\R}{\mathrm{R}}
\DeclareMathOperator{\Z}{\mathbf{Z}}
\DeclareMathOperator{\F}{\mathbf{F}}

\DeclareMathOperator{\N}{\mathbf{N}}

\DeclareMathOperator{\Spec}{\operatorname{Spec}}

\DeclareMathOperator{\Hom}{\operatorname{Hom}}

\DeclareMathOperator{\ord}{\operatorname{ord}}
\DeclareMathOperator{\Aut}{\operatorname{Aut}}

\DeclareMathOperator{\Lie}{\mathrm{Lie}}

\DeclareMathOperator{\Ext}{\operatorname{Ext}}

\DeclareMathOperator{\Id}{\mathrm{Id}}
\DeclareMathOperator{\Og}{\mathcal{O}}
\DeclareMathOperator{\pr}{\mathrm{pr}}

\DeclareMathOperator{\Pic}{\mathrm{Pic}}

\DeclareMathOperator{\Br}{\mathrm{Br}}
\DeclareMathOperator{\et}{\acute{\mathrm{e}}{\mathrm{t}}}
\DeclareMathOperator{\proet}{\mathrm{pro}\acute{\mathrm{e}}{\mathrm{t}}}
\DeclareMathOperator{\Gm}{\mathbf{G}_m}
\DeclareMathOperator{\Ga}{\mathbf{G}_a}

\DeclareMathOperator{\fppf}{\mathrm{fppf}}
\DeclareMathOperator{\sHom}{\mathscr{H}\!\mathit{om}}
\DeclareMathOperator{\sExt}{\mathscr{E}\!\mathit{xt}}
\DeclareMathOperator{\sEnd}{\mathscr{E}\!\mathit{nd}}

\DeclareMathOperator{\Res}{\mathrm{Res}}

\DeclareMathOperator{\sep}{{^\mathrm{sep}}}

\DeclareMathOperator{\RP}{{\mathrm{RP}}}

\DeclareMathOperator{\CH}{{\mathrm{CH}}}

\theoremstyle{remark}

\newtheorem{remark}[theorem]{Remark}

\newtheorem{example}[theorem]{Example}
\DeclareMathOperator{\ffm}{\mathfrak{m}}
\newcommand{\indtorus}[2]{\Res_{{#1/#2}}{\mathbf{G}_{\mathrm{m}}}}

\newcommand{\powser}[2]{{#1}[\![{#2}]\!]}

\title{Unirational algebraic groups and tame ramification}
\author{Otto Overkamp and Ismaele Vanni}
\date{}
\begin{document}
\maketitle
{\abstract{Let $\Og_K$ be a complete discrete valuation ring with field of fractions $K$ and algebraically closed residue field $k.$ Let $G$ be a smooth connected commutative algebraic group over $K$ which does not contain a copy of $\Ga.$ For each $d$ prime to $p:=\mathrm{char}\, k,$ let $K(d)$ be the unique extension of $K$ of degree $d.$ We investigate how the Néron lft-model of $G$ behaves under base change to the ring of integers $\Og_{K(d)}.$ Information about this behaviour is encoded in the \it jumps \rm of Edixhoven's filtration on the special fibre of the Néron lft-model of $G,$ as well as in Halle-Nicaise's motivic zeta function of $G.$ If $G$ is unirational (e. g. an algebraic torus), we show that the jumps of $G$ are rational numbers and that the motivic zeta function of $G$ is a rational function. We also deduce analogous results for Abelian varieties with potentially totally multiplicative reduction. This answers a question of Halle-Nicaise and partially one of Edixhoven. Along the way, we answer a question of Oesterlé about the structure of unipotent algebraic groups over function fields in positive characteristic. Under stronger conditions on $G,$ we obtain rationality of jumps even for separably closed but imperfect $k.$}}
\tableofcontents
\section{Introduction}
Let $\Og_K$ be a complete discrete valuation ring with field of fractions $K$ and separably closed residue field $k$ of characteristic $p>0.$ A smooth algebraic group $G$ over $K$ will be called \it wound \rm if $\Hom_K(\Ga, G)=0$ (all group schemes appearing on this article will be commutative). All semiabelian varieties over $K$ are wound, but there also exist unipotent examples. A fundamental result about these groups is that they admit \it Néron lft-models \rm over $\Og_K,$ i. e. for each smooth connected wound algebraic group $G$ over $K,$ there exists a smooth separated group scheme $\mathscr{G} \to \Spec \Og_K$ with generic fibre $G$ such that, for each smooth morphism $S\to \Spec \Og_K$ and each map $\phi_K\colon S\times _{\Og_K} \Spec K \to G,$ there exists a unique $\Og_K$-morphism $\phi\colon S\to  \mathscr{G}$ extending $\phi_K$ \cite[Chapter 10.1, Theorem 2 (b')]{BLR}. Such a model of $G$ is unique up to unique isomorphism. 
Now let $F$ be a finite separable extension of $K$ and let $\Og_F$ be the integral closure of $\Og_K$ in $F;$ this is another complete discrete valuation ring which is finite and flat over $\Og_K.$ Suppose $G$ is a smooth connected wound algebraic group over $K$ with Néron lft-model $\mathscr{G} \to \Spec \Og_K.$  It is well-known that the base change $G_F$ is wound as well \cite[Corollary B.3.5]{CGP}, so it admits a Néron lft-model $\mathscr{G}_F$ over $\Og_F.$ However, the canonical morphism $\mathscr{G}\times_{\Og_K}\Spec\Og_F \to \mathscr{G}_F$ is usually not an isomorphism, and it is an interesting (and rather delicate) problem to predict this morphism's behaviour. One approach to this problem in the case of \it tame \rm extensions, which we shall now briefly describe informally, is due to Edixhoven \cite{Edixhoven}; it will be recalled more precisely in the next section.

For each $d$ prime to $p,$ there exists a tamely ramified extension $K\subseteq K(d)$ of degree $d,$ which is unique up to non-unique isomorphism. All finite tamely ramified extensions of $K$ are of this form. Edixhoven constructs a filtration $\mathscr{F}^\bullet\mathscr{G}_k$ on the special fibre of $\mathscr{G}$ indexed by the set $\Z_{\langle p \rangle}\cap [0,1] = \{\frac{i}{d}\colon 0\leq i \leq d, \gcd(d,p)=1\}$ which measures the failure of the canonical maps 
$$\mathscr{G}\times_{\Og_K}\Og_{K(d)} \to \mathscr{G}_{K(d)}$$ to induce isomorphisms on identity components. For example, $\mathscr{F}^{\frac{1}{d}}\mathscr{G}_k$ is the kernel of the canonical map $\mathscr{G}_k \to \mathscr{G}_{K(d),k}.$ It turns out that this filtration \it jumps \rm at finitely many elements $j_1\leq ... \leq j_n \in [0,1]$ (with $n=\dim G$). The jumps of a smooth connected wound algebraic group $G$ capture much information about the behaviour of the Néron lft-model of $G$ under tame base change. Because the index set is dense in $[0,1],$ it is not at all clear whether these are \it rational \rm numbers, which is the main open question about these invariants (\cite[p. 301]{Edixhoven}, \cite[Question 10.1.2]{HN}). In general, this question is wide open. It is known to have an affirmative answer in the following two cases: 
\begin{enumerate}
\item $G$ is a semiabelian variety which acquires semiabelian reduction over a finite \it tamely ramified \rm extension of $K$ \cite{HNII}, and
\item the residue field $k$ is algebraically closed and $G = \Pic^0_{C/K}$ for a smooth, projective, and geometrically connected curve $C$ over $K$ of index 1 \cite[Theorem 6.3.1.3]{HN} (see also \cite[Theorem 4.4.5]{EHN} and \cite[Corollary 1.8]{KMW} for different approaches using logarithmic geometry), or (more generally) a geometrically  a geometrically reduced curve with smooth geometrically integral normalization of index 1 and a \it push-out singularity \rm \cite{OvI}.\footnote{Several of the (rather strong) technical assumptions on $C$ imposed in \cite{OvI} were necessary because no general framework for constructing "good" models of singular curves was known at the time. Such a framework has since been developed by the first-named author \cite{OvIII, OvII}, so it seems likely that these results hold much more generally. This will be addressed in future work.} In particular, the answer is affirmative for induced tori, i.e. tori of the form  $\indtorus MK$ for a finite separable extension $M/K$. 
\end{enumerate}
Beyond that (for example, for wound unipotent groups), only individual examples are known. 

The main aim of the present article is to give an affirmative answer to this question for \it unirational\rm\footnote{An algebraic group $G$ is \it unirational \rm if its underlying scheme is, i. e. if it is (scheme-theoretically) dominated by an open subset of some affine space. Note that unirational algebraic groups are always smooth, connected, and affine.} algebraic groups if the residue field $k$ is algebraically closed. Using rigid uniformisation, we deduce the same result for Abelian varieties with potentially totally multiplicative reduction. To state the result more precisely, recall that, for any semiabelian variety $G$ over $K,$ there exists a finite Galois extension $K\subseteq L$ such that $G\times_KL$ has semiabelian reduction over $\Og_L.$

\begin{customthm}{A} \rm (Theorems \ref{jumpsratthm}, \ref{Abelianjumpsratthm}, and \ref{jumpsratthmGEN}) \it
Assume that $k$ is algebraically closed. Let $G$ be a smooth connected wound algebraic group over $K,$ and suppose moreover that either
\begin{enumerate}[(i)]
\item $G$ is unirational, or
\item $G$ is an Abelian variety which acquires \rm totally multiplicative \it semiabelian reduction over a finite extension $K\subseteq L.$ 
\end{enumerate}
Let $j$ be a jump of $G.$ Then $j$ is rational. In fact, we have $[L:K]j\in \Z$ for a finite Galois extension $K\subseteq L$ which splits $G$ in a suitable sense. If $G$ is an algebraic torus, $L$ can be taken to be the classical splitting field of $G.$
\end{customthm}
Note in particular that there is no tameness assumption (which would, in any case, only make sense for semiabelian $G$ as there is no notion of semiabelian reduction for more general algebraic groups), and that case (i) above incudes all algebraic tori. Remarkably, this result is already new for norm tori, i. e. kernels of the norm maps $N_{L/K} \colon \Res_{L/K}\Gm \to \Gm$ for finite separable extensions $K\subseteq L.$ 

Moreover, we shall investigate a finer invariant which captures \it recurring patterns \rm in the special fibre of the Néron lft-models of $G\times_K{K(d)}$ over $\Og_{K(d)}$ for $d$ prime to $p$: Following Halle-Nicaise \cite{HNII}, we shall consider the \it motivic zeta function \rm 
$$Z_G(x) := \sum_{p\nmid d} [\mathscr{G}_{K(d),k}^{\mathrm{qc}}]\mathbf{L}^{\mathrm{ord}_G(d)} x^d \in K_0(\mathrm{Var}_k)[\![x]\!].$$
Here, $K_0(\mathrm{Var}_k)$ denotes the Grothendieck ring of varieties over $k$ \cite{NS}, $[X]$ denotes the class of a smooth separated $k$-scheme $X$ of finite type over $k$ in $K_0(\mathrm{Var}_k),$ we put $\mathbf{L}:=[\mathbf{A}^1_k],$ and define $$\mathrm{ord}_G(d):=\ell_{\Og_{K(d)}}(\mathrm{coker}(\Lie \mathscr{G}\otimes_{\Og_K}\Og_{K(d)} \to \Lie\mathscr{G}_{K(d)})).$$ Moreover, if $\Delta$ is a smooth group scheme over $k$ with finitely generated group of connected components, $\Delta^{\mathrm{qc}}$ refers to the maximal quasi-compact\footnote{If $\Phi$ denotes the group of connected components of $\Delta,$ $\Delta^{\mathrm{qc}}$ is defined to be the preimage of $\Phi_{\mathrm{tors}}$ under the canonical map $\Delta\to \Phi.$ The condition imposed on $\Delta$ is always satisfied for special fibres of Néron lft-models by \cite[Proposition 3.5]{HNII}. The technicality of considering $\mathscr{G}(d)^{\mathrm{qc}}_k$ instead of $\mathscr{G}(d)_k$ in the definition of the motivic zeta function is necessary because $K_0(\mathrm{Var}_k)$ is defined as the set of equivalence classes of smooth separated \it quasi-compact \rm $k$-schemes with certain ring operations. Were one to allow all smooth $k$-schemes instead, $K_0(\mathrm{Var}_k)$ would be the trivial ring due to the existence of arbitrary disjoint unions in the category of smooth $k$-schemes; hence one would obtain an empty theory.} open subgroup scheme of $\Delta.$ The promised result concerning recurring patterns is then 
\begin{customthm}{B}\rm(Theorem \ref{Zetaratthm}) \it 
Assume that $k$ is algebraically closed. Suppose $G$ is a unirational algebraic group over $K$ (e. g. an algebraic torus) or an Abelian variety with potentially totally multiplicative reduction. Then $Z_G(x)$ is a rational function.
\end{customthm}
It is precisely the rationality of the jumps which lets one control the contribution of $\mathbf{L}^{\mathrm{ord}_G(d)}.$ The difficulty in analysing the contribution from $[\mathscr{G}_{K(d),k}^{\mathrm{qc}}]$ comes mainly from the need to control the number of irreducible components of $\mathscr{G}_{K(d),k}^{\mathrm{qc}},$ for which one needs a rather powerful duality theory for $G.$ Such a theory is well-known for algebraic tori (due to Bégueri, Suzuki, Xarles and others), as well as for unipotent groups (due to Suzuki \cite{SuzIII}), but some additional work is required in order to treat \it extensions \rm of such groups. Additionally, we use some of these methods to answer a question of Oesterlé about the structure of unipotent algebraic groups over function fields in positive characteristic, and use some of the methods from \cite{OS} to deduce an analogue of Chai's conjecture for tame base change conductors (see Definition \ref{ordZctamedef}).

Throughout Sections 1-7, we shall assume that $k$ is algebraically closed. In Section 8, we shall show how one can obtain rationality of jumps of an algebraic torus $T$ even for imperfect residue fields $k$ under stronger assumptions on $T$; this is the first (non-trivial) such result for wildly ramified tori in the case of imperfect residue field.\\
\\
\noindent$\mathbf{Acknowledgement}.$ The authors would like to express their gratitude to Professor A. Bertapelle, Professor L. H. Halle, Professor J. Nicaise, Professor S. Schröer, Dr. A. Bode, Dr. Z. Rosengarten, Dr. T. Suzuki, and Dr. A. Waetershoot for very helpful conversations. This research was conducted in the framework of the research training group \it GRK 2240: Algebro-geometric methods in Algebra, Arithmetic and Topology; \rm the first-named author was partially supported by the Deutsche Forschungsgemeinschaft through the grant \it SCHR 671/10-1 Varieties with Free Tangent Sheaf. \rm 
The second-named author's interest in Edixhoven's jumps started during the writing of his PhD thesis \cite{Vannithesis}, where some of the results of the last section first appeared. He would like to heartily thank his supervisor Professor A. Bertapelle for her great mentorship and help. His research has benefited from stays, partially funded by Sapienza University through the grants \textit{Mobilità Internazionale PhD 2022 II edizione}, \textit{Avvio alla Ricerca 2022/2023} and \textit{Avvio alla Ricerca 2023/2024}, at KU Leuven and Università di Padova, which provided excellent working conditions.
\section{Background}
\subsection{Edixhoven's filtration}
Let $d$ be a positive integer prime to $p$ and let $K(d)$ be the extension of $K$ of degree $d$ (this extension is unique up to non-canonical isomorphism). There is a canonical isomorphism $\Gal(K(d)/K)=\mu_d:=\mu_d(k).$ Let $G$ be a smooth connected wound algebraic group over $K.$ For each $d$ prime to $p,$ we shall denote by $\mathscr{G}(d)$ the Néron lft-model of $G(d):=G\times_K K(d)$ over $\Og_{K(d)}.$ Edixhoven constructs a filtration 
$$\mathscr{G}_k = F^0_d \mathscr{G}_k \supseteq ... \supseteq F^d_d \mathscr{G}_k = 0$$
on $\mathscr{G}_k$ by smooth closed subgroup schemes. On $k$-points, the $k$-group schemes $F^i_d\mathscr{G}_k$ can be described as follows: Note that $\Res_{\Og_{K(d)}/\Og_K}\mathscr{G}(d)$ is canonically isomorphic to the Néron lft-model of $\Res_{K(d)/K} (G\times_K K(d)).$ Let $\mathfrak{m}_d\subseteq \Og_{K(d)}$ be the maximal ideal. Then $F^i_d\mathscr{G}_k(k)$ is the kernel of the composition 
$$\mathscr{G}_k(k) \subseteq (\Res_{\Og_{K(d)}/\Og_{K}}\mathscr{G}(d))(k) = \mathscr{G}(d)(\Og_{K(d)}/\mathfrak{m}_d^d) \to \mathscr{G}(d)(\Og_{K(d)}/\mathfrak{m}_d^i)$$ for $1\leq i\leq d,$ where the first inclusion follows from \cite[Theorem 4.2]{Edixhoven}. The construction of the scheme structure requires the use of Greenberg functors; see \cite[Section 4.1]{HNII} for the details\footnote{In fact, $G$ is assumed to be an Abelian variety in \cite{Edixhoven}. However, the arguments carry over \it verbatim \rm to the more general situation \cite{HNII}. We shall later mostly restrict to unirational algebraic groups and some particular semiabelian varieties.}. This filtration can be seen as measuring how badly the Néron lft-model of $G$ fails to commute with \it tame \rm base change; for example, $F^1_d\mathscr{G}_k$ is precisely the kernel of the canonical map $\mathscr{G}_k \to \mathscr{G}(d)_k$ for each $d$ prime to $p.$  One checks that, for each $d, n$ prime to $p$ and each $0\leq i \leq d,$ we have a canonical identification 
$$F^{in}_{dn} \mathscr{G}_k = F^i_d\mathscr{G}_k$$ \cite[Lemma 4.11]{HNII}; in particular, for all $\alpha = i/d \in \Z_{\langle p \rangle}\cap [0,1],$ the $k$-subgroup scheme 
$$\mathscr{F}^{\alpha}\mathscr{G}_k:= F ^ i_d\mathscr{G}_k$$ is well-defined. This means that we obtain a decreasing separated and exhaustive filtration $\mathscr{F}^{\alpha}\mathscr{G}_k$ by smooth closed algebraic subgroups on $\mathscr{G}_k$ indexed by the topological space $\Z_{\langle p \rangle}\cap [0,1].$ The invariants of interest to us will be the \it jumps \rm of this filtration, which are \it a priori \rm real numbers. We recall the following 
\begin{definition}\label{multiplicitydef} \rm (Cf. \cite[Section 6.1.2]{HN}) \it 
\begin{enumerate}[(i)]
\item For $0\leq i \leq d-1,$ we put $$\mathrm{Gr}^i_d\mathscr{G}_k:=F^i_d\mathscr{G}_k / F^{i+1}_d \mathscr{G}_k.$$ We shall call an element $j\in \{0,..., d-1\}$ a \rm $d$-jump \it of $G$ if $\dim \mathrm{Gr}_d^j\mathscr{G}_k>0;$ this dimension will be called the \rm multiplicity \it of the $d$-jump $j.$ 
\item Let $\rho\in(0,1)$ and $1 \gg \epsilon>0.$ Then, for $j\in (\rho-\epsilon, \rho)\cap \Z_{\langle p\rangle}$ (resp. $j\in (\rho, \rho+\epsilon)\cap \Z_{\langle p\rangle}$), the algebraic groups $\mathscr{F}^{j}\mathscr{G}_k$ stabilise as $\epsilon\to 0;$ we denote this stable value by $\mathscr{F}^{<\rho}\mathscr{G}_k$ (resp. $\mathscr{F}^{>\rho}\mathscr{G}_k$). Moreover, we put $\mathscr{F}^{>1}\mathscr{G}_k:=0$ and $\mathscr{F}^{<0}\mathscr{G}_k:=\mathscr{G}_k.$ Finally, put
$$\mathrm{Gr}^{\rho}\mathscr{G}_k:=\mathscr{F}^{<\rho}\mathscr{G}_k/\mathscr{F}^{>\rho}\mathscr{G}_k.$$ We say that $j\in [0,1]$ is a \rm jump \it of $G$ if $\dim \mathrm{Gr}^j\mathscr{G}_k>0$ and call this dimension the \rm multiplicity \it of $j.$ 
\end{enumerate}
\end{definition} 
\begin{remark}\label{leftconremark}
    Consider the function $\phi_G\colon \Z_{\langle p \rangle} \cap [0,1] \to \mathbf{R}$ given by $\alpha \mapsto \dim \mathscr{F}^{\alpha}\mathscr{G}_k.$ This function takes only finitely many values, is monotonically decreasing, and satisfies $\phi_G(0) = \dim  G$ as well as $\phi_G(1)=0.$ It is easy to see that $\phi_G$ is discontinuous precisely at the elements of $\Z_{\langle p \rangle} \cap [0,1]$ which are also jumps of $G.$ The behaviour of $\phi_G$ at the points of discontinuity is not well-understood. Some of the arguments in \cite{HN} seem implicitly to rely on the assumption that $\phi_G$ is left-continuous (which would, in particular, imply that 1 is never a jump, as is stated without proof in \cite[p. 94]{HN}). This is true in all cases where $\phi_G$ is explicitly known; in particular, it can be easily checked for tamely ramified semiabelian varieties \cite[Proposition 4.13]{HNII} and for induced tori. However, it is not clear to the authors whether this left-continuity always holds; we shall give new proofs of the relevant results from \cite{HN} which avoid this assumption.
\end{remark}
We shall make essential use of the following entirely different way of describing the $d$-jumps of a smooth connected wound algebraic group $G$ over $K.$ For each $d$ prime to $p,$ the group $\mu_d=\Gal(K(d)/K)$ acts (from the right) on the Néron lft-model $\mathscr{G}(d)$ of $G\times_KK(d)$ in such a way that the action commutes with the group operation and the structure morphism $\mathscr{G}(d)\to \Spec \Og_{K(d)}$ is equivariant. In particular, if $\mathfrak{m}_{G(d)}\subseteq \widehat{\Og}_{\mathscr{G}(d)_k,0}$ denotes the maximal ideal of the completed local ring of $\mathscr{G}(d)_k$ at the origin, then $\mu_d$ acts (from the left) on the $k$-vector space $\mathfrak{m}_{G(d)}/\mathfrak{m}_{G(d)}^2.$ Recall that the category of $k[\mu_d]$-modules is semisimple; if $\chi_d$ denotes the one-dimensional representation on which a generator $\zeta\in \mu_d$ acts with weight 1, its simple objects are precisely the tensor powers $\chi_d^{\otimes j}$ for $0 \leq j \leq d-1.$ Then we have 
\begin{proposition} \label{jumpspropV}
Let $n:=\dim G$ and let $d$ be prime to $p.$ Let $j_{d,1} \leq ... \leq j_{d,n}$ be the $d$-jumps of $G,$ each listed as many times as indicated by its multiplicity. 
\begin{enumerate}[(i)] 
\item There exists an isomorphism 
$$\mathfrak{m}_{G(d)}/\mathfrak{m}_{G(d)}^2 \cong \chi_d^{\otimes j_{d,1}} \oplus ... \oplus \chi_d^{\otimes j_{d,n}}$$
of $k[\mu_d]$-modules. 
\item Let $\mathscr{G}$ and $\mathscr{G}(d)$ be the Néron lft-models of $G$ and $G\times_KK(d)$ over $\Og_K$ and $\Og_{K(d)},$ respectively. Then there exists an exact sequence
$$0 \to \Lie \mathscr{G}\otimes_{\Og_K}\Og_{K(d)}\to \Lie \mathscr{G}(d) \to \bigoplus_{i=1}^n \Og_{K(d)}/\langle \pi_{K(d)}^{j_{d,i}}\rangle \to 0$$ of $\Og_{K(d)}$-modules, where $\pi_{K(d)}\in \Og_{K(d)}$ is a uniformiser. 
\end{enumerate}
\end{proposition}
\begin{proof}
Note that the left-representation $\mathfrak{m}_{G(d)}/\mathfrak{m}_{G(d)}^2$ of $\mu_d$ is canonically dual to the right-representation $\Lie \mathscr{G}(d)_k.$ Therefore, (i) is precisely the statement of \cite[Corollary 4.8]{HNII}. Claim (ii) is \cite[Theorem 10.4]{HNII}. 
\end{proof}

Hence we can recover the $d$-jumps of $G$ together with their multiplicities from the $k[\mu_d]$-module $\mathfrak{m}_{G(d)}/\mathfrak{m}_{G(d)}^2,$ or from the cokernel of $\Lie \mathscr{G}\otimes_{\Og_K}\Og_{K(d)}\to \Lie \mathscr{G}(d)$ as an $\Og_{K(d)}$-module.

In general, if $0 \to G_1 \to G_2 \to G_3 \to 0$ is an exact sequence of smooth connected wound algebraic groups over $K,$ the induced sequence $0\to \mathscr{G}_1 \to \mathscr{G}_2 \to \mathscr{G}_3 \to 0$ of Néron lft-models over $\Og_K$ need not be exact. In particular, it is difficult to extract information about the jumps of $G_2$ from those of $G_1$ and $G_3.$ The following well-known argument by Chai (see e.g. \cite[Satz 4.2.2]{Brahm}, \cite[Remark 4.8]{Chai}, \cite[Lemma 4.3]{LorLiu}) measures this failure cohomologically. We recall the proof for the reader's convenience. Let $j\colon\Spec K\to\Spec\Og_K$ be the inclusion of the generic point into the spectrum of the discrete valuation ring $\Og_K$, and denote by $j^{\mathrm{sm}}\colon (\Spec K)_{\mathrm{sm}} \to (\Spec \Og_K)_{\mathrm{sm}}$ the induced morphism from the small smooth site of $\Spec K$ (i. e. schemes smooth over $K$ endowed with the étale topology) to the small smooth site of $\Spec \Og_K$ given by base change. If $G$ is a smooth connected wound algebraic group over $K$ with Néron lft-model $\mathscr{G},$ then $j_\ast^{\mathrm{sm}} G$ is represented by $\mathscr{G}.$

\begin{proposition}\label{Chaiexactn}
    Let $0\to G_1\to G_2\to G_3\to 0$ be a short exact sequence of smooth $K$-group schemes admitting Néron lft-models. If $R^1j^{\mathrm{sm}}_*G_1=0$, then the complex of Néron lft-models \[0\to \mathscr{G}_1\to \mathscr{G}_2\to \mathscr{G}_3\to 0\] is exact as a sequence of $\Og_K$-group schemes. The condition is satisfied if $G_1 = \Res_{F/K} \mathbf{G}_{\mathrm{m}}^s$ for a finite extension $K\subseteq F$ and $s\in \N.$ \end{proposition}
 \begin{proof}
     Let $f_K\colon G_2\to G_3$ be the morphism in the statement and $f\colon\mathscr{G}_2\to\mathscr{G}_3$ its extension to the Néron lft-models. By \cite[Lemma 4.3]{LorLiu}, $\ker f$ is a smooth $\Og_K$-group scheme, whereby it suffices to show that it is separated and satisfies the Néronian property. Separatedness follows from \cite[Tag 01KV]{Stacks}. The Néronian property follows from the fact that the generic fibre of a smooth $\Og_K$-scheme $S$ is scheme-theoretically dense in $S.$ The last claim follows from \cite[Remark 4.6]{Chai}.
 \end{proof}

\subsection{The jump group}
For some $d$ prime to $p,$ the $d$-jumps of $G$ are usually regarded as elements of the set $\{0,...,d-1\}\subseteq \Z.$ However, it will be more convenient for us to view them as elements of the \it group \rm $\Z/d\Z,$ for reasons that will become apparent later. We shall order $\Z/d\Z$ by declaring that $a\leq b$ if and only if the unique preimages of $a$ and $b$ in $\{0,...,d-1\}$ under the canonical map $\Z\to \Z/d\Z$ satisfy the corresponding relation. Similarly, we order the group $\mathbf{R}/\Z$ by descending the order on $[0,1)$ inherited from $\mathbf{R}.$ One should always bear in mind that the orderings we consider do \it not \rm respect the group structures. For any $d$ prime to $p,$ we have a canonical order-preserving map $\Z/d\Z \to\mathbf{R}/\Z$ given by $[n] \mapsto n/d;$ these can be patched together to obtain a canonical map 
$$\varinjlim \Z/d\Z \to \mathbf{R}/\Z$$ (where the index set is ordered by divisibility), the image of which is precisely $\Z_{\langle p\rangle}/\Z.$ From now on, the $d$-jumps of a smooth connected algebraic group $G,$ as well as the \it weights \rm of a representation of $\mu_d$ on a $k$-vector space, will be viewed as elements of $\Z/d\Z.$ The jumps of $G$ will be seen as elements of $\mathbf{R}/\Z;$ this allows us to consider $\Z$-linear combinations of such weights (or jumps). We note in particular that a jump of $G$ is rational if and only if it is a torsion element of $\mathbf{R}/\Z.$  Once we know the $d$-jumps of $G$ for sufficiently many $d,$ we automatically also know the jumps (cf. \cite[(6.1.3.7)]{HN}). To make this precise, we need the following definition (writing $\N=\{1,2,3,...\}$ and $\N_0:=\N\cup\{0\})$:
\begin{definition}
    A \rm grid sequence \it is a sequence $(d_{\ell})_{\ell\in \N}$ of positive integers prime to $p$ such that $d_{\ell}\mid d_{\ell+1}$ for all $\ell$ and such that $d_{\ell}\to\infty$ as $\ell\to \infty.$
\end{definition}

\begin{proposition} \label{jumpslimitprop}
Let $(d_{\ell})_{\ell\in \N}$ be a grid sequence. Let $j_{1, \ell}\leq ... \leq j_{n, \ell} \in \Z/d_{\ell}\Z$ be the $d_{\ell}$-jumps of $G,$ and let $j_1 \leq ... \leq j_n \in \mathbf{R}/\Z$ be the jumps of $G.$ Then, for each $i=1,...,n,$ we have
$$j_i = \lim_{\ell\to \infty} \frac{j_{i, \ell}}{d_{\ell}}$$ in $\mathbf{R}/\Z.$ 
\end{proposition}
The proof carries over without change from \cite[(6.1.3.7)]{HN}; the point is that the sequences $(j_{i,\ell}/d_\ell)_{\ell}$ are monotonically increasing and that the set $\{\nu/d_\ell\colon \ell \in \N, 0\leq \nu \leq d_{\ell}\}$ is dense in $\Z_{\langle p \rangle}\cap[0,1].$ As a matter of fact, with a bit more effort one can show that the conclusion of the previous proposition remains true if $(d_\ell)_\ell$ is \textit{any} sequence of integers coprime to $p$ tending to infinity; see Remark \ref{rmk:anysequence}.
%We note in particular that a jump of $G$ is rational if and only if it is a torsion element of $\mathbf{R}/\Z.$ 

The jumps of a smooth connected wound algebraic group $G$ over $K$ naturally form a finite multiset (i. e. an object similar to a finite set but which can contain the same element multiple times). By definition, the \it sum \rm of two finite multisets $\mathcal{M}$ and $\mathcal{N}$ contains precisely the elements of $\mathcal{M}\cup \mathcal{N};$ the multiplicity of an element in the sum is the sum of the multiplicities of that element in $\mathcal{M}$ and $\mathcal{N}.$ The sum of $\mathcal{M}$ and $\mathcal{N}$ will be denoted by $\mathcal{M}\uplus\mathcal{N}$. 
Moreover, for a smooth connected wound algebraic group $G$ over $K$ and a positive integer $d$ prime to $p$, we shall denote by $J(G)$ (resp. $J_d(G)$) the {multiset} of jumps (resp. $d$-jumps) of $G$. 

If the conclusion of Proposition \ref{Chaiexactn} is satisfied, there is a very close relationship between the ($d$-)jumps of the algebraic groups in a short exact sequence:

\begin{lemma} \label{universallyexactlem}
Suppose $\widetilde{G}$ is an algebraic group over $K$ which sits inside an exact sequence $0\to T \to \widetilde{G} \to H \to 0$ of smooth connected wound algebraic groups over $K.$ Suppose that, for some $d$ prime to $p,$ the induced sequence $0 \to \mathscr{T}(d) \to \widetilde{\mathscr{G}}(d) \to \mathscr{H}(d) \to 0$ of Néron lft-models over $\Og_{K(d)}$ is exact. Then $J_d(\widetilde{G}) = J_d(T) \uplus J_d(H).$ The same holds for the multisets of jumps if there exists a grid sequence $(d_{\ell})_\ell$ such that the exactness holds for all $d_\ell.$
\end{lemma}
\begin{proof}
This (essentially) already appears in \cite{OvI}; we repeat the argument for the reader's convenience. Let $\mathfrak{m}_{T(d)},$ $\mathfrak{m}_{\widetilde{G}(d)},$ and $\mathfrak{m}_{H(d)}$ be the maximal ideals of the local rings of $\mathscr{T}(d)_k,$ $\widetilde{\mathscr{G}}(d)_k,$ and $\mathscr{H}(d)_k$ at the origin, respectively. By the exactness of the sequence of Néron lft-models, the sequence
$$0 \to \mathfrak{m}_{H(d)}/\mathfrak{m}_{H(d)}^2 \to \mathfrak{m}_{\widetilde{G}(d)}/\mathfrak{m}_{\widetilde{G}(d)}^2 \to \mathfrak{m}_{T(d)}/\mathfrak{m}_{T(d)}^2 \to 0$$ of $k[\mu_d]$-modules is exact. Because the category $k[\mu_d]-\mathrm{mod}$ is semisimple, the multiset of weights of the middle representation is the sum of the multisets of weights of the first and last non-trivial ones; this proves the first claim.

Let $j^T_{1,d} \leq...\leq j^T_{m,d}$ be the $d$-jumps of $T,$ $j^H_{1,d} \leq...\leq j^H_{h,d}$ those of $H$ (with $m:=\dim T$ and $h:=\dim H$), and $j^{\widetilde{G}}_{1,d} \leq...\leq j^{\widetilde{G}}_{m+h,d}$ those of $\widetilde{G}.$ Now choose a grid sequence $(d_{\ell})_{\ell\in \N}.$ By the pigeonhole principle, there exist strictly increasing functions $\alpha\colon \{1,...,m\}\to \{1,...,m+h\}$ and $\beta\colon \{1,..,h\}\to \{1,...,m+h\}$ with disjoint images such that, for all $\ell,$ 
$$j^{\widetilde{G}}_{\alpha(i),d_\ell} = j^T_{i,d_\ell} \text{\hspace{.1 in} for $i=1,...,m$ \hspace{.1 in} and \hspace{.1 in}} j^{\widetilde{G}}_{\beta(i),d_\ell} = j^H_{i,d_\ell} \text{\hspace{.1 in} for $i=1,...,h$}$$ after passing to a subsequence if necessary. The claim now follows from Proposition \ref{jumpslimitprop}.
\end{proof}

The following result about the jumps is known in the tamely ramified case (see, for example, \cite[Proposition 4.13]{HNII}), but does not seem to have been observed in general:
\begin{proposition}
Let $G$ be a smooth connected wound algebraic group over $K$ and let $j_1 \leq ... \leq j_n\in \mathbf{R}/\Z$ be the jumps of $G.$ Then, for any $d$ prime to $p,$ the jumps of $G\times_KK(d)$ are the elements $dj_1,..., dj_n\in \mathbf{R}/\Z$ (which may no longer be in ascending order). In particular, the jumps of $G$ are rational if and only if so are the jumps of $G(d):=G\times_KK(d).$ 
\end{proposition}
\begin{proof}
Fix $d$ as in the proposition and let $(d_{\ell})_{\ell\in \N}$ be a grid sequence. Now consider the commutative diagram 
$$\begin{tikzcd}
\Z/dd_{\ell}\Z \arrow[swap]{d}{\mathrm{pr}_{\ell}} \arrow[r] & \mathbf{R}/\Z \arrow{d}{\cdot d} \\
\Z/d_{\ell}\Z \arrow[r] & \mathbf{R}/\Z.
\end{tikzcd}$$ 
The image of an element $x$ of $\Z/dd_{\ell}\Z$ (resp. $\Z/d_{\ell}\Z$) in $\mathbf{R}/\Z$ will be denoted by $x/(dd_\ell)$ (resp. $x/d_{\ell}$) by a slight abuse of notation. Now let $j_1\leq ... \leq j_n\in \Z/dd_\ell\Z$ be the $dd_{\ell}$-jumps of $G.$ Then, by Proposition \ref{jumpspropV}(i),  $\pr_{\ell}(j_{\tau_{\ell}(1)}) \leq ... \leq \pr_{\ell}(j_{\tau_{\ell}(n)})$ are the $d_{\ell}$-jumps of $G(d)$ for some suitable $\tau_{\ell}\in S_n.$ By the pigeonhole principle, we may assume that the sequence $(\tau_{\ell})_\ell$ is constant with stable value $\tau.$ Then $\mathrm{pr}_{\ell}(j_{\tau(i)})/d_\ell = d\cdot j_{\tau(i)}/(dd_\ell)$ by the commutativity of the diagram above. Taking the limit as $\ell\to \infty,$ the claim follows using Proposition \ref{jumpslimitprop}.
\end{proof}

\subsection{Outline of the proof}
We shall now briefly describe the main ideas behind the proof in the case where $G$ is an algebraic torus; the result for general unirational $G$ follows by a very similar argument, and that for Abelian varieties with potentially totally multiplicative reduction will be deduced via well-known arguments using rigid uniformisation. This outline will take up more space than might be usual because the underlying ideas are rather simple, but are somewhat buried underneath a heap of technical details. 
\begin{enumerate}
\item Let $T$ be an algebraic torus of dimension $n$ over $K$ split by the finite Galois extension $K\subseteq L.$ It is well-known that there exists a smooth surjective homomorphism $\Res_{L/K}\mathbf{G}_{\mathrm{m}}^{n} \to T$ with connected kernel. For each $d$ prime to $p,$ we shall let $\mathscr{R}(d)$ and $\mathscr{T}(d)$ be the Néron lft-models of $\Res_{L/K}\mathbf{G}_{\mathrm{m}}^n \times_KK(d)$ and $T\times_KK(d),$ respectively. The induced morphism 
$$g(d)^0_k\colon \mathscr{R}(d)^0_k \to \mathscr{T}(d)^0_k$$ of identity components of special fibres usually fails to be smooth, but it is always faithfully flat. Let $\mathfrak{m}_{R(d)} \subseteq \widehat{\Og}_{\mathscr{R}(d)^0_k,0}$ and $\mathfrak{m}_{T(d)} \subseteq \widehat{\Og}_{\mathscr{T}(d)^0_k,0}$ be the maximal ideals in the completed local rings of $\mathscr{R}(d)^0_k$ and $\mathscr{T}(d)^0_k,$ respectively. We are ultimately interested in the weights of the $k[\mu_d]$-module $\mathfrak{m}_{T(d)}/\mathfrak{m}_{T(d)}^2.$ Since the weights of $\mathfrak{m}_{R(d)}/\mathfrak{m}_{R(d)}^2$ are explicitly known, we shall seek to compare the two multisets of weights. Here we shall use that the homomorphism $\widehat{\Og}_{\mathscr{T}(d)^0_k,0} \to \widehat{\Og}_{\mathscr{R}(d)^0_k,0}$ is injective (it is necessary to consider the local rings themselves because the induced map on cotangent spaces usually fails to be injective due to the lack of smoothness).
\item Put $m:=n[L:K]=\dim \Res_{L/K}\mathbf{G}_{\mathrm{m}}^n.$ It is well-known that we can choose regular systems of parameters $s_1,..., s_n$ of $\widehat{\Og}_{\mathscr{T}(d)^0_k,0}$ (resp. $t_1,..., t_m$ of $\widehat{\Og}_{\mathscr{R}(d)^0_k,0}$) on which $\mu_d$ acts with the same weights as on $\mathfrak{m}_{T(d)}/\mathfrak{m}_{T(d)}^2$ (resp. $\mathfrak{m}_{R(d)}/\mathfrak{m}_{R(d)}^2$). By Cohen's structure theorem, we obtain an isomorphism $k[\![ s_1,..., s_n]\!] = \widehat{\Og}_{\mathscr{T}(d)^0_k,0}$ and similarly for $\widehat{\Og}_{\mathscr{R}(d)^0_k,0}.$ The map of local rings therefore admits a description of the form 
$$s_i \mapsto \sum_{\xi\in \N_0^m} \lambda_{i, \xi} t_1^{\xi_1}\cdot ... \cdot t_m^{\xi_m}$$ for $i=1,...,n$ and suitable $\lambda_{i, \xi}\in k.$ Suppose $\mu_d$ acts with weights $j_1,..., j_m\in \Z/d\Z$ on $t_1,...,t_m.$ For each $i=1,...,n,$ there exists some $\xi\in \N_0^m$ such that $\lambda_{i, \xi}\not=0$ by injectivity. Choosing such a $\xi$ for a given $i,$ an easy computation shows that $\mu_d$ acts with weight $\xi_1 j_1 + ... + \xi_m j_m\in \Z/d\Z$ on $s_i.$ It is precisely here that we need to consider the weights as elements of $\Z/d\Z,$ rather than $\Z.$ In fact, we shall see that there exists a very simple such description of the map of local rings under some restriction on the Frobenius and the Verschiebung of a certain infinitesimal group scheme. In general, we shall directly deduce the comparison between the weights of the action on $\widehat{\Og}_{\mathscr{T}(d)^0_k,0}$ and on $\widehat{\Og}_{\mathscr{R}(d)^0_k,0}$ by induction. The crucial new invariants which make this possible will be called \it scale \rm and \it thickness; \rm the scale (in essence) makes a prediction about the tuples $(\xi_1,..., \xi_k)$ for the various $i,$ whereas the thickness measures the \it sum \rm of the $\xi_j$ for all $i.$
\item Finally, we shall use Proposition \ref{jumpslimitprop} in order to show that the jumps of $T$ can be written as $\Z$-linear combinations of jumps of $\Res_{L/K}\mathbf{G}_{\mathrm{m}}^{n}.$ Since the latter are already known to be $[L:K]$-torsion elements in $\mathbf{R}/\Z,$ the rationality of the jumps of $T$ follows. To achieve this, we choose a sequence $(d_{\ell})_{\ell\in \N}$ as in Proposition \ref{jumpslimitprop}. Suppose that $\mu_{d_{\ell}}$ acts with weights $j_{1, \ell} \leq...\leq j_{m, \ell}$ on the regular system of parameters $t_{1, \ell},..., t_{m, \ell}$ of $\widehat{\Og}_{\mathscr{R}(d)^0_k,0}.$ Now fix an $i\in \{1,...,n\}.$ For each $\ell,$ we choose $\xi_{\ell}\in \N_0^m$ such that $\lambda_{i, \xi_\ell}\not=0.$ Then $\mu_{d_{\ell}}$ acts with weight $\xi_{1,\ell} j_{1,\ell} + ... + \xi_{m, \ell} j_{m, \ell}$ on $s_i.$ The crucial step will now be to give a bound on the thickness (effectively a bound on $\xi_{1, \ell} + ... + \xi_{m, \ell}$) which is independent of $\ell.$ Once we have such a bound, we deduce that there are only finitely many possibilities for $\xi_\ell,$ so we may assume without loss of generality that the sequence $(\xi_{\ell})_{\ell\in \N}$ is constant with stable value $\xi.$ If $j_1,..., j_m$ are the jumps of $\Res_{L/K}\mathbf{G}_{\mathrm{m}}^{n},$ then Proposition \ref{jumpslimitprop} shows that the $i$-th jump of $T$ is $\xi_1 j_1 + ... + \xi_n j_n \in \mathbf{R}/\Z.$ The bound on the thickness will be proven using intersection theory \cite{Fulton}. In fact, all we shall use is a version of the classical fact that the \it degree \rm of a projective variety (relative to an ample line bundle) is constant in a flat family. 
\end{enumerate}

\section{Weights and scales}\label{weightsscalessect}
Let $k$ be an algebraically closed field of characteristic $p>0$ and let $R$ be a regular complete local $k$-algebra of dimension $n\in \N_0$ with residue field $k.$  Let $t_1,..., t_n$ be a regular system of parameters of $R,$ i. e. a sequence of elements of the maximal ideal $\mathfrak{m}_R\subseteq R$ whose images $\overline{t}_1,..., \overline{t}_n$ modulo $\mathfrak{m}_R^2$ form a basis of the $k$-vector space $\mathfrak{m}_R/\mathfrak{m}_R^2.$ By Cohen's structure theorem, the canonical map $k[\![ t_1,..., t_n ]\!]\to R$ is then an isomorphism \cite[Tag 0C0S]{Stacks}. Let $d\in \N$ be a positive integer invertible in $k$ (and hence in $R$), and suppose that the finite cyclic group $\mu_d:=\mu_d(k)$ of $d$-th roots of unity in $k$ acts (from the left) on $R$ by automorphisms of $k$-algebras. 
The following result is well-known; we include a proof for the reader's convenience:
\begin{proposition} \label{weightsexistprop}
\begin{enumerate}[(i)]
\item The system $t_1,...,t_n$ can be chosen in such a way that $$\zeta\ast t_i = \zeta^{j_i} t_i$$ for all $i=1,...,n,$ all $\zeta\in\mu_d,$ and some $j_1,...,j_n\in \Z/d\Z.$
\item If $t_1',..., t_n'$ is another regular system of parameters such that $\zeta\ast t_i'= \zeta^{j_i'} t_i'$ for suitable $j_1',...,j_n'\in \Z/d\Z$ and all $i=1,...,n,$ then there exists a permutation $\sigma\in S_n$ such that $j_i'=j_{\sigma(i)}$ for all $i.$
\end{enumerate}
\end{proposition}
\begin{proof}
The $\mu_d$-equivariant homomorphism $\mathfrak{m}_R \to \mathfrak{m}_R/\mathfrak{m}_R^2$ of $k$-vector spaces admits a $\mu_d$-equivariant splitting $\varsigma$ since $d\in k^\times.$ The category of $k[\mu_d]$-modules is semisimple; if $\chi_d$ denotes the one-dimensional $k$-representation of $\mu_d$ on which $\zeta$ acts with weight one, the simple objects of this category are the representations $\chi_d^{\otimes j}$ for $j\in \Z/d\Z.$ In particular, we can write $\mathfrak{m}_R/\mathfrak{m}_R^2$ as $\chi_d^{\otimes j_1} \oplus ... \oplus \chi_d^{\otimes j_n}$ for suitable $j_1,..., j_n\in \Z/d\Z.$ Choosing a basis for each direct summand yields a basis of $\mathfrak{m}_R/\mathfrak{m}_R^2,$ the image of which under $\varsigma$ then has the property described in (i). Since the choice of elements of $\mathfrak{m}_R$ as in (ii) induces another decomposition of $\mathfrak{m}_R/\mathfrak{m}_R^2$ into simple representations, the uniqueness of this decomposition up to permutation of the direct summands shows (ii).
\end{proof}

\begin{definition}
We shall say that $\mu_d$ \rm acts with weights $j_1, ... , j_n$ on $R$ \it in this situation and call $t_1,..., t_n$ a regular system of \rm eigenparameters. \it
\end{definition}

By Proposition \ref{weightsexistprop} (ii), the finite sequence $j_1,..., j_n$ is an invariant of $(R, \ast)$ up to permutation. Our main task in the present section will be the following: Suppose we have a finite ring extension $R\subseteq S$ with $S$ regular, local, (necessarily complete), and residue field $k$ such that $\mu_d$ acts on $S$ and $R$ in a compatible way. Then we would like to find a way of comparing the weights with which $\mu_d$ acts on $R$ and on $S.$ We shall see that there is a satisfactory description if the map $\Spec S \to \Spec R$ is a torsor for some infinitesimal $k$-group scheme $\boldsymbol{\alpha}$ on which $\mu_d$ also acts. First recall that the \it Frobenius height \rm of $\boldsymbol{\alpha}$ is the smallest integer $m$ such that the relative Frobenius $F^m_{\boldsymbol{\alpha}/k} \colon \boldsymbol{\alpha} \to \boldsymbol{\alpha}^{(p^m)}$ vanishes; such an integer always exists. 

The following is likely well-known (see, e. g., \cite[Proposition 16.2]{Pink} for the unipotent case):

\begin{lemma}\label{ablem}
Let $\boldsymbol{\alpha}$ be an infinitesimal group scheme over an algebraically closed field $k$ of characteristic $p>0.$ Suppose that $\boldsymbol{\alpha}$ has Frobenius height $\leq 1$ and that the image of the relative Verschiebung $V_{\boldsymbol{\alpha}/k} \colon \boldsymbol{\alpha}^{(p)} \to \boldsymbol{\alpha}$ is multiplicative. Then there exist nonnegative integers $a$ and $b$ such that $\boldsymbol{\alpha} \cong \boldsymbol{\alpha}_p ^  a \oplus \boldsymbol{\mu}_p ^ b.$
\end{lemma}
\begin{proof}
Since (commutative) group schemes of Frobenius height $\leq 1$ are classified by (commutative) $p$-Lie algebras over $k$ \cite[Chapitre II, §7, no. 4, Proposition 4.1]{DG}, the claim is equivalent to the following statement: Let $\mathfrak{g}:=\Lie\boldsymbol{\alpha}$ be the finite-dimensional commutative $p$-Lie algebra over $k$ corresponding to $\boldsymbol{\alpha}.$ Then $\mathfrak{g} = \mathfrak{g}_{\mathrm{u}} \oplus \mathfrak{g}_{\mathrm{m}},$ where $\mathfrak{g}_{\mathrm{u}}, \mathfrak{g}_{\mathrm{m}} \subseteq \mathfrak{g}$ are $p$-Lie subalgebras such that the $p$-th power map is trivial on $\mathfrak{g}_{\mathrm{u}}$ and such that $\mathfrak{g}_{\mathrm{m}}$ contains a basis $x_1,..., x_b$ with $x_i^{[p]} = x_i$ for $i=1,...,b.$ It is well-known that there is a decomposition as above with $\mathfrak{g}_{\mathrm{u}}$ nilpotent and $ \mathfrak{g}_{\mathrm{m}}$ semisimple \cite[p. 477]{Jac}; the description of $ \mathfrak{g}_{\mathrm{m}}$ is then \cite[Theorem 3]{Jac}. Now we observe that $\boldsymbol{\alpha}^{(p)}$ corresponds to the Frobenius twist $\mathfrak{g}^{(p)};$ moreover, the $p$-th power map can be seen as a \it $k$-homomorphism \rm $\mathfrak{g}^{(p)} \to \mathfrak{g}$ which corresponds to the Verschiebung. Our assumption therefore implies that the $p$-th power map of $\mathfrak{g}_{\mathrm{u}}$ is indeed trivial, so we may put $a:=\dim_k \mathfrak{g}_{\mathrm{u}}$ and deduce the claim. 
\end{proof}

In order to treat torsors under more general infinitesimal group schemes, we shall need the following 

\begin{lemma} \label{filtrationlem}
Let $\boldsymbol{\alpha}$ be an infinitesimal group scheme over $k.$ Then there exists an $\Aut_k(\boldsymbol{\alpha})$-equivariant filtration 
$$0 = \boldsymbol{\alpha}_{(0)} \subseteq ... \subseteq \boldsymbol{\alpha}_{(h)} = \boldsymbol{\alpha}$$
such that each successive quotient has Frobenius height $\leq 1$ and multiplicative Verschiebung image. 
\end{lemma}
\begin{proof}
By induction on the order of $\boldsymbol{\alpha},$ it suffices to exhibit a canonical characteristic subgroup scheme of $\boldsymbol{\alpha}$ with the claimed properties which is of positive order as soon as so is $\boldsymbol{\alpha}.$ Denote by $F_{\boldsymbol{\alpha}/k} \colon \boldsymbol{\alpha} \to \boldsymbol{\alpha}^{(p)}$ the relative Frobenius and by $V^l_{\boldsymbol{\alpha}/k} \colon \boldsymbol{\alpha}^{(p^l)} \to \boldsymbol{\alpha}$ the $l$-fold iteration of relative Verschiebung on $\boldsymbol{\alpha}.$ Note that both these morphisms are functorial in $\boldsymbol{\alpha}.$ Replacing $\boldsymbol{\alpha}$ by $\ker F_{\boldsymbol{\alpha}/k},$ we may hence assume that the Frobenius height of $\boldsymbol{\alpha}$ is $\leq 1.$ Now let $l$ be the minimal positive integer such that the scheme-theoretic image of $V^l_{\boldsymbol{\alpha}/k}$ is multiplicative. Replacing $\boldsymbol{\alpha}$ by the scheme-theoretic image of $V^{l-1}_{\boldsymbol{\alpha}/k},$ we obtain a closed subgroup scheme of $\boldsymbol{\alpha}$ with trivial Frobenius and multiplicative Verschiebung image which is of positive order if so is $\boldsymbol{\alpha}.$ Equivariance follows from the functoriality of Frobenius and Verschiebung. 
\end{proof}

For the remainder of this paragraph, we shall consider the following situation: let $R\subseteq S$ be a finite extension of regular complete local $k$-algebras of dimension $n$ which is a left-torsor for some infinitesimal $k$-group scheme $\boldsymbol{\alpha}.$ Assume that the group $\mu_d$ acts (from the left) on both $R$ and $S$ by automorphisms of $k$-algebras in a compatible way. Finally, suppose that $\mu_d$ acts (from the right) on $\boldsymbol{\alpha}$ by automorphisms of $k$-group schemes and that for all local sections $s$ of $\Spec S$ and $\alpha$ of $\boldsymbol{\alpha},$ as well as for all $\zeta\in \mu_d,$ we have 
\begin{align}\alpha(s) \ast \zeta = (\alpha \ast \zeta)(s \ast \zeta).\label{actioncondition}\end{align}

\begin{proposition} \label{Height1weightsprop}
Assume that $\boldsymbol{\alpha}$ is of Frobenius height $\leq 1,$ and that the image of Verschiebung on $\boldsymbol{\alpha}$ is multiplicative. Let $r\in \N$ such that $p^r=\dim_k\Gamma(\boldsymbol{\alpha}, \Og_{\boldsymbol{\alpha}}).$ Then there exists a regular system of eigenparameters $t_1,..., t_n$ of $S$ such that $t_1^p,..., t_r^p, t_{r+1},....,t_n$ are contained in $R$ and form a regular system of parameters of $R.$ 
\end{proposition}

\begin{proof}
By our assumption on the Frobenius height and the Verschiebung, Lemma \ref{ablem} tells us that there exist nonnegative integers $a$ and $b$ and an isomorphism 
$$\boldsymbol{\alpha} \cong \boldsymbol{\alpha}_p ^  a \oplus \boldsymbol{\mu}_p ^ b.$$ Put $r:=a+b.$ Since $R$ is local, we have $H^1(\Spec R, \Ga)=0=H^1(\Spec R, \Gm).$ By Kummer theory (see \cite[Chapter III, §4]{Milne}) and its additive analogue, we obtain an isomorphism 
$$S\cong R[X_1,..., X_r] / \langle X_1^p -f_1,..., X_r^p-f_r \rangle$$ for suitable $f_1,..., f_r\in R,$ which we may choose to lie in $\mathfrak{m}_R$ because $k$ is algebraically closed. In particular, $s^p\in R$ for all $s\in S.$ We shall denote the images of the $X_1,..., X_r$ in $S$ by the same symbols. Note that the canonical map 
$$R[X_1,..., X_r] / \langle X_1^p -f_1,..., X_r^p-f_r \rangle \to R[\![X_1,..., X_r]\!] / \langle X_1^p -f_1,..., X_r^p-f_r \rangle$$ is an isomorphism, and that we have an exact sequence
\begin{align}0 \to\langle \overline{f}_1,..., \overline{f}_r \rangle \to \mathfrak{m}_R/\mathfrak{m}_R^2 \to \mathfrak{m}_S/\mathfrak{m}_S^2 \to \langle \overline{X}_1,..., \overline{X}_r\rangle \to 0\label{exseqI}\end{align} of $k[\mu_d]$-modules; here bars denote images modulo the square of the maximal ideal. Because $\dim S = \dim R$ and because $S$ is regular, the elements $f_1,..., f_r$ of $\mathfrak{m}_R$ form part of a regular system of parameters of $R.$ Now choose eigenparameters $t_1,..., t_r\in \mathfrak{m}_S$ whose images in $\mathfrak{m}_S/\mathfrak{m}_S^2$ map to a basis of $\langle \overline{X}_1,..., \overline{X}_r\rangle.$ Then we can find a matrix $(a_{ij})_{ij} \in \mathrm{GL}_r(S)$ as well as $g_j\in \mathfrak{m}_R$, $h_j\in \mathfrak{m}_S^2$ such that 
$$t_j = \sum_{i=1}^r a_{ij} X_i + g_j + h_j$$
for $j=1,...,r.$ This implies that 
$t_j^p = \sum_{i=1}^r a_{ij}^p f_i + g_j^p + h_j^p$ for all $j.$ Because we clearly have $g_j^p, h_j^p \in \mathfrak{m}_R^2$ and $(a_{ij}^p)_{ij} \in \mathrm{GL}_r(R),$ we find that $t_1^p,..., t_r^p$ form part of a regular system of parameters of $R.$ Now we extend $t_1,...,t_r$ to a regular system of eigenparameters $t_1,..., t_n$ of $S$ such that $t_{r+1},..., t_n\in R.$ The exactness of the sequence (\ref{exseqI}) then shows that $t_1^p,...,t_r^p,t_{r+1},...,t_n$ is a regular system of parameters of $R.$
\end{proof}

\begin{proposition}\label{weightspermprop}
Let $\Spec S \to \Spec R$ be a $\mu_d$-equivariant torsor for an infinitesimal group scheme $\boldsymbol{\alpha}$ as described at the beginning of this paragraph. Let $n:=\dim S.$ Suppose that $\mu_d$ acts with weights $j_1,..., j_n\in \Z/d\Z$ on $S.$ Then there exists a finite sequence $e_1,..., e_n$ of nonnegative integers with the following properties: 
\begin{enumerate}[(i)]
\item $\mu_d$ acts with weights $p^{e_1}j_1,..., p^{e_n}j_n$ on $R,$ and 
\item $\dim_k \Gamma(\boldsymbol{\alpha}, \Og_{\boldsymbol{\alpha}})=p^{e_1 + ... + e_n}.$
\end{enumerate}
\end{proposition}
\begin{proof}
We prove the claim by induction on the length $h$ of the filtration from Lemma \ref{filtrationlem}. If $h=1,$ we choose a regular system of eigenparameters $t_1,..., t_n$ of $S$ and $1\leq r \leq n$ such that $p^r=\dim_k\Gamma(\boldsymbol{\alpha}, \Og_{\boldsymbol{\alpha}})$ and such that $t_1^p,..., t_r^p, t_{r+1},..., t_n$ are contained in $R$ and form a regular system of parameters (as in Proposition \ref{Height1weightsprop}). Up to reordering the weights, we may suppose that $\zeta\ast t_i = \zeta^{j_i} t_i$ for $i=1,...,n$ and compute 
$$\zeta\ast t_i^p = (\zeta\ast t_i)^p = \zeta^{p j_i} t_i^p$$ for $i=1,...,r,$ so the claim follows in this case since $\dim_k \Gamma(\boldsymbol{\alpha}, \Og_{\boldsymbol{\alpha}})=p^{r}.$ In general, let $\boldsymbol{\beta}$ be the first non-zero part of the filtration from Lemma \ref{filtrationlem}. Note that the quotient of $\Spec S$ by the action of $\boldsymbol{\beta}$ exists as an affine scheme by a theorem of Grothendieck \cite[Théorème 1 iii) on p. 82]{RayI}; we shall denote its ring of global functions by $S^{\boldsymbol{\beta}}.$ Then $S^{\boldsymbol{\beta}}$ is clearly finite over $R$ and local (hence complete); it is moreover regular because the map $S^{\boldsymbol{\beta}}\subseteq S$ is faithfully flat \cite[Tag 00OF]{Stacks}. Finally, we observe that, by the compatibility (\ref{actioncondition}) of the $\mu_d$-actions on $S$ and $\boldsymbol{\alpha},$ $\mu_d$ acts on $S^{\boldsymbol{\beta}}.$ The morphism $\Spec S^{\boldsymbol{\beta}} \to \Spec R$ is then a $\mu_d$-equivariant torsor for $\boldsymbol{\alpha}/\boldsymbol{\beta},$ whose filtration is of length $h-1.$ Now let $t_1,..., t_n$ be a system of eigenparameters of $S$ such that $t_1^p,..., t_r^p, t_{r+1},..., t_n$ are contained in $S^{\boldsymbol{\beta}}$ and form a regular system of parameters (Proposition \ref{Height1weightsprop}) with $p^r:=\dim_k\Gamma(\boldsymbol{\beta}, \Og_{\boldsymbol{\beta}}),$ and let $j_1,..., j_n$ be the corresponding weights. We already know that $pj_1,..., pj_r, j_{r+1},..., j_n$ are the weights of the $\mu_d$-action on $S^{\boldsymbol{\beta}},$ so by the induction hypothesis, there exists a finite sequence $e_1',..., e_n'$ such that $p^{e'_1+1}j_1,..., p^{e'_r+1}j_r, p^{e'_{r+1}} j_{r+1},..., p^{e'_n}j_n$ are the weights of the $\mu_d$-action on $R$ and such that $\dim_k \Gamma(\boldsymbol{\alpha}/\boldsymbol{\beta}, \Og_{\boldsymbol{\alpha}/\boldsymbol{\beta}}) = p^{e_1'+...+e_n'}.$ Hence we may put $e_i:=e_i'+1$ for $1\leq i \leq r$ and $e_i:=e_i'$ for $r+1\leq i \leq n;$ since $\dim_k \Gamma(\boldsymbol{\beta}, \Og_{\boldsymbol{\beta}}) = p^r,$ we moreover obtain $\dim_k\Gamma(\boldsymbol{\alpha}, \Og_{\boldsymbol{\alpha}})= p^{e_1+...+e_n},$ as claimed.
\end{proof}

\begin{definition}\label{scaledef}
In the situation of Proposition \ref{weightspermprop}, we shall call any sequence $e_1,..., e_n$ of nonnegative integers with the two properties described there a \rm scale \it of the $\mu_d$-equivariant $\boldsymbol{\alpha}$-torsor $\Spec S \to \Spec R$ relative to the weights $j_1,..., j_n.$  
\end{definition}

\begin{remark} This terminology has been chosen since the scale will mainly be used as a device to measure the weights of the $\mu_d$-action on $R,$ provided we already know those of the action on $S.$ Note that the scale is not uniquely determined by the torsor (even after choosing an ordering on the multiset of weights); in general, the larger $d$ is compared to $\dim_k \Gamma(\boldsymbol{\alpha}, \Og_{\boldsymbol{\alpha}}),$ the more information is contained in the scale(s). 
\end{remark}
\section{Multiplicities}
Let $k$ be an algebraically closed field and let $\psi\colon G\to H$ be a surjective morphism of smooth connected algebraic groups over $k.$ We shall need a measure of the failure of $\psi$ to be smooth. To this end, let $\Delta_\psi:=(\ker \psi)^0.$ Then $\Delta_\psi$ is a connected algebraic group over $k,$ and because $k$ is perfect, the maximal reduced closed subscheme $\Delta_{\psi, \mathrm{red}}$ of $\Delta_\psi$ is a subgroup scheme of $\Delta_\psi$ which is, moreover, smooth over $k.$ The quotient $\boldsymbol{\alpha}_\psi:=\Delta_\psi/\Delta_{\psi, \mathrm{red}}$ is then an infinitesimal group scheme over $k.$

\begin{definition}
The \rm thickness \it of $\psi$ is 
$$\Theta(\psi):=\dim_k \Gamma(\boldsymbol{\alpha}_{\psi}, \Og_{\boldsymbol{\alpha}_{\psi}}).$$
\end{definition}
Clearly, $\psi$ is smooth if and only if $\Theta(\psi)=1.$ The main reason why this invariant is useful for us is the fact that $\Theta(\psi)$ can be interpreted as the \it multiplicity \rm of the cycle $[\Delta] \in \mathrm{CH}_{\dim \Delta}(\Delta) \cong \Z$: 

\begin{proposition} \label{thicknessmultprop}
Let $\eta$ be the generic point of $\Delta_\psi.$ Then 
$$\ell_{\Og_{\Delta_\psi,\eta}}(\Og_{\Delta_\psi,\eta}) = \Theta(\psi).$$
\end{proposition}
\begin{proof}
To simplify the notation, we shall write $\Delta:=\Delta_\psi$ and $\boldsymbol{\alpha}:=\boldsymbol{\alpha}_{\psi}.$ The map $\Delta \to \boldsymbol{\alpha}$ is faithfully flat. Since $\boldsymbol{\alpha}$ is infinitesimal, we obtain a faithfully flat homomorphism $\Gamma(\boldsymbol{\alpha}, \Og_{\boldsymbol{\alpha}}) \to \Og_{\Delta, \eta}.$ Choose a maximal strictly increasing chain 
$$0=I_0 \subseteq I_1 \subseteq ... \subseteq I_m = \Gamma(\boldsymbol{\alpha}, \Og_{\boldsymbol{\alpha}})$$ of ideals in $\Gamma(\boldsymbol{\alpha}, \Og_{\boldsymbol{\alpha}}).$ The successive quotients are then all isomorphic to $k;$ in particular, we have $\Theta(\psi)=m.$ By construction, we have 
$$\Og_{\Delta, \eta} \otimes_{\Gamma(\boldsymbol{\alpha}, \Og_{\boldsymbol{\alpha}})} k = \Og_{\Delta_{\mathrm{red}}, \eta},$$ which is a field. By faithful flatness, we obtain a strictly increasing chain 
$$0=\Og_{\Delta, \eta} \otimes_{\Gamma(\boldsymbol{\alpha}, \Og_{\boldsymbol{\alpha}})} I_0 \subseteq \Og_{\Delta, \eta} \otimes_{\Gamma(\boldsymbol{\alpha}, \Og_{\boldsymbol{\alpha}})} I_1 \subseteq ... \subseteq \Og_{\Delta, \eta} \otimes_{\Gamma(\boldsymbol{\alpha}, \Og_{\boldsymbol{\alpha}})} I_m  = \Og_{\Delta, \eta}$$ with successive quotients isomorphic to $ \Og_{\Delta_{\mathrm{red}}, \eta}.$ This implies that $$\ell_{\Og_{\Delta,\eta}}(\Og_{\Delta,\eta}) = m = \Theta(\psi),$$ as claimed. 
\end{proof}

We shall mainly employ the thickness in later paragraphs by bounding it; one useful result in this direction is 
\begin{lemma} \label{Thetaboundlem}
Let $M,$ $G,$ and $H$ be smooth connected algebraic groups over $k$ and let $\gamma\colon M\to G$ and $\phi\colon G\to H$ be faithfully flat homomorphisms of algebraic groups. Then $$\Theta(\phi) \leq \Theta(\phi\circ \gamma).$$
\end{lemma}
\begin{proof}
The canonical map $\ker(\phi\circ \gamma) \to \ker \phi$ is surjective in the fppf-topology; hence so is the map $ \Delta_{\phi\circ \gamma} \to  \Delta_{\phi}.$ Now consider the induced commutative diagram
$$\begin{tikzcd}
0 \arrow[r] &  \Delta_{\phi\circ \gamma, \mathrm{red}} \arrow[r] \arrow[d] & \Delta_{\phi\circ \gamma} \arrow[r]\arrow[d] & \boldsymbol{\alpha}_{\phi\circ\gamma} \arrow[r]\arrow[d] & 0 \\ 
0 \arrow[r] &  \Delta_{\phi, \mathrm{red}} \arrow[r] &  \Delta_{\phi} \arrow[r] & \boldsymbol{\alpha}_{\phi} \arrow[r] & 0.
\end{tikzcd}$$
The snake lemma now shows that the map $\boldsymbol{\alpha}_{\phi\circ\gamma} \to \boldsymbol{\alpha}_{\phi}$ is surjective in the fppf-topology, which implies that the morphism $\Gamma(\boldsymbol{\alpha}_{\phi}, \Og_{\boldsymbol{\alpha}_{\phi}}) \to \Gamma(\boldsymbol{\alpha}_{\gamma\circ\phi}, \Og_{\boldsymbol{\alpha}_{\gamma\circ\phi}})$ of $k$-vector spaces is injective.
\end{proof}

\section{Bounding the thickness} \label{boundingthicknesspar}
In this paragraph, we shall let $\Og_K$ be a complete discrete valuation ring with field of fractions $K$ and algebraically closed residue field $k.$ Let $T$ be an algebraic torus over $K,$ let $K\subseteq L$ be a finite étale $K$-algebra, and let
$$\xi_{\eta} \colon \Res_{L/K}\mathbf{G}_{\mathrm{m}}^m \to T$$ be a smooth surjective morphism of algebraic groups with connected kernel for some positive integer $m.$ Denote by $\Og_L$ the integral closure of $\Og_K$ in $L.$  Let $\mathscr{R}$ and $\mathscr{T}$ be the the Néron lft-models of $\Res_{L/K}\mathbf{G}_{\mathrm{m}}^m$ and $T,$ respectively.  We begin by observing that $\mathscr{R}^0$ admits a canonical smooth compactification: we have a canonical isomorphism 
$$\mathscr{R}^0 = \Res_{\Og_L/\Og_K}\mathbf{G}_{\mathrm{m}}^m$$ by \cite[Proposition A.5.9]{CGP}. Since Weil restriction commutes with open immersions \cite[Chapter 7.6, Proposition 2 (i)]{BLR}, we have open immersions
$$\mathscr{R}^0 \to \Res_{\Og_L/\Og_K} \mathbf{A}^m_{\Og_L} = \Spec \mathrm{Sym}\, \Og_L^{m, \vee} \to \mathbf{P}(\Og_L^{m, \vee} \oplus \Og_K) =: \overline{\mathscr{R}}.$$ 
We let $\omega$ denote the composition of these open immersions and let $\mathscr{X}$ be the scheme-theoretic closure of $\ker \xi_{\eta}$ in $\overline{\mathscr{R}}\cong\mathbf{P}^{m\dim_KL}_{\Og_K};$ in particular, $\mathscr{X} \to \Spec \Og_K$ is a flat projective morphism. Note that $\overline{\mathscr{R}}$ carries a canonical $\Og_K$-very ample line bundle $\Og_{\overline{\mathscr{R}}}(1).$ We shall denote by $\overline{\mathscr{R}}_\eta$ and $\overline{\mathscr{R}}_k$ the generic and the special fibre of $\overline{\mathscr{R}},$ respectively. 

Before we proceed, we recall the following (well-known)

\begin{lemma} \label{flatlem}
Let $T_1, T_2$ be algebraic tori over $K$ and let $T_1 \to T_2$ be a smooth surjective homomorphism of algebraic groups with connected kernel. Then the induced morphism $\mathscr{T}_1 \to \mathscr{T}_2$ between the Néron lft-models of $T_1$ and $T_2$ is faithfully flat. In particular, the same is true for the induced morphism $\mathscr{T}_1^0 \to \mathscr{T}_2^0$ of identity components. 
\end{lemma}
\begin{proof}
Let $T_0$ be the kernel of $T_1\to T_2,$ which is also an algebraic torus over $K.$ It is sufficient to show that the induced morphism $\mathscr{T}_{1,k} \to \mathscr{T}_{2,k}$ is surjective. Indeed, a surjective homomorphism between smooth group schemes over a field is automatically (faithfully) flat, so the first claim then follows from the fibre-wise criterion of flatness \cite[Tag 039E]{Stacks}. The map $\mathscr{T}_1(\Og_K) = T_1(K) \to T_2(K) = \mathscr{T}_2(\Og_K)$ is surjective because $H^1(K, T_0)=0$ \cite[Lemma 4.3]{Chai}; in particular, so is the map $\mathscr{T}_{1,k}(k) \to \mathscr{T}_{2,k}(k)$ because $\Og_K$ is Henselian. Because flatness is local on the source and the target in the Zariski topology, it also follows that the map $\mathscr{T}^0_{1,k} \to \mathscr{T}^0_{2,k}$ is (faithfully) flat. 
\end{proof} 

\begin{proposition} \label{openprop}
Let $$\xi^0\colon \mathscr{R}^0 \to \mathscr{T}^0$$ be the morphism between the identity components of the Néron lft-models of $\Res_{L/K} \mathbf{G}_{\mathrm{m}} ^ m$ and $T$ induced by $\xi_{\eta}.$ Then there is a canonical isomorphism 
$$\ker \xi^0 = \omega^{-1}(\mathscr{X}).$$
\end{proposition}
\begin{proof}
The scheme $\ker \xi^0$ has generic fibre $\ker \xi_{\eta},$ and is closed in $\mathscr{R}^0$ as well as flat over $\Og_K$ by Lemma \ref{flatlem}. In particular, $\ker \xi^0$ is the scheme-theoretic closure of $\ker \xi_\eta$ in $\mathscr{R}^0.$ Hence the claim follows because scheme-theoretic images of quasi-compact morphisms commute with base change along open immersions \cite[Tag 01R8 (3)]{Stacks}. 
\end{proof}

Let $\xi^0_k \colon \mathscr{R}^0_k \to \mathscr{T}^0_k$ be the morphism on special fibres induced by $\xi_{\eta}.$ We shall now come to the crucial result of this paragraph, which is a bound on $\Theta(\xi^0_k)$ purely in terms of an invariant of $\mathscr{X}_{\eta}:=\mathscr{X}\times_{\Og_K} K \subseteq \overline{\mathscr{R}}_\eta$ stable under finite extensions of $K.$ For this purpose, let $\rho:=\dim \ker \xi_{\eta} = \dim \mathscr{X}_{\eta}$ and consider the \it specialisation homomorphism \rm
$$\CH_{\rho}(\overline{\mathscr{R}}_\eta) \to \CH_{\rho}(\overline{\mathscr{R}}_k)$$ \cite[p. 399]{Fulton}, which maps the class of a purely $\rho$-dimensional closed subscheme of $\overline{\mathscr{R}}_\eta$ to that of the special fibre of its scheme-theoretic closure in $\overline{\mathscr{R}}$ \cite[Section 2.6, Proposition 2.6 (d)]{Fulton}\footnote{Note that the Chow group $\CH_s(X)$ of $s$-cycles on a projective variety $X$ over a field $\kappa$ up to rational equivalence is denoted by $A_s(X)$ in \it op. cit. \rm}. Finally, recall that the diagram
\begin{equation}
\begin{tikzcd} \label{diag}
\CH_{\rho}(\overline{\mathscr{R}}_\eta) \arrow[swap]{d}{c_1(\Og_{\overline{\mathscr{R}}_\eta}(1))^{\rho}\cap -}\arrow[r] & \CH_{\rho}(\overline{\mathscr{R}}_k)  \arrow{d}{c_1(\Og_{\overline{\mathscr{R}}_k}(1))^{\rho}\cap -} \\
\CH_{0}(\overline{\mathscr{R}}_\eta) \arrow[r] & \CH_{0}(\overline{\mathscr{R}}_k),
\end{tikzcd}
\end{equation}
is commutative \cite[Chapter 20.3, Example 20.3.3]{Fulton}, where the horizontal arrows are the specialisation homomorphisms and the vertical ones the iterated Chern class operations \cite[Chapter 2.5, Proposition 2.5 (a)]{Fulton}. For a 0-cycle $\tau$ on a projective variety $X$ over a field $\kappa,$ we shall denote by $\int_X \tau$ the degree of $\tau$ \cite[Chapter 1.4, Definition 1.4]{Fulton}. Moreover, for a closed subscheme $Y\subseteq X$ of pure dimension $s,$ we shall write $[Y]$ for the associated cycle class in $\CH_s(X)$ (cf. \cite[p. 15]{Fulton}).

\begin{proposition} \label{Thetaboundprop}
We have 
$$\Theta(\xi^0_k) \leq \int_{\overline{\mathscr{R}}_\eta} c_1(\Og_{\overline{\mathscr{R}}_\eta}(1))^{\rho} \cap [\mathscr{X}_\eta].$$
\end{proposition}

\begin{proof}
Let $Y_1,..., Y_q$ be the reduced irreducible components of $\mathscr{X}_k;$ after changing the ordering if necessary, we may suppose that $(\ker \xi^0_k)^0$ is open in $Y_1.$ Now write
$$[\mathscr{X}_k] = \sum_{i=1}^q m_i[Y_i]$$ in $\CH_{\rho}(\overline{\mathscr{R}}_k)$ for suitable positive integers $m_1,..., m_q.$ By Proposition \ref{thicknessmultprop} and because $m_j = \ell_{\Og_{Y_j, \eta_j}}(\Og_{Y_j, \eta_j})$ \cite[p. 15]{Fulton}, we have $\Theta(\xi_k^0) = m_1.$ Moreover, for all $j=1,...,q,$ we have 
$$\int_{\overline{\mathscr{R}}_k} c_1(\Og_{\overline{\mathscr{R}}_k}(1))^{\rho} \cap [Y_j] > 0 $$ by \cite[Chapter 12.1, Lemma 12.1]{Fulton}. Hence we obtain
\begin{align*}
m_1 &\leq \sum_{j=1}^q m_j \int_{\overline{\mathscr{R}}_k} c_1(\Og_{\overline{\mathscr{R}}_k}(1))^{\rho} \cap [Y_j] \\
& = \int_{\overline{\mathscr{R}}_k} c_1(\Og_{\overline{\mathscr{R}}_k}(1))^{\rho} \cap [\mathscr{X}_k] \\
& = \int_{\overline{\mathscr{R}}_\eta} c_1(\Og_{\overline{\mathscr{R}}_\eta}(1))^{\rho} \cap [\mathscr{X}_\eta],
\end{align*}
where the final equality comes from the commutativity of the diagram (\ref{diag}) together with the fact that the degree map commutes with the specialisation homomorphism (this follows from \cite[Chapter 20.3, Proposition 20.3 (a)]{Fulton} since the degree map is proper pushforward to the ground field).
\end{proof}

\section{Rationality of jumps}
\subsection{Algebraic tori} \label{torisubsec}
Let $\Og_K$ be a complete discrete valuation ring with field of fractions $K$ and algebraically closed residue field $k$ of characteristic $p>0.$ For a natural number $d$ prime to $p,$ there is  an extension $K\subseteq K(d)$ of degree $d,$ which is unique up to (non-unique) isomorphism. We shall identify $\Gal(K(d)/K)$ with $\mu_d=\mu_d(k).$ For a finite étale $K$-algebra $L,$ we shall denote by $\Og_L$ the integral closure of $\Og_K$ in $L;$ then $\Og_L$ is a finite product of complete discrete valuation rings with residue field $k$ and the extension $\Og_K\subseteq \Og_L$ is finite and free. 

Now let $g_{\eta} \colon T_1 \to T_2$ be a smooth surjective morphism of algebraic tori with connected kernel. The goal of the present paragraph is to establish a connection between the jumps of $T_1$ and those of $T_2$ without making any assumptions on the induced morphism $\mathscr{T}_1 \to \mathscr{T}_2$ of Néron lft-models. Recall that there exists a finite Galois extension $K\subseteq L$ such that $T_1$ is split by $L,$ i. e. such that there exists an isomorphism $T_1 \times_K L \cong \mathbf{G}_{\mathrm{m}}^m$ with $m:=\dim T_1.$ We shall fix such an $L$ and such an isomorphism once and for all. Moreover, we recall the following (well-known)

\begin{lemma}\label{Reslem}
Suppose the Galois extension $K\subseteq L$ splits $T_1.$ Then there exists a smooth surjective morphism $\Res_{L/K} \mathbf{G}_{\mathrm{m}}^m \to T_1$ with connected kernel.
\end{lemma}
\begin{proof}
Since the morphism $f\colon \Spec L \to \Spec K$ is finite étale, we have a canonical isomorphism $f_\ast = f_!$ of functors from Abelian sheaves on the big étale site of $\Spec L$ to Abelian sheaves on that of $\Spec K$ (this follows formally from the corresponding result for small étale sites \cite[Tag 03S7]{Stacks}). We shall show that the canonical map
$\Res_{L/K} T_{1,L} = f_! f^\ast T_1 \to T_1$ has the desired properties for \it any \rm finite étale $K$-algebra $L.$ Smoothness, surjectivity, and connectedness can be checked after a finite extension of $K.$ Therefore, by \cite[Tag 03S6]{Stacks}, we may assume that $L$ is the product of $[L:K]$ copies of $K.$ But then the canonical map $T_1^{[L:K]} = f_! f^\ast T_1 \to T_1$ is simply the summation map \cite[Tag 03SH (2)]{Stacks}, which is clearly smooth, surjective, and has connected kernel.
\end{proof}

For each $d$ prime to $p,$ we let $T_i(d):=T_i\times_K K(d)$ and denote by $\mathscr{T}_i(d)$ the Néron lft-model of $T_i(d)$ over $\Og_{K(d)}$ for $i=1,2.$ Moreover, we let
$$g(d) \colon \mathscr{T}_1(d) \to \mathscr{T}_2(d)$$ be the morphism on Néron lft-models induced by $g_{\eta}\times_K \Id_{K(d)},$ and by $\overline{\mathscr{R}}(d)$ the canonical compactification of $\mathscr{R}(d)^0$ constructed in Paragraph \ref{boundingthicknesspar}. Note that $\overline{\mathscr{R}}(d)_\eta = \overline{\mathscr{R}}_\eta\times_KK(d).$ 

\begin{proposition}\label{constantprop}
There exists a constant $B=B(g_{\eta})$ such that, for all $d$ prime to $p,$ we have 
$$\Theta(g(d)_k^0) \leq B.$$ Here, $g(d)_k^0$ denotes the map $\mathscr{T}_{1}(d)^0_k \to \mathscr{T}_{2}(d)^0_k$ on identity components of special fibres. In particular, $B$ does not depend upon $d.$ 
\end{proposition} 
\begin{proof}
Put $m:=\dim  T_1$ and choose a finite Galois extension $K\subseteq L$ which splits $T_1.$ Moreover, choose a smooth surjection $\Res_{L/K}\mathbf{G}_{\mathrm{m}}^m \to T_1$ with connected kernel as in Lemma \ref{Reslem} and denote the composition $\Res_{L/K}\mathbf{G}_{\mathrm{m}}^m \to T_2$ by $\xi_{\eta}.$ Let $\mathscr{R}(d)$ denote the Néron lft-model of $(\Res_{L/K}\mathbf{G}_{\mathrm{m}}^m) \times _K K(d)  = \Res_{L\otimes_KK(d)/K(d)}\mathbf{G}_{\mathrm{m}}^m$ and denote by $\xi(d)^0_k$ the composition
$$\mathscr{R}(d)^0_k \to \mathscr{T}_1(d)^0_k \overset{g(d)^0_k}{\to} \mathscr{T}_2(d)^0_k$$ of the induced morphisms on identity components of special fibres. Now let $\mathscr{X}(d)$ be the scheme-theoretic closure of $\ker\xi_{\eta}\times_KK(d)$ in $\overline{\mathscr{R}}(d)$ as in Paragraph \ref{boundingthicknesspar}. Let $f\colon \Spec K(d) \to \Spec K$ be the map induced by the field extension. Using Lemma \ref{Thetaboundlem} and Proposition \ref{Thetaboundprop}, we obtain 
\begin{align*}\Theta(g(d)^0_k) \leq \Theta(\xi(d)^0_k) &\leq \int_{\overline{\mathscr{R}}(d)_\eta} c_1(\Og_{\overline{\mathscr{R}}(d)_\eta}(1))^{\rho} \cap [\mathscr{X}(d)_{\eta}]\\
&=\int_{\overline{\mathscr{R}}(d)_\eta} c_1(f^\ast\Og_{\overline{\mathscr{R}}_\eta}(1))^{\rho} \cap f^\ast[\mathscr{X}_{\eta}], \end{align*}
 where $\rho:=\dim \ker \xi_{\eta} = \dim \mathscr{X}(d)_{\eta}$ and the last equality uses that scheme-theoretic images of quasi-compact morphisms commute with flat base change \cite[Tag 081I]{Stacks}. Finally, we have 
$$\int_{\overline{\mathscr{R}}(d)_\eta} c_1(f^\ast\Og_{\overline{\mathscr{R}}_\eta}(1))^{\rho} \cap f^\ast[\mathscr{X}_{\eta}] = \int_{\overline{\mathscr{R}}_\eta} c_1(\Og_{\overline{\mathscr{R}}_\eta}(1))^{\rho}\cap [\mathscr{X}_{\eta}]=:B,$$
since the degree map $\CH_0(\overline{\mathscr{R}}_\eta) \to \Z$ as well as the Chern class operations from diagram (\ref{diag}) commute with finite extensions of $K$ (see \cite[Chapter 3.2, Theorem 3.2 (d)]{Fulton} for the second claim). 
\end{proof}

\begin{remark} Proposition \ref{constantprop} is the reason why the thickness is more useful in this context than, for example, Néron's measure for the defect of smoothness \cite[Chapter 3.3]{BLR}. Indeed, suppose that $e(d) \colon \Spec \Og_{K(d)} \to \mathscr{K}(d):=\ker g(d)$ is the zero section. Then the number 
$$\delta(e(d)):=\ell_{\Og_{K(d)}}((e(d)^\ast\Omega^1_{\mathscr{K}(d)/\Og_{K(d)}})_{\mathrm{tors}})$$ also measures how badly $g(d)$ fails to be smooth. However, this invariant contains much more information than the thickness, which is why it is usually \it unbounded \rm as $d$ becomes arbitrarily highly divisible. For example, suppose $K\subseteq L$ is a finite Galois extension and $g(d)^0 \colon \Res_{\Og_{L(d)}/\Og_{K(d)}} \Gm \to \Gm$ is induced by the norm map, with $d$ prime to $p[L:K].$ If $Q(d):=\mathrm{coker}\, \mathrm{tr}_{\Og_{L(d)}/\Og_{K(d)}},$ then one sees easily that there is a canonical isomorphism $$(e(d)^\ast\Omega^1_{\mathscr{K}(d)/\Og_{K(d)}})_{\mathrm{tors}} = \Ext^1_{\Og_{K(d)}}(Q(d), \Og_{K(d)}).$$ If $\Og_K = W(\F_2^{\mathrm{alg}})$ and $L:=K(\sqrt{2}),$ we find $Q(d)=\Og_{K(d)}/2\Og_{K(d)},$ whose length is indeed unbounded for $d\to\infty.$ We shall return to this example later. 

Moreover, this example also shows why de Shalit's explicit description of the Lie algebra \cite[Proposition A1.7]{CY} did not feature in the proof: Suppose $T$ is an algebraic torus over $K$ split by the finite Galois extension $K\subseteq L.$
Put $\Gamma:=\Gal(L/K),$ $T_1:=\Res_{L/K} T_L,$ and $T_2:=T_1/T.$ Finally, let $\mathscr{T}(d)$ be the Néron lft-model of $T\times_KK(d)$ for any $d$ prime to $p[L:K].$ In this situation, de Shalit gives a precise description of the image of the canonical map 
$\Lie\mathscr{T}(d)\to \Lie\mathscr{K}(d) = \Hom_\Gamma(X^\ast(T), \Og_{L(d)}).$ Such a description is particularly useful if one can control the invariant $c_{\mathrm{lift}}(T\times_KK(d)):=\ell_{\Og_{K(d)}}(\Lie \mathscr{K}(d)/\Lie \mathscr{T}(d))$ \cite[A1.14]{CY} measuring the discrepancy between these two $\Og_{K(d)}$-modules. This is possible if one works over a fixed base field in some situations \cite[A1]{CY}. However, if one allows $d$ to vary, $c_{\mathrm{lift}}(T\times_KK(d))$ is usually unbounded as well. Indeed, it follows from \cite[Theorem 2.1 (a)]{LLR} together with \cite[Chapter 9.6, Lemma 2]{BLR} and \cite[Lemma 4.3]{Chai} that $\delta(e(d)) = c_{\mathrm{lift}}(T\times_KK(d))$ for all such $d,$ so the same example as above works. This is explained by the dependence of all known non-trivial bounds of $c_{\mathrm{lift}}(T\times_KK(d))$ on the ramification filtration of $\Gamma,$ which is not preserved by the canonical group isomorphism $\Gal(L(d)/K(d))\to\Gamma.$
\end{remark}

\begin{lemma} \label{torsorlem}
Let $k$ be an algebraically closed field and let $\psi\colon G\to H$ be a faithfully flat morphism of smooth connected algebraic groups over $k.$ Let $\boldsymbol{\gamma}:=(\ker \psi)/(\ker\psi)_{\mathrm{red}}$ and $A:=G/(\ker \psi)_{\mathrm{red}},$ so that we have exact sequences
$$0 \to (\ker \psi)_{\mathrm{red}} \to \ker \psi \to \boldsymbol{\gamma} \to 0$$ and $$0 \to (\ker \psi)_{\mathrm{red}} \to G \to A \to 0.$$ Then the induced morphism 
$$\Spec \widehat{\Og}_{A,0} \to \Spec \widehat{\Og}_{H,0}$$ is a torsor for the infinitesimal $k$-group scheme $\boldsymbol{\gamma}.$
\end{lemma}

\begin{proof}
 First note that $\boldsymbol{\gamma}$ is indeed infinitesimal. This can be checked on $k$-points and hence follows immediately from the fact that any $k$-point of $\ker \psi$ factors through $(\ker \psi)_{\mathrm{red}}.$ Note moreover that we have an exact sequence
 $$0 \to \boldsymbol{\gamma} \to A \to H \to 0.$$
 The map $\Spec \Og_{A,0} \to \Spec \Og_{H,0}$ is a torsor for $\boldsymbol{\gamma}$ since $A\to H$ is radicial. In particular, the claim follows once we can show that the canonical map 
$$\widehat{\Og}_{H,0} \otimes_{\Og_{H,0}} \Og_{A,0} \to \widehat{\Og}_{A,0}$$ is an isomorphism. Here the only non-trivial part is that the ring on the left is local, which follows again because $A\to H$ is radicial \cite[Tag 07N9]{Stacks}. 
\end{proof}
Notice that if, for some positive integer $d$, there is a $\mu_d$-action on $G$ that respects the group operation, then by construction the action of $\boldsymbol{\gamma}$ on $\Spec \widehat{\Og}_{A,0}$ satisfies condition (\ref{actioncondition}).

The following is slightly more general than immediately needed (but we shall use the general case later):

\begin{lemma} \label{pseudoreductivejumpslem}\label{inducedjumpsprop}
Suppose $D=\Res_{L/K} \mathbf{G}_{\mathrm{m}}$ for some finite (not necessarily separable) extension $K\subseteq L.$ Let $d\equiv 1 \mod [L:K]$ be prime to $p.$ Then the $d$-jumps of $D$ are $0,(d-1)/[L:K],...,([L:K]-1)(d-1)/[L:K].$ In particular, the jumps of $D$ are the elements
$$0, \frac{1}{[L:K]},..., \frac{[L:K]-1}{[L:K]}$$ of $\mathbf{R}/\Z.$
\end{lemma}

\begin{proof}
The same proof as that of \cite[Corollary 2.2.2]{OvI} works (the separability hypothesis imposed there is superfluous).
\end{proof}

We are now ready to prove

\begin{theorem} \label{jumpsratthm}
\begin{enumerate}[(i)]
\item Let $T$ be an algebraic torus over $K$ split by the Galois extension $K\subseteq L$ and let $j\in \mathbf{R}/\Z$ be a jump of $T.$ Then $[L:K]j = 0.$ In particular, $j$ is rational. 
\item Let $g_{\eta} \colon T_1 \to T_2$ be a smooth surjective morphism of algebraic tori over $K$ with connected kernel. Let $n:=\dim T_2$ and $m:=\dim T_1.$ Let $j_1\leq ... \leq j_m \in \mathbf{R}/\Z$ be the jumps of $T_1.$ Then there exist an injective function $\beta\colon \{1,..., n\} \to \{1,...,m\}$ and nonnegative integers $f_1,..., f_n$ such that $p^{f_1} j_{\beta(1)},..., p^{f_n} j_{\beta(n)}$ are the jumps of $T_2.$
\end{enumerate}
\end{theorem}
\begin{proof}
We shall show (ii) first. Choose a grid sequence $(d_{\ell})_{\ell\in \N}.$ For each $\ell,$ let $g(d_\ell) \colon \mathscr{T}_1(d_\ell) \to \mathscr{T}_2(d_\ell)$ be the morphism of Néron lft-models induced by $g_\eta\times_K \Id_{K(d_\ell)}$ and let $g(d_\ell)^0_k$ be the induced map on identity components of special fibres. We shall write 
$$\boldsymbol{\alpha}(d_\ell) := (\ker g(d_\ell)^0_k)/(\ker g(d_\ell)^0_k)_{\mathrm{red}},$$ which is an infinitesimal $k$-group scheme. Finally, for $\ell\in\N,$ put $$A(d_\ell):=\mathscr{T}_1(d_\ell)^0_k/(\ker g(d_\ell)^0_k)_{\mathrm{red}}.$$ By Lemma \ref{torsorlem}, the map $\Spec \widehat{\Og}_{A(d_\ell),0} \to \Spec \widehat{\Og}_{\mathscr{T}_2(d_\ell)^0_k,0}$ is a torsor for $\boldsymbol{\alpha}(d_\ell).$ Note that the group $\mu_{d_{{\ell}}}$ naturally acts on the complete local rings $$\widehat{\Og}_{\mathscr{T}_2(d_\ell)^0_k,0}\subseteq \widehat{\Og}_{A(d_\ell),0} \subseteq \widehat{\Og}_{\mathscr{T}_1(d_\ell)^0_k,0}$$ in a compatible way (condition (\ref{actioncondition}) is satisfied because the $\mu_{d_\ell}$-actions considered here respect the group structures). Suppose $\mu_{d_{\ell}}$ acts with weights $j'_{1,\ell} \leq j'_{2, \ell} \leq... \leq j'_{n, \ell}\in \Z/d_{\ell}\Z$ on $\widehat{\Og}_{A(d_\ell),0}$ and with weights $j_{1, \ell} \leq...\leq j_{m, \ell}$ on $\widehat{\Og}_{\mathscr{T}_1(d_\ell)^0_k,0}.$ Because the morphism $\mathscr{T}_1(d_{\ell})^0_k \to A(d_{\ell})$ is smooth, there exists a strictly increasing function $\beta_{\ell} \colon \{1,...,n\} \to \{1,...,m\}$ such that $j'_{i, \ell} = j_{\beta_{\ell}(i), \ell}$ for all $i=1,...,n.$ Now choose a scale $e_{1,\ell},...,e_{n, \ell}$ of the $\mu_{d_\ell}$-equivariant $\boldsymbol{\alpha}(d_\ell)$-torsor $\Spec \widehat{\Og}_{A(d_\ell),0} \to \Spec \widehat{\Og}_{\mathscr{T}_2(d_\ell)^0_k,0}$ relative to the chosen weights (cf. Definition \ref{scaledef}). By Proposition \ref{weightspermprop}, there exist permutations $\tau_{\ell} \in S_n$ such that $\mu_d$ acts with weights $$p^{e_{\tau_{\ell}(1), \ell}} j'_{\tau_{\ell}(1), \ell} \leq ... \leq p^{e_{\tau_{\ell}(n), \ell}} j'_{\tau_{\ell}(n), \ell}$$ on $\widehat{\Og}_{\mathscr{T}_2(d_\ell)^0_k,0}.$ By the pigeonhole principle, we may choose a subsequence of $(d_{\ell})_{\ell}$ such that the sequence $(\tau_\ell)_\ell$ of permutations, as well as the sequence $(\beta_{\ell})_{\ell},$ are constant. We shall denote by $\tau$ and $\beta_0$ the stable values of these sequences. Putting $\beta:= \beta_0\circ \tau,$ we find that $\mu_{d_\ell}$ acts with weights 
$$p^{e_{\tau(1), \ell}} j_{\beta(1), \ell} \leq ... \leq p^{e_{\tau(n), \ell}} j_{\beta(n), \ell}$$
on $\widehat{\Og}_{\mathscr{T}_2(d_\ell)^0_k,0}.$ Now choose a constant $B$ as in Proposition \ref{constantprop} (so that, in particular, $B$ is independent of $\ell$). Because
\begin{align*}p^{e_{\tau(1), \ell} + ... + e_{\tau(n), \ell}} &= \dim_k \Gamma(\boldsymbol{\alpha}(d_{\ell}), \Og_{\boldsymbol{\alpha}(d_{\ell})})\\
&= \Theta(g(d_\ell)^0_k) \leq B,
\end{align*}
the sequences $(e_{\tau(1),\ell})_\ell, ..., (e_{\tau(n), \ell})_\ell$ take values in a finite set. Again by the pigeonhole principle, we may choose another subsequence of $(d_\ell)_\ell$ such that $e_{\tau(i),\ell}$ is constant for all $i=1,...,n;$ we shall denote this stable value by $f_i.$ By Proposition \ref{jumpslimitprop}, for $i=1,..,n,$ the $i$-th jump of $T_2$ is 
$$\lim_{\ell \to \infty} \frac{p^{f_i} j_{\beta(i), \ell}}{d_{\ell}} = p^{f_i} \lim_{\ell \to \infty} \frac{j_{\beta(i), \ell}}{d_{\ell}} = p^{f_i} j_{\beta(i)},$$ as claimed. 

Part (i) follows immediately from part (ii) together with Lemma \ref{Reslem}. Indeed, if $T$ is split by the finite Galois extension $K\subseteq L,$ then the jumps of $\Res_{L/K}\mathbf{G}_{\mathrm{m}}^{\dim T}$ are known to be $[L:K]$-torsion elements of $\mathbf{R}/\Z$ by Proposition \ref{inducedjumpsprop}.
\end{proof}

\begin{example} \label{Example01} We shall consider two brief examples which show in particular that Theorem \ref{jumpsratthm} cannot be substantially improved upon. 
\begin{enumerate} 
\item In some simple situations, one can describe the maps $g(d)^0_k$ which appear in the proof of Theorem \ref{jumpsratthm} explicitly. Let $K\subseteq F$ be a finite Galois extension of $K,$ and consider the norm map $g(d)_{\eta}:=N_{F(d)/K(d)} \colon \Res_{F(d)/K(d)}\Gm \to \Gm.$ To keep things simple, we shall only consider $d$ prime to $[F:K],$ which covers all interesting cases. Let $\pi_{F(d)}$ be a uniformiser of $F(d).$ Then $\Og_{F(d)} = \Og_{K(d)}[\pi_{F(d)}]$ and $\pi_{F(d)}\otimes 1$ is nilpotent in $\Og_{F(d)}\otimes_{\Og_{K(d)}} k.$ The induced map 
$$g(d)^0_k \colon \Res_{(\Og_{F(d)}\otimes_{\Og_{K(d)}} k)/k } \Gm \to \Gm$$ is then simply the norm map $N_{(\Og_{F(d)}\otimes_{\Og_{K(d)}} k)/k}.$ One can write the algebraic group on the left as $\Gm \times_k U$ for some smooth connected unipotent algebraic group $U$ over $k;$ the map $g(d)^0_k\colon \Gm \times_k U \to \Gm$ is then given by $(a,b)\mapsto a^{[F:K]}.$ In particular, if $p^s$ is the maximal power of $p$ which divides $[F:K],$ then $\Theta(g(d)^0_k) = p^s,$ which is bounded by $[F:K].$

In fact, this is precisely the bound predicted by Proposition \ref{Thetaboundprop}: after base change to $F(d),$ the norm map becomes the summation map $\mathbf{G}_{\mathrm{m}}^{[F:K]} \to \Gm,$ which is given in coordinates by $(x_1,..., x_{[F:K]}) \mapsto x_1\cdot...\cdot x_{[F:K]}.$ This shows that the scheme $\mathscr{X}_\eta(d) \times_{K(d)} F(d)$ is the projective variety inside $\mathbf{P}^{[F:K]}_{F(d)}$ cut out by the homogeneous polynomial $x_0^{[F:K]} - x_1\cdot...\cdot x_{[F:K]},$ which is visibly of degree $[F:K].$ One can generalise this observation and express the bound from Proposition  \ref{Thetaboundprop} in terms of simpler invariants of the map $\xi_{\eta}$ introduced at the beginning of Paragraph \ref{boundingthicknesspar}, which we shall leave to the interested reader.
\item Now suppose that $\Og_K$ is of residue characteristic 2. We shall construct a smooth surjection $T_1\to T_2$ of algebraic tori with connected kernel such that the multisets of jumps of $T_1$ and $T_2$ are disjoint. This shows that multiplying jumps with powers of $p$ as in Theorem \ref{jumpsratthm} cannot be avoided. Let $K\subseteq L$ be a Galois extension such that $\Gal(L/K)\cong \Z/2\Z \times \Z/2\Z.$ Such an extension always exists and can be constructed, for example, as the compositum of two quadratic extensions with different discriminants. By Galois theory, $K\subseteq L$ contains quadratic subextensions $F_1, F_2, F_3$ such that $F_i\cap F_j = K$ for $i\not=j.$ Put $T_1:=\Res_{L/K}\Gm/\Res_{F_1/K}\Gm.$ We obtain an exact sequence
$$0 \to (\Res_{F_2/K}\Gm)/\Gm \to T_1 \to T_2 \to 0,$$ where $T_2$ is defined to be the obvious cokernel. Letting $\mathscr{T}_i$ be the Néron lft-models of $T_i$ for $i=1,2,$ we obtain the commutative diagram 
\tiny
$$\begin{tikzcd}
& 0 \arrow[d] & 0 \arrow[d] & 0 \arrow[d] \\
0 \arrow[r] & \Gm \arrow[r] \arrow[d]& \Res_{\Og_{F_2}/\Og_K}\Gm\arrow[r]\arrow[d] & (\Res_{\Og_{F_2}/\Og_K}\Gm)/\Gm \arrow[r]\arrow[d] & 0 \\
0 \arrow[r] & \Res_{\Og_{F_1}/\Og_K}\Gm \arrow[r]\arrow[d] & \Res_{\Og_L/\Og_K}\Gm  \arrow[r] \arrow[d] & \mathscr{T}_1^0 \arrow[r]\arrow[d] & 0 \\
0 \arrow[r] & (\Res_{\Og_{F_1}/\Og_K}\Gm)/\Gm \arrow[r]\arrow[d] & \Res_{\Og_L/\Og_K}\Gm /  \Res_{\Og_{F_2}/\Og_K}\Gm  \arrow[r] \arrow[d] & \mathscr{T}_2^0 \arrow[r]\arrow[d] & 0 \\
& 0 & 0 & 0
\end{tikzcd}$$
\normalsize
of identity components of Néron lft-models. Generically, all rows and columns of this diagram are exact; the same is true at the integral level for the first two rows as well as the first two columns by Proposition \ref{Chaiexactn}. One checks easily that there is a canonical isomorphism $T_2 = (\Res_{F_3/K}\Gm)/\Gm$ and that the map $T_1 \to T_2$ is induced by the norm map $N_{L/F_3} \colon \Res_{L/K}\Gm \to \Res_{F_3/K}\Gm.$

We may replace $K,$ $F_i,$ and $L$ by $K(d),$ $F_i(d)$, and $L(d)$ for any $d \equiv 1 \mod 4$ without affecting the truth of these claims. In particular, the multisets of jumps of $T_1$ and $T_2$ are $\{1/4, 3/4\}$ and $\{1/2\},$ respectively. This phenomenon is explained by the failure of the rightmost column to be exact. The map $\mathscr{T}^0_{1,k} \to \mathscr{T}^0_{2,k}$ has a particularly simple description if the extension $K\subseteq L$ is sufficiently wildly ramified, i. e. if the action of $\Gal(L/K)$ on $\Og_L\otimes_{\Og_K} k$ is trivial. According to Herbrand's theorem, this can always be arranged by replacing $L/K$ with $L(d)/K(d),$ for $d$ as above such that $d>4=[L:K]$ (cf. \cite[Proposition 3.1.2.1]{HN}). If $\pi_L$ denotes a uniformiser of $L,$ the $k$-algebra $\Og_L\otimes_{\Og_K} k$ is generated by the element $x:=\pi_L\otimes1$ which satisfies $x^4=0$ and the map $(\Og_L\otimes_{\Og_K}k)^\times  \to (\Og_{F_3}\otimes_{\Og_K} k)^\times$ induced by $N_{L/F_3}$ is given by $t\mapsto t^2.$ In particular, the subalgebra $\Og_{F_2}\otimes_{\Og_K} k$ is generated by $x^2.$  Writing the general element of $\Og_{L}\otimes_{\Og_K}k$ as $a_0 + a_1 x + a_2 x^2 + a_3 x^3$ and the general element of $\Og_{F_3}\otimes_{\Og_K}k$ as $b_0 + b_1x^2,$ we obtain isomorphisms $\widehat{\Og}_{\mathscr{T}^0_{1,k}} = k[\![a_1, a_3]\!]$ and $\widehat{\Og}_{\mathscr{T}^0_{2,k}} = k[\![b_1]\!].$ Replacing $K,$ $F_i,$ and $L$ by $K(d),$ $F_i(d)$, and $L(d)$ again (as above), we see that $\mu_d$ acts on $a_i$ with weight $i(d-1)/4$ for $i=0,...,3$ and on $b_j$ with weights $j(d-1)/2$ for $j=0,1.$ Finally, the map $\widehat{\Og}_{\mathscr{T}^0_{2,k}} \to \widehat{\Og}_{\mathscr{T}^0_{1,k}}$ is given in the chosen coordinates by $b_1 \mapsto a_1^2.$
\end{enumerate}
\end{example}
\subsection{Abelian varieties with potentially totally multiplicative reduction}

Now let $A$ be an Abelian variety with potentially totally multiplicative reduction. In other words, suppose that $A$ acquires semiabelian reduction over some finite Galois extension $K\subseteq L$ such that the identity component of the special fibre of the Néron lft-model of $A_L$ is a torus. These Abelian varieties admit a \it rigid uniformisation \rm by an algebraic torus (see \cite[Section 3.3.7]{HN}). Recall that such a rigid uniformisation is an exact sequence
$$0 \to M \to T \to A \to 0$$ of rigid analytic $K$-groups such that $M$ is an étale $K$-lattice (i. e. $M$ becomes isomorphic to $\Z^q$ with $q=\dim A$ over a finite separable extension of $K$) and $T$ is an algebraic torus (by a slight abuse of notation, we denote the rigid analytic $K$-groups associated with $T$ and $A$ by the same symbols). This is useful for us because, if $\mathscr{T}$ and $\mathscr{A}$ denote the Néron lft-models of $T$ and $A,$ respectively, we obtain a canonical étale morphism $\mathscr{T}^0_k \to \mathscr{A}^0_k.$ Since rigid uniformisation is compatible with finite extensions of $K,$ we obtain 
\begin{theorem}
Let $A$ be an Abelian variety over $K$ which acquires totally multiplicative (semiabelian) reduction over the finite Galois extension $K\subseteq L.$ Let $j\in \mathbf{R}/\Z$ be a jump of $A.$ Then $[L:K] j =0.$ In particular, $j$ is rational. \label{Abelianjumpsratthm}
\end{theorem}
\begin{proof}
By the same argument as in the proof of \cite[Proposition 6.2.3.2]{HN}, we obtain canonical isomorphisms $\widehat{\Og}_{\mathscr{A}(d)^0_k,0} = \widehat{\Og}_{\mathscr{T}(d)^0_k,0}$ for all $d$ prime to $p.$ Hence the claim follows from Theorem \ref{jumpsratthm}.
\end{proof}

\begin{remark} The reason why one cannot deduce the rationality of the jumps of all semiabelian varieties in the same way as for algebraic tori is that Lemma \ref{flatlem} no longer holds for Abelian varieties. More precisely, in \cite[Section 4.2]{OS}, Suzuki and the first-named author show that there exists an ordinary elliptic curve $A$ over $K:=\F^\mathrm{alg}_p(\!(t)\!)$ with good supersingular reduction over $\Og_K:=\F^\mathrm{alg}_p[\![t]\!]$ and a finite Galois extension $K\subseteq L$ with the following property: if $C$ is the Abelian variety over $K$ which fits into the exact sequence $0 \to A \to \Res_{L/K}A_L \to C \to 0,$ and if $\mathscr{B}$ is the Néron lft-model of $\Res_{L/K}A_L$ and $\mathscr{C}$ that of $C,$ then the canonical map $\mathscr{B}\to \mathscr{C}$ is not flat.
\end{remark}

\begin{lemma} \label{toricranklem}
Let $G$ be a smooth connected wound \rm affine \it algebraic group over $K$ and $\rho_{\mathrm{spl}}(G)$ be the dimension of the maximal \rm split \it subtorus of $G.$ Let $\mathscr{G}$ be the Néron lft-model of $G$ over $\Og_K.$ Then the dimension of the maximal torus of $\mathscr{G}^0_k$ is equal to $\rho_{\mathrm{spl}}(G).$  
\end{lemma}
\begin{proof}
Let $\mathscr{N}$ be the Néron lft-model of the maximal split subtorus of $G.$ Then the map $\mathscr{N} \to \mathscr{G}$ is a closed immersion and $\mathscr{G}/\mathscr{N}$ is the Néron lft-model of $G/T$ by Proposition \ref{Chaiexactn}. Since the group of connected components of $\mathscr{N}_k$ is free, we obtain an exact sequence $0 \to \mathscr{N}^0_k \to \mathscr{G}^0_k \to (\mathscr{G}_k/\mathscr{N}_k)^0 \to 0.$ In particular, we may assume without loss of generality that $\rho_{\mathrm{spl}}(G)=0.$ Now consider the exact sequence $0 \to T \to G \to U \to 0,$ where $T$ is a torus and $U$ wound unipotent. By \cite[Proposition 4.4]{N}, the identity component $\mathscr{T}^0_k$ is unipotent, and hence annihilated by $p^ a$ for some $a\in \N.$ Moreover, $U$ is annihilated by $p^b$ for some $b\in \N.$ The morphism $[p^{a+b}] \colon \mathscr{G}^0 \to \mathscr{G}^0$ hence factors through the map $[p^a]\colon \mathscr{T}^0\to \mathscr{T}^0,$ so the induced map on the special fibre vanishes. This shows that $\mathscr{G}_k^0$ is unipotent and hence contains no torus. 
\end{proof}

For the following result only, we shall regard the jumps as elements of $[0,1]$ rather than $\mathbf{R}/\Z$ in order to be able to distinguish between 0 and 1 (cf. Remark \ref{leftconremark}). Let $G$ be a smooth connected wound algebraic group over $K.$ For any $j\in[0,1],$ we let the \it multiplicity of $j$ as a jump of $G$ \rm be the multiplicity from Definition \ref{multiplicitydef} if $j$ is a jump of $G,$ and 0 otherwise.

\begin{proposition}
    Let $G$ be an algebraic torus or an Abelian variety with potentially totally multiplicative reduction over $K$ with Néron lft-model $\mathscr{G}\to\Spec\Og_K$. Then the multiplicity $m^0(G)$ of $0 \in [0,1]$ as a jump of $G$ is equal to the dimension of the maximal torus of $\mathscr{G}^0_k.$ If $G$ is a torus, this is equal to the dimension $\rho_{\mathrm{spl}}(G)$ of the maximal \rm split \it subtorus of $G.$ 
\end{proposition}
\begin{proof}
    We show the claim for algebraic tori first. By Proposition \ref{Chaiexactn} and Lemma \ref{universallyexactlem}, we may assume without loss of generality that $\rho_{\mathrm{spl}}(G)=0$ and show that $m^0(G)=0.$ Note that it is sufficient to exhibit a single $d$ prime to $p$ such that $0$ is not a $d$-jump of $G$: To see this, choose a grid sequence $(d_{\ell})_\ell$ such that $d=d_{s}$ for some $s\in\N.$ Let $j_{1, \ell} \leq ... \leq j_{n, \ell}$ be the $d_{\ell}$-jumps of $G$ with $n:=\dim G.$ Then the sequences $(j_{i, \ell}/d_{\ell})_\ell$ are monotonically increasing \cite[p. 94]{HN} and their limits are the jumps of $G$. Let $L$ be a finite Galois extension of $K$ splitting $G$ and choose an isomorphism $G_L\cong \mathbf{G}_{\mathrm{m}}^n$ over $L.$ Since $\rho_{\mathrm{spl}}(G)=0,$ we obtain a smooth surjective morphism $R:=(\Res_{L/K}\mathbf{G}_{\mathrm{m}}^n)/\mathbf{G}_{\mathrm{m}}^n \to G$ with connected kernel. Now let $d\equiv 1 \mod p[L:K].$ Then the $d$-jumps of $R$ are the elements $\nu(d-1)/[L:K]$ for $\nu=1,...,[L:K]-1,$ each with multiplicity $n$ by Lemma \ref{pseudoreductivejumpslem}. We write them in increasing order as $j'_{1,s}\leq...\leq j'_{r,s}$ with $r:=n([L:K]-1)=\dim R.$ As in the proof of Theorem \ref{jumpsratthm}, we can find an injective function $\beta\colon \{1,...,n\}\to \{1,...,r\}$ and nonnegative integers $e_1,...,e_n$ such that $j_{1,s} =p^{e_1}j'_{\beta(1),s},...,j_{n,s} =p^{e_n}j'_{\beta(n),s}.$ Since $d$ is prime to $p,$ none of these elements vanish in $\Z/d\Z;$ both claims for tori then follow from Lemma \ref{toricranklem}.
    
    If $G$ is an Abelian variety with potentially totally multiplicative reduction, we choose a torus $T$ uniformising $G$ with Néron lft-model $\mathscr{T}\to \Spec \Og_K.$ Then $m^0(G)=m^0(T)$ and we have a canonical étale map $\mathscr{T}^0_k \to\mathscr{G}^0_k,$ so the dimensions of the respective maximal tori also coincide. %\color{orange} The final claim follows from Lemma \ref{toricranklem}. \color{black}
\end{proof}

\subsection{Unirational wound algebraic groups}
It is possible to generalise the results just obtained for algebraic tori to all unirational wound algebraic groups. As wound unirational groups which are not algebraic tori only exist over imperfect fields (this will be recalled below), we shall assume throughout this subsection that $K$ is of characteristic $p>0.$ We shall call an algebraic group $G$ \it unirational \rm if its underlying scheme is (scheme-theoretically) dominated by an open subscheme of some affine space over $K.$ In particular, unirational algebraic groups are smooth, connected, and affine. Unirationality is insensitive to finite separable extensions, but a non-unirational unipotent algebraic group always becomes unirational after a finite \it inseparable \rm extension.

Now let $D$ be a smooth connected affine algebraic group over $K.$ Then $D$ fits into an exact sequence $$0\to T \to D \to U \to 0,$$ where $T$ is an algebraic torus and $U$ is a smooth connected unipotent algebraic group over $K$ \cite[Proposition A.2.11]{CGP}.
\begin{proposition}
The algebraic group $D$ is unirational (resp. wound) if and only if so is $U.$
\end{proposition}
\begin{proof}
The first claim is \cite[Lemma 2.10]{Achet}. Because $T$ is a torus, we have $\Hom_K(\Ga, T) = 0 = \Ext^1_K(\Ga, T),$ so the map $\Hom_K(\Ga, D) \to \Hom_K(\Ga, U)$ is an isomorphism.
\end{proof}

It is well-known that algebraic tori can be described in terms of Galois representations on finitely generated free Abelian groups. A very similar description exists \it up to relative perfection \rm for unirational wound unipotent algebraic groups, on which large parts of what follows will rely. To make this precise, we shall recall some background from \cite{OSII}. 

First observe that, since $\Og_K$ is excellent and $k$ is algebraically closed, we have $[K:K^p]=p$ by \cite[Lemma 2.1]{OvII}. A morphism $Y\to X$ of $\F_p$-schemes is is \it relatively perfect \rm if the relative Frobenius $Y\to Y^{(p)}$ over $X$ is an isomorphism \cite[Definition 1.1]{Kato}; this is the case, for example, for weakly étale morphisms \cite[Tag 0F6W]{Stacks}. We shall denote by $K_{\mathrm{RP}}$ the site of relatively perfect $K$-schemes endowed with the étale topology (see \cite[Section 2]{OSII} for more details and further references). For a scheme $X$ over $K,$ we let $X^{(1/p^n)}$ be the Weil restriction of $X$ along the $n$-fold Frobenius on $K;$ if $X$ is a group scheme over $K,$ then so is $X^{(1/p^n)}.$ Each scheme $Y\to \Spec K$ admits a \it relative perfection \rm $Y^{\mathrm{RP}};$ this operation is right adjoint to the forgetful functor \cite[Lemma 1.5]{Kato}. More precisely, for any $K$-scheme $Y,$ we have a canonical map $Y^{(1/p)} \to Y,$ which is the composition of the relative Frobenius $Y^{(1/p)} \to Y^{(1/p)(p)}$ with the canonical map $Y^{(1/p)(p)} \to Y.$ We have $Y^{\mathrm{RP}} = \varprojlim Y^{(1/p^n)}$ \cite[\S 1]{Kato}; again, this is a group scheme if $Y$ is (relatively perfect algebraic groups were studied in detail in \cite[Section 8]{BS}). 

Moreover, we denote by $\nu_n(1)_K^{\mathrm{RP}}$ the dlog-part of the Hodge-Witt sheaf $W_n\Omega^1_{K/\mathbf{F}_p}$ and put $\nu_{\infty}(1)_K^{\mathrm{RP}}:=\varinjlim \nu_n(1)_K^{\mathrm{RP}}.$ We have $\nu_n(1)_K^{\mathrm{RP}} = \mathbf{G}_{\mathrm{m}}^{\mathrm{RP}}\otimes_{\Z}\Z/p^n\Z$ \cite[Proof of Proposition 6.4]{Kato}. A result of Suzuki shows that, if $G$ is a smooth wound unipotent algebraic group over $K,$ the sheaf $\sHom_{K_{\mathrm{RP}}}(G^{\mathrm{RP}}, \nu_{\infty}(1)_K^{\mathrm{RP}})$ is representable by the relative perfection of a smooth (not necessarily connected) wound unipotent algebraic group $H$ over $K$ \cite[Proposition 3.3]{SuzIII}; we shall call this relative perfection the \it dual \rm of $G^{\mathrm{RP}}.$ Finally, denoting by $\R\sHom_{K_{\mathrm{RP}}}(-,-)$ the internal Hom in the unbounded derived category of sheaves of Abelian groups on $K_{\mathrm{RP}},$ Suzuki shows that $\R\sHom_{K_{\mathrm{RP}}}(G^{\mathrm{RP}}, \nu_{\infty}(1)_K^{\mathrm{RP}}) = \sHom_{K_{\mathrm{RP}}}(G^{\mathrm{RP}}, \nu_{\infty}(1)_K^{\mathrm{RP}})$ and the canonical morphism 
$$G^{\mathrm{RP}} \to \R \sHom_{K_\mathrm{RP}}(H^{\mathrm{RP}}, \nu_{\infty}(1)_K^{\mathrm{RP}})$$ is an isomorphism \cite[Propositions 3.3, 3.4 and 3.5]{SuzIII}. The description of unirational wound unipotent groups is then the following statement: 
\begin{proposition}\rm (\cite[Theorem 10.5]{OSII}) \it \label{classprop}
A wound unipotent algebraic group $G$ over $K$ is unirational if and only if $\sHom_{K_\mathrm{RP}}(G^{\mathrm{RP}}, \nu_{\infty}(1)_K^{\mathrm{RP}})$ is representable by a ($p$-primary) \rm étale \it group scheme $H$ over $K.$ Under this duality, the constant étale $K$-group scheme $\Z/p^n\Z$ corresponds to the relative perfection of $\mathbf{G}_{\mathrm{m}}^{(1/p^n)}/\Gm.$
\end{proposition}

This means in particular that, up to relative perfection, any unirational wound unipotent algebraic group becomes isomorphic to an algebraic group of the form 
$$G = \mathbf{G}_{\mathrm{m}}^{(1/p^{n_1})}/\Gm \oplus ... \oplus \mathbf{G}_{\mathrm{m}}^{(1/p^{n_r})}/\Gm$$ for some $r\in \N_0$ over some finite separable extension of $K.$

\begin{definition} \label{Galoissplitdef}
A unirational wound unipotent algebraic group $G$ is \rm Galois-split \it if the étale $K$-group scheme representing $\sHom_{K_\mathrm{RP}}(G^{\mathrm{RP}}, \nu_{\infty}(1)_K^{\mathrm{RP}})$ is constant (i. e. if $G^{\mathrm{RP}}$ is of the form described just above this definition).
\end{definition}

Ultimately, the goal of the present section is to show that the jumps of every unirational wound algebraic group are rational numbers. The proof will follow the same line of reasoning as that for algebraic tori; we shall therefore only indicate the necessary changes. The first main step towards this goal will be the following: 

\begin{proposition}\label{cohomologyvanishingprop}
Let $G$ be a unirational wound algebraic group over $K.$ Then $H^i(K,G)=0$ for $i\geq 1.$ 
\end{proposition}
\begin{proof}
For the moment, let $G$ be a \it Galois-split \rm unirational wound \it unipotent \rm algebraic group. We shall establish the claim for such $G$ first. Using Proposition \ref{classprop}, we obtain an isomorphism 
$$(\mathbf{G}_{\mathrm{m}}^{(1/p^{n_1})}/\Gm \oplus...\oplus \mathbf{G}_{\mathrm{m}}^{(1/p^{n_r})}/\Gm)^{\mathrm{RP}} \to G^{\mathrm{RP}}.$$ 
for suitable $n_1,..., n_r\in \N_0.$ Hence, by \cite[Proposition 8.9]{BS}, it suffices to show that we have $H^i(K,\mathbf{G}_{\mathrm{m}}^{(1/p^{\nu})}/\Gm) = 0$ for all $\nu\geq 0$ and $i\geq 1.$ Because $\Og_K$ is a $C_1$-field, the strict cohomological dimension of $\Gal(K\sep/K)$ is $\leq 2,$ so by considering the exact sequence $0 \to \Gm \to \mathbf{G}_{\mathrm{m}}^{(1/p^{\nu})} \to \mathbf{G}_{\mathrm{m}}^{(1/p^{\nu})}/\Gm \to 0,$ we deduce the claim for $i\geq2$ and observe that $H^1(K, \mathbf{G}_{\mathrm{m}}^{(1/p^{\nu})}/\Gm)$ is a subgroup of $\Br K,$ which vanishes as well because $K$ is a $C_1$-field.

Now let $G$ be any unirational wound unipotent algebraic group over $K.$ The rest of the proof is inspired by that of \cite[Lemma 4.3]{Chai}. Again by Proposition \ref{classprop}, there exists a finite Galois extension $K\subseteq L$ such that $G_L:=G\times_KL$ is Galois-split. Exactly as in the proof of Lemma \ref{Reslem}, we construct an exact sequence
$$0 \to G' \to \Res_{L/K} G_L \to G \to 0$$ such that $G'$ is smooth, connected, and (necessarily) wound unipotent. Moreover, $$H^ i(K, \Res_{L/K}G_L) = H^ i(L,G_L) =0$$ for $i\geq 1$ as we have just seen. We deduce that $H^i(K,G)=0$ for $i\geq 2.$ Now observe that $G'_L$ is isomorphic to $G_L^{[L:K]-1}$ and hence unirational; in particular, so is $G'$ \cite[Theorem 7.9]{RosII}. Hence $H^2(K,G')=0$ and we deduce that $H^1(K,G)$ vanishes as well. In general, $G$ fits into an exact sequence $0 \to T \to G \to U \to 0$ with $T$ a torus and $U$ wound unipotent. Since the Proposition is known for tori \cite[Lemma 4.3]{Chai}, the long exact cohomology sequence shows the claim for $G.$ 
\end{proof}

\begin{remark} The reader may contrast Proposition \ref{cohomologyvanishingprop} with a result of Rosengarten \cite[Theorem 1.6]{Ros}, according to which the first étale cohomology group of a non-split smooth connected unipotent group over a field which is finitely generated over $\F_p$ is \it infinite. \rm \end{remark}

\begin{proposition} \label{unipcoveringprop}
Let $G$ be a unirational wound unipotent algebraic group over $K$ which becomes Galois-split over the finite Galois extension $K\subseteq L.$ Then there exist nonnegative integers $n_1,..., n_r$ and a faithfully flat homomorphism $$\Res_{L/K} \mathbf{G}_{\mathrm{m}}^{(1/p^{n_1})} \oplus ... \oplus \Res_{L/K} \mathbf{G}_{\mathrm{m}}^{(1/p^{n_r})} \to G$$ whose kernel $\Delta$ satisfies $H^1(F, \Delta_F)=0$ for all finite separable extensions $K\subseteq F.$
\end{proposition}

\begin{proof}
%Choose $K\subseteq L$ such that $G$ becomes Galois-split over $L.$ 
By \cite[Tag 01ZC]{Stacks}, Proposition \ref{classprop}, and the universal property of relative perfection, there exist $n'_1,..., n'_r$ as in the Proposition, some $N\geq 0,$ and an homomorphism 
$$(\mathbf{G}_{\mathrm{m}}^{(1/p^{n'_1})}/\Gm \oplus ...\oplus \mathbf{G}_{\mathrm{m}}^{(1/p^{n'_r})}/\Gm)^{(1/p^N)} \to G_L$$ over $L$
which becomes an isomorphism after relative perfection. Taking the Weil restriction along $K\subseteq L,$ we obtain an homomorphism 
$$(\Res_{L/K} \mathbf{G}_{\mathrm{m}}^{(1/p^{n_1})} \oplus ... \oplus \Res_{L/K} \mathbf{G}_{\mathrm{m}}^{(1/p^{n_r})})/\Res_{L/K}\mathbf{G}_{\mathrm{m}}^{(1/p^N), \oplus r} \to \Res_{L/K}G_L$$ (with $n_j:=n_j'+N$) which becomes an isomorphism after relative perfection. In particular, this morphism is faithfully flat and its kernel $\Delta'$ satisfies $H^1(F, \Delta_F')=0$ for all finite separable extensions $K\subseteq F$ by \cite[Proposition 8.9]{BS}. Now we consider the canonical morphism $\Res_{L/K}G_L\to G$ constructed in the proof of Proposition \ref{cohomologyvanishingprop}, where we have also seen that its kernel $G'$ is unirational and wound unipotent, so satisfies $H^1(F, G_F')=0$ by Proposition \ref{cohomologyvanishingprop}. Hence we obtain a morphism as in the proposition whose kernel $\Delta$ admits a filtration with successive quotients $G',$ $\Delta',$ and $\Res_{L/K}\mathbf{G}_{\mathrm{m}}^{(1/p^N), \oplus r}.$ The desired cohomological property is known in each case, so the claim follows. 
\end{proof}

\begin{theorem} \label{jumpsratthmGEN}
Let $G$ be a unirational wound algebraic group over $K.$ Let $j\in \mathbf{R}/\Z$ be a jump of $G.$ Then $j$ is a torsion element of $\mathbf{R}/\Z$ (i. e. $j$ is rational). 
\end{theorem}

\begin{proof}
The strategy is the same as for algebraic tori; we shall therefore only indicate what needs to be changed. 
The group $G$ fits into an exact sequence $0 \to T \to G \to U \to 0,$ where $T$ is a torus and $U$ a smooth connected unirational wound unipotent group. Let $K\subseteq L$ be a finite Galois extension of $K$ which splits $T$ and such that $U$ is Galois-split over $L.$
Choose an homomorphism $$\Res_{L/K} \mathbf{G}_{\mathrm{m}}^{(1/p^{n_1})} \oplus ... \oplus \Res_{L/K} \mathbf{G}_{\mathrm{m}}^{(1/p^{n_r})} \to U$$ as in Proposition \ref{unipcoveringprop}. Applying \cite[Corollary A.5.4(3)]{CGP} we obtain a commutative diagram 

$$\begin{tikzcd}
0 \arrow[r] &\Res_{L/K} T_L \arrow[r]\arrow[swap]{d}{\Id} &\widetilde{G} \arrow[r]\arrow[d] & \Res_{L/K} \Big(\mathbf{G}_{\mathrm{m}}^{(1/p^{n_1})} \oplus ... \oplus \mathbf{G}_{\mathrm{m}}^{(1/p^{n_r})}\Big) \arrow[r]\arrow[d] & 0 \\ 
0 \arrow[r] &\Res_{L/K} T_L \arrow[r]\arrow[d] &\Res_{L/K} G_L \arrow[r] \arrow[d] &\Res_{L/K} U_L \arrow[r]\arrow[d] &0\\
0 \arrow[r] & T \arrow[r] & G \arrow[r] & U \arrow[r] & 0, 
\end{tikzcd}$$
with exact rows, where $\widetilde{G}$ is the obvious pullback. Denote by $\widetilde{\mathscr{G}}$ the Néron lft-model of $\widetilde{G}$ and choose an isomorphism $T_L \to \mathbf{G}_{\mathrm{m}}^{\dim T}$ over $L.$ The sequence of Néron lft-models induced by the top row is exact by Proposition \ref{Chaiexactn}. Since the group of connected components of the first non-zero Néron lft-model is free, we moreover obtain an exact sequence
$$0 \to \Res_{\Og_L/\Og_K}\mathbf{G}_{\mathrm{m}}^{\dim T} \to \widetilde{\mathscr{G}}^0 \to \Res_{\Og_L/\Og_K} \Big(\mathbf{G}_{\mathrm{m}}^{(1/p^{n_1})} \oplus ... \oplus \mathbf{G}_{\mathrm{m}}^{(1/p^{n_r})}\Big) \to 0$$ of identity components. We claim that the second-to-last map admits a \it scheme-theoretic \rm section generically; the Néron mapping property then induces such a section at the integral level. To see this, let $V$ denote the scheme underlying the last nontrivial object in the top row. Since Weil restriction commutes with open immersions, $V$ is an open subscheme of some affine space. Now we observe that $H^1(V, \Res_{L/K}T_L) = H^1(V_L, T_L).$ Since $V_L$ is also an open subscheme of some affine space and since $T_L$ is split, this group vanishes. We fix such a section from now on and obtain an open immersion
$$\widetilde{\mathscr{G}}^0 \to \mathbf{P}(\Og_L^{\vee, \dim T} \oplus \Og_{L^{(1/p^{n_1})}}^\vee \oplus ... \oplus \Og_{L^{(1/p^{n_r})}}^\vee \oplus \Og_K)=:\overline{\mathscr{G}}.$$ Now let $d$ be prime to $p.$ The section we chose above induces one over $\Og_{K(d)};$ we let $\overline{\mathscr{G}}(d)$ be the resulting compactification, which satisfies $\overline{\mathscr{G}}(d)_\eta = \overline{\mathscr{G}}_\eta \times_KK(d)$ canonically. Let $g(d)^0 \colon \widetilde{\mathscr{G}}(d)^0 \to \mathscr{G}(d)^0$ be the morphism of identity components of Néron lft-models over $\Og_{K(d)}$ induced by the map $\widetilde{G}\times_KK(d) \to G\times_KK(d).$ We let $\mathscr{X}(d)$ be the scheme-theoretic closure of $\ker g(d)^0_{\eta}$ inside $\overline{\mathscr{G}}(d).$ It follows from Proposition \ref{unipcoveringprop}, \cite[Lemma 4.3]{Chai} and the same argument as in the proof of Lemma \ref{flatlem} that $g(d)^0_k$ is faithfully flat. The arguments analogous to those from Section \ref{boundingthicknesspar} and Subsection \ref{torisubsec} then go through \it mutatis mutandis; \rm in particular, one obtains a bound $B$ on $\Theta(g(d)^0_k)$ independent of $d$ as in Proposition \ref{constantprop}. The proof of Theorem \ref{jumpsratthm} can now be taken without change (but using a grid sequence $(d_\ell)_\ell$ such that $d_\ell$ is prime to $[L:K]$ for all $\ell$) once we can show that the jumps of $\widetilde{G}$ are rational. This follows from Lemma \ref{pseudoreductivejumpslem} and Lemma \ref{universallyexactlem}.
\end{proof}

\begin{remark} Our proof that the jumps of a unirational wound algebraic group $G$ are rational is not constructive in the sense that it does not provide a formula for the jumps. Even if $G$ is an algebraic torus, no such formula is known except if $G$ admits an explicit description in terms of \it induced \rm tori. It would be very interesting to know how the jumps depend on the character lattice of an algebraic torus, or the finite $p$-primary Galois module dual to $G$ if $G$ is wound unipotent. \end{remark}

\section{Motivic zeta functions}
\subsection{Duality for unirational wound unipotent groups}
In this subsection, we shall recall a duality result for unirational wound unipotent groups over $K$ due to Suzuki \cite{SuzIII}\footnote{More precisely, we shall give a new proof of a duality theorem for unirational unipotent algebraic groups (resembling a similar result for tori) which is a special case of results proven in \it op cit. \rm We shall give precise references below and explain in more detail why we chose this particular shortcut to the results we need. Readers familiar with the duality theory set out in \cite{Suz, SuzII, SuzIII, SuzIV} can skip this paragraph and refer to it when necessary.}. In general, such duality results are useful for controlling the group of connected components of a Néron lft-model. We are ultimately interested in doing so for general unirational algebraic groups over $K,$ but satisfactory duality theories exist only for tori and unipotent groups separately. The main difficulty we need to overcome here is that these duality theorems have been proven in slightly different site-theoretic formulations, whereas we shall need a unified framework. 
We shall use methods introduced by Suzuki \cite{Suz, SuzII}; see \cite[§1.2]{Suz} for a general picture and \cite[Section 3]{OS} for a brief introduction.\footnote{The notation in \cite{Suz, SuzII} is slightly different from ours; what we shall call $\boldsymbol{\Gamma}(K,-)$ is called $\widetilde{\boldsymbol{\Gamma}}(K,-)$ in \it op. cit., \rm whereas $\boldsymbol{\Gamma}(K,-)$ is reserved for the corresponding \it étale \rm sheaf. Our notation is the same as that used in \cite{OS}. The two sheafifications, as well as their derived functors, coincide under the condition of \it P-acyclicity \rm \cite[Section 2.4]{Suz}, which is satisfied in particular if the argument is a finite flat group scheme, a lattice, or a smooth connected wound algebraic group \cite[Propositions 3.4.2 and 3.4.3]{Suz}.} Recall that $k$ is the (algebraically closed) residue field of $\Og_K.$ As in \cite[Definition 2.1.1]{SuzII}, we shall call a $k$-algebra \it rational \rm if it is a finite product of (inverse) perfections of finitely generated field extensions of $k.$ A $k$-algebra is \it ind-rational \rm if it is a filtered direct limit of rational $k$-algebras. We endow the category of ind-rational $k$-algebras with the pro-étale topology and call the resulting site $\Spec k^{\mathrm{indrat}}_{\proet}$ \cite{Suz}. The same category endowed with the étale topology will be called $\Spec k^{\mathrm{indrat}}_{\et}$ \cite{Suz}. The main reason why these sites are useful is that a (pro)algebraic group $G$ over $k$ can be recovered up to perfection from its functor of points on $\Spec k^{\mathrm{indrat}}_{\proet}$ (or $\Spec k^{\mathrm{indrat}}_{\et}$; cf. \cite[p. 799]{SuzII}). Again as in \cite[Section 2.3]{SuzII}, we associate, to any ind-rational $k$-algebra $k',$ a flat $\Og_K$-algebra $\mathbf{O}_K(k')$ and a $K$-algebra $\mathbf{K}(k')$ by putting
$$\mathbf{O}_K(k') := \Og_K \widehat{\otimes}_k k' := \varprojlim ( \Og_K/\mathfrak{m}_K^n \otimes_kk')$$ and $$\mathbf{K}(k') := \mathbf{O}_K(k')\otimes_{\Og_K}K.$$ 

We begin with the following simple observation: 

\begin{lemma}\label{relperflem}
For any ind-rational $k$-algebra $k',$ the $\Og_K$-algebra $\mathbf{O}_K(k')$ (resp. the $K$-algebra $\mathbf{K}(k')$) is relatively perfect. 
\end{lemma}

\begin{proof}
It suffices to show the claim for the morphism $k[\![t]\!] \cong \Og_K \to \mathbf{O}_K(k') \cong k'[\![t]\!].$ Since $k$ and $k'$ are both perfect, this is equivalent to showing that the canonical map $k[\![t]\!] \otimes_{k[\![t^p]\!]} k'[\![t^p]\!] \to k'[\![t]\!] $ is an isomorphism, which is clear. 
\end{proof}

Given a sheaf $\mathscr{F}$ of Abelian groups on the big fppf-site of $K,$ we let $\boldsymbol{\Gamma}(K, \mathscr{F})$ be the pro-étale sheafification of the presheaf
$$k' \mapsto \Gamma(\mathbf{K}(k'), \mathscr{F})$$ \cite[Section 3]{Suz}. 
By Lemma \ref{relperflem}, the functor $\boldsymbol{\Gamma}(K,-)$ factors canonically as 
$$\boldsymbol{\Gamma}(K,-)  = \boldsymbol{\Gamma}_{\mathrm{RP}}(K, \rho_\ast -),$$ where $\rho_\ast \colon \mathrm{Ab}((\mathrm{Sch}/K)_{\fppf}) \to \mathrm{Ab}(K_{\mathrm{RP}})$ is the forgetful functor and $\boldsymbol{\Gamma}_{\mathrm{RP}}(K,-)$ is defined exactly as $\boldsymbol{\Gamma}(K,-),$ with $\mathscr{F}$ now ranging through sheaves on $K_{\mathrm{RP}}.$ 
Slightly more formally, we have premorphisms\footnote{See \cite[Definition 2.3]{SuzIV} for the notion of \it premorphism of sites. \rm A premorphism is similar to a morphism of sites, except that the pullback functor need not be exact. Hence some additional care is required; the derived functors of pushforward compose as expected by \cite[Proposition 2.6]{SuzIV}.} of sites
\begin{align*}
(\mathrm{Sch}/K)_{\fppf} \overset{\rho}{\to} K_{\mathrm{RP}} \overset{\pi}{\to} \Spec k^{\mathrm{indrat}}_{\et} \overset{P}{\gets} \Spec k^{\mathrm{indrat}}_{\proet},
\end{align*}
where $\pi$ is given by $k'\mapsto \mathbf{K}(k')$ and $P$ is the identity functor (in fact, $\rho$ and $P$ are morphisms of sites); here and in the following, we use the same notation as \cite[Tag 00UZ]{Stacks} for functors between (derived) categories of sheaves. In this notation, we have $\boldsymbol{\Gamma}_{\mathrm{RP}}(K,-) = P^\ast\pi_\ast -.$
 Note that, for all $i\geq 0,$ the functors $\R^ i \boldsymbol{\Gamma}(K, -)$ and $\R^ i \boldsymbol{\Gamma}_{\mathrm{RP}}(K, -)$ commute with filtered colimits (since $\Spec \mathbf{K}(k')$ is clearly quasi-compact and quasi-separated for any ind-rational $k$-algebra $k',$ this follows from \cite[Tags 0739 and 0H7B]{Stacks}).
We shall write $D(K_{\mathrm{RP}})$ and $D(k)$ for the (unbounded) derived categories of sheaves of Abelian groups on $K_{\mathrm{RP}}$ and $\Spec k^{\mathrm{indrat}}_{\proet},$ respectively, and $\R\sHom_{K_{\mathrm{RP}}}(-,-)$ and $\R\sHom_k(-,-)$ for the internal Hom in those categories.

\begin{lemma} \label{RPcohomlem}
\begin{enumerate}[(i)]
\item For a \rm smooth \it group scheme $X$ over $K,$ we have $\R\rho_\ast X = X^{\mathrm{RP}};$ in particular,
$$\R\boldsymbol{\Gamma}(K,X) = \R\boldsymbol{\Gamma}_{\mathrm{RP}}(K, X^{\mathrm{RP}}).$$ The same is true if $X$ is a filtered colimit of smooth $K$-group schemes.
\item Moreover, for $A,B\in D(K_{\mathrm{RP}}),$ there is a canonical morphism 
$$\R\boldsymbol{\Gamma}_{\mathrm{RP}}(K, \R\sHom_{K_{\mathrm{RP}}}(A,B)) \to \R\sHom_k(\R\boldsymbol{\Gamma}_{\mathrm{RP}}(K,A),\R\boldsymbol{\Gamma}_{\mathrm{RP}}(K,B)).$$
\end{enumerate}
\end{lemma}

\begin{proof}
For claim (i), it is sufficient to treat the case of smooth $X.$ The functor $\rho_\ast$ can be factored as $\mathrm{Ab}((\mathrm{Sch}/K)_{\fppf}) \overset{\epsilon_{\ast}}{\to} \mathrm{Ab}((\mathrm{Sch}/K)_{\et}) \overset{f_\ast}{\to} \mathrm{Ab}(K_{\mathrm{RP}}),$ where $\epsilon_\ast$ and $f_\ast$ are the forgetful functors. By a theorem of Grothendieck \cite[Chapter III, Remark 3.11 (b)]{Milne}, we have $\R^i \epsilon_\ast X = 0 $ for $i\geq 1,$ whereas $f_\ast$ is exact (cf. \cite[Proof of Proposition 8.9]{BS}). Since $\epsilon_\ast$ sends injective sheaves to injective sheaves, we obtain $\R^i\rho_\ast X = 0$ for $i\geq 1,$ while $\rho_\ast X = X^{\RP}$ by definition. For the second claim, using \cite[Proposition 2.4]{SuzIV}, the proof of \cite[Proposition 3.3.8]{Suz} can be taken without change.
\end{proof}

Finally, suppose $\mathscr{G} \to \Spec \Og_K$ is a smooth separated group scheme. Then the (inverse) perfection of $\varprojlim \Res_{(\Og_K/\mathfrak{m}_K^i)/k} (\mathscr{G}\times_{\Og_K} \Og_K/\mathfrak{m}_K^ i)$ is called the \it Greenberg transform \rm of $\mathscr{G}.$ This is an object of the pro-category of perfect algebraic groups over $k.$ If $A$ is a wound algebraic group over $K$ with Néron lft-model $\mathscr{A}\to \Spec \Og_K,$ then $\boldsymbol{\Gamma}(K,A)$ is represented by the Greenberg transform of the Néron lft-model of $A$ over $\Og_K.$ 
Indeed, let $k'$ be an ind-rational $k$-algebra. The group of $k'$-points of the Greenberg transform of $\mathscr{A}$ is equal to $\varprojlim_i\mathscr{A}{(k'\otimes_k\Og_K/\mathfrak{m}_K^i)}$, which coincides with $\mathscr{A}(\mathbf{O}_K(k'))$ by \cite[Proposition 3.1.3 (d)]{Suz} and hence with $A(\mathbf{K}(k'))$ by \cite[Proposition 3.1.3 (c)]{Suz}. Regarding $\boldsymbol{\Gamma}(K,A)$ as a scheme, one also obtains a canonical morphism $\boldsymbol{\Gamma}(K,A)\to \mathscr{A}_k$, which induces an isomorphism of groups of components by \cite[Proposition 3.4.2 (a)]{Suz}.

We shall now proceed along the same lines as in the case of tori \cite[Section 8]{Suz}. We first construct a \it trace map \rm (cf. \cite[Section 4]{SuzIII}) $$\R\boldsymbol{\Gamma}_{\mathrm{RP}}(K, \nu_{\infty}(1)_K^{\mathrm{RP}}) \to \Q_p/\Z_p$$ as follows: for any $n\in \N,$ $\R\mathbf{\Gamma}(K, \mathbf{G}_{\mathrm{m}}^{(1/p^n)}/\Gm)$ is concentrated in degree 0 by Proposition \ref{cohomologyvanishingprop}. Moreover, the sheaf $\mathbf{\Gamma}(K, \mathbf{G}_{\mathrm{m}}^{(1/p^n)}/\Gm)$ is represented by the Greenberg transform of the Néron model $\mathscr{G}_n$ of $\mathbf{G}_{\mathrm{m}}^{(1/p^n)}/\Gm.$ Since the group of connected components of $\mathscr{G}_{n,k}$ is equal to $\Z/p^n\Z,$ we obtain a canonical map 
$$\R\mathbf{\Gamma}(K, \mathbf{G}_{\mathrm{m}}^{(1/p^n)}/\Gm) = \mathbf{\Gamma}(K, \mathbf{G}_{\mathrm{m}}^{(1/p^n)}/\Gm) \to \Z/p^n\Z.$$ Since $\nu_{\infty}(1)_K^{\mathrm{RP}} = \varinjlim (\mathbf{G}_{\mathrm{m}}^{(1/p^n)}/\Gm)^{\mathrm{RP}}$ (e. g. by \cite[(2.2)]{OSII}), we obtain the trace map by taking colimits on both sides. Note that, if $G$ is a \it unirational \rm wound algebraic group over $K,$ then $\R\boldsymbol{\Gamma}(K,G) = \boldsymbol{\Gamma}(K,G)$ by Proposition \ref{cohomologyvanishingprop} together with the fact that the functors $\R^i(\pi\circ \rho)_\ast G$ are locally of finite presentation for $i\geq 1$ by \cite[Proposition 3.4.3]{Suz}. We shall write $D(K_{\mathrm{RP}})$ and $D(k)$ for the (unbounded) derived categories of sheaves of Abelian groups on $K_{\mathrm{RP}}$ and $\Spec k^{\mathrm{indrat}}_{\proet},$ respectively, and $\R\sHom_{K_{\mathrm{RP}}}(-,-)$ and $\R\sHom_k(-,-)$ for the internal Hom in those categories. 

Now let $G$ be any unirational wound unipotent algebraic group over $K$, and denote by $H$ the $p$-primary finite étale group scheme over $K$ representing $\sHom_{K_{\mathrm{RP}}}(G, \nu_{\infty}(1)_K^{\mathrm{RP}}).$ Let $\vartheta_G$ be the composition 
\begin{align*}
\R\boldsymbol{\Gamma}(K,H) \to &\R\boldsymbol{\Gamma}_{\mathrm{RP}}(K, \R\sHom_{K_{\mathrm{RP}}}(G^{\mathrm{RP}}, \nu_{\infty}(1)_K^{\mathrm{RP}}))\\
 \to &\R\sHom_k(\boldsymbol{\Gamma}_{\mathrm{RP}}(K,G^{\mathrm{RP}}), \boldsymbol{\Gamma}_{\mathrm{RP}}(K, \nu_{\infty}(1)_K^{\mathrm{RP}})) \\
 \to &\R\sHom_k(\boldsymbol{\Gamma}(K,G),  \Q_p/\Z_p) \\
= &\R\sHom_k(\boldsymbol{\Gamma}(K,G),  \Z)[1].
\end{align*}
The last equality uses that $G$ is annihilated by some power of $p.$ 

The map $\vartheta_G$ is compatible with Weil restriction in the following sense: Suppose $K\subseteq F$ is a finite separable extension. Suppose $G_F$ is a unirational wound unipotent algebraic group over $F$ such that $\sHom_{F_{\mathrm{RP}}}(G_F^{\mathrm{RP}}, \nu_{\infty}(1)_F^{\mathrm{RP}})$ is represented by the finite $p$-primary étale group scheme $H_F.$ Let $f\colon \Spec F \to \Spec K$ be the canonical map, which induces a morphism of sites $F_{\mathrm{RP}} \to K_{\mathrm{RP}}$ by base change. We have associated (exact) functors $f_\ast = f_!$ and $f^\ast = f^!$ between the corresponding topoi of sheaves of Abelian groups which have the expected adjunction properties. In particular, we have an adjunction map $f_!\nu_\infty(1)_F^{\mathrm{RP}} = f_!f^! \nu_\infty(1)_K^{\mathrm{RP}} \to \nu_\infty(1)_K^{\mathrm{RP}},$ which induces an isomorphism
$$\sHom_{K_{\mathrm{RP}}}(\Res_{F/K}G_F^{\RP}, \nu_\infty(1)_K^{\mathrm{RP}}) = f_\ast \sHom_{F_{\mathrm{RP}}}(G_F^{\RP}, \nu_\infty(1)_F^{\mathrm{RP}}) = \Res_{F/K}H_F.$$
Moreover, since $f^!$ preserves injectives (or using \cite[Chapter V, Proposition 1.13]{Milne}), we obtain a a canonical isomorphism 
$$f_\ast \R\sHom_{F_{\mathrm{RP}}}(G_F^{\mathrm{RP}}, \nu_\infty(1)_F^{\mathrm{RP}}) \to \R\sHom_{K_{\mathrm{RP}}}(\Res_{F/K}G_F^{\mathrm{RP}}, \nu_{\infty}(1)_K^{\mathrm{RP}}).$$ 
Using the canonical identification $\R\boldsymbol{\Gamma}(F,-) = \R\boldsymbol{\Gamma}(K, \Res_{F/K}-)$ together with the fact that the composition $\R\boldsymbol{\Gamma}_{\mathrm{RP}}(F, \nu_{\infty}(1)_F^{\mathrm{RP}}) \to \R\boldsymbol{\Gamma}_{\mathrm{RP}}(K, \nu_{\infty}(1)_K^{\mathrm{RP}}) \to \Q_p/\Z_p$ is equal to the trace map, one obtains a canonical identification of the maps $\vartheta_{G_F}$ and $\vartheta_{\Res_{F/K}G_F}.$

The following two results are partially special cases of results proven in \cite[Section 4]{SuzIII}. However, the functor used there to formulate the duality theorem is slightly different from the one used here and its definition is significantly more intricate. In particular, it is \it a priori \rm not the right derived functor of some left exact "global sections" functor, although it can be checked to coincide with the one used here for unirational unipotent groups. Were one to apply the results form \cite{SuzIII} directly, one would still have to check that the derived "global sections" functor applied to a unirational group is concentrated in degree 0 in order to recover the duality map above (the full force of which we shall need later). This would require a similar amount of work to verify, as well as assuming significantly more familiarity with the abstract duality formalism. The present approach, while applying only to unirational unipotent groups, is much more self-contained and has the advantage of relying (up to a few standard reductions) only on the duality isomorphism for $\Gm,$ which is completely explicit and has been worked out in great detail in \cite[Section 2.6]{SuzII}. 

\begin{proposition} \rm (Cf. \cite[Proposition 4.4]{SuzIII}) \it \label{dualityprop}
For any unirational wound unipotent algebraic group $G$ over $K,$ the map $\vartheta_G$ constructed above is an isomorphism. 
\end{proposition}

\begin{proof}
We follow the strategy set out in the proof of \cite[Theorem 8.1]{Suz}, which consists of two steps: first one verifies that $\vartheta_G$ is an isomorphism in the Galois-split case, and then one uses Galois descent to deduce the general claim. In the Galois-split case, it suffices to check the claim for $G=\mathbf{G}_{\mathrm{m}}^{(1/p^n)}/\Gm$ for some $n.$ Consider the commutative diagram 

\scriptsize$$\begin{tikzcd}
\R\boldsymbol{\Gamma}(K, \Z) \arrow[r]\arrow[swap]{d}{\vartheta_{\Gm}} & \R\boldsymbol{\Gamma}(K, \Z) \arrow[r]\arrow[swap]{d}{\vartheta_{\Gm}} & \R\boldsymbol{\Gamma}(K, \Z/p^n\Z) \arrow[r]\arrow{d}{\vartheta_{G}} &\R\boldsymbol{\Gamma}(K, \Z)[1]\arrow{d}{\vartheta_{\Gm}[1]}\\
\R\sHom_k(\boldsymbol{\Gamma}(K, \Gm), \Z)\arrow[r] & \R\sHom_k(\boldsymbol{\Gamma}(K, \Gm), \Z)\arrow[r] & \R\sHom_k(\boldsymbol{\Gamma}(K, G),\Z)[1] \arrow[r]& \R\sHom_k(\boldsymbol{\Gamma}(K, \Gm),\Z)[1],
\end{tikzcd}$$\normalsize
where the maps $\vartheta_{\Gm}$ are those from \cite[Section 2.6]{SuzII}. The map $\vartheta_{\Gm}$ is an isomorphism by \cite[Theorem 2.6.1]{SuzII}; hence so is $\vartheta_G$ since the rows are distinguished triangles (in fact, our map $\vartheta_{G}$ is the same as that considered in \cite[Lemma 2.7.2]{SuzII} modulo the canonical isomorphism $\R\rho_{\ast} \boldsymbol{\mu}_{p^n} = (\mathbf{G}_{\mathrm{m}}^{(1/p^n)}/\Gm)^{\mathrm{RP}}[-1]$).

We deduce the general case as in \cite[Proof of Lemma 2.7.3]{SuzII}. More precisely, if $\Gal(L/K)$ is a $p$-group, then $G$ (resp. $H$) admits a filtration with successive quotients $\mathbf{G}_{\mathrm{m}}^{(1/p)}/\Gm$ (resp. $\F_p$), so we obtain the result in this case by dévissage. Hence we may now assume that $L=K(\delta)$ for some $\delta$ prime to $p.$ 
Then the composition of the canonical maps $G\to \Res_{L/K}G_L \to G$ is multiplication by $\delta.$ In particular, $G$ is canonically a direct summand of $\Res_{L/K} G_L.$ The maps $\vartheta_{G_L}$ and $\vartheta_{\Res_{L/K}G_L}$ can be canonically identified, and the induced map coming from the direct summand $G$ is precisely $\vartheta_G.$ Hence the claim follows. 
\end{proof}

\begin{corollary}\label{dualitycor}  \rm (Cf. \cite[Proposition 4.7 (2)]{SuzIII}) \it 
Let $G$ be a unirational wound unipotent algebraic group over $K.$ Let $\mathscr{G}$ be the Néron model of $G$ over $\Og_K$ and let $\Phi_G$ be the group of connected components of $\mathscr{G}_k.$ Moreover, let $H$ be the finite ($p$-primary) étale $K$-group scheme representing the sheaf $\sHom_{K_{\mathrm{RP}}}(G^{\mathrm{RP}}, \nu_{\infty}(1)_K^{\mathrm{RP}}).$ Then there are canonical isomorphisms
$$\Gamma(K,H) = \Hom_{\Z}(\Phi_G, \Q_p/\Z_p) \text{\hspace{.3 in} and \hspace{.3 in}} H^1(K,H) = \Ext^1_k(\boldsymbol{\Gamma}(K,G), \Q_p/\Z_p).$$
\end{corollary}
\begin{proof}
Proposition \ref{dualityprop} provides an isomorphism $\boldsymbol{\Gamma}(K,H)(k) = \Hom_k(\boldsymbol{\Gamma}(K,G), \Q_p/\Z_p).$ Because $k$ is algebraically closed, we have $\boldsymbol{\Gamma}(K,H)(k)=\Gamma(K,H)$ (pro-étale sheafification does not change the value on algebraically closed fields). Moreover, $\boldsymbol{\Gamma}(K, G)$ is represented by the Greenberg transform of $\mathscr{G},$ and the canonical map $\boldsymbol{\Gamma}(K, G) \to \mathscr{G}_k$ is surjective (as a map of sheaves) with connected kernel. This provides a canonical isomorphism 
$$ \Hom_{\Z}(\Phi_G, \Q_p/\Z_p) = \Hom_k (\boldsymbol{\Gamma}(K,G), \Q_p/\Z_p),$$ which shows the first claim. The second assertion follows by applying $\R\Gamma(k,-)$ to $\vartheta_G.$ 
\end{proof}
\subsection{Rationality of motivic zeta functions}
We shall now describe an application of Theorems \ref{jumpsratthm}, \ref{Abelianjumpsratthm}, and \ref{jumpsratthmGEN} to the theory of motivic zeta functions. Let $\Og_K$ be a complete discrete valuation ring with field of fractions $K$ and algebraically closed residue field $k.$ Recall that $K_0(\mathrm{Var}_k)$ is the Grothendieck ring of varieties over $k$ \cite{NS}. Moreover, let $G$ be a wound algebraic group over $K.$ We would like to make a a prediction about general recurring patterns in the behaviour of the Néron lft-model of $G$ under tame base change. Following Halle-Nicaise, we shall define a \it motivic zeta function \rm of $G$ which is an element $Z_G(x)\in K_0(\mathrm{Var}_k)[\![x]\!]$ that encodes such information. In the case where $G$ is a \it unirational \rm wound algebraic group (e. g. an algebraic torus) or an Abelian varietiy with potentially totally multiplicative reduction, we shall show that $Z_G(x)$ is a \it rational \rm function. 

\begin{definition}\label{ordZctamedef}\rm(cf. \cite{HNII}) \it 
Let $G$ be a smooth connected wound algebraic group over $K$ with Néron lft-model $\mathscr{G}$ over $\Og_K.$ For each $d$ prime to $p,$ we shall denote by $\mathscr{G}(d)$ the Néron lft-model of $G(d):=G\times_KK(d)$ (as before).   
\begin{enumerate}[(i)]
\item The \rm order function \it is the map $\mathrm{ord}_G \colon \{d\in \N\colon p\nmid d\} \to \N_0$
given by
$$\ord_G(d):=\ell_{\Og_K(d)}(\mathrm{coker}(\Lie \mathscr{G} \otimes_{\Og_K}\Og_{K(d)} \to \Lie\mathscr{G}(d))).$$
\item Moreover, we define the \rm motivic zeta function \it of $G$ as
$$Z_G(x):=\sum_{p\nmid d} [\mathscr{G}(d)^{\mathrm{qc}}_k]\mathbf{L}^{\mathrm{ord}_G(d)}x^d \in K_0(\mathrm{Var}_k)[\![x]\!];$$ here, $\mathscr{G}(d)^{\mathrm{qc}}_k$ denotes the maximal quasi-compact open subgroup scheme of $\mathscr{G}(d)_k,$ $[X]$ denotes the class of a smooth algebraic variety $X\to \Spec k$ in $K_0(\mathrm{Var}_k),$ and $\mathbf{L}:=[\mathbf{A}^1_k].$
\item Finally, we define the \rm tame base change conductor \it $c_{\mathrm{tame}}(G)$ of $G$ as the sum (in $\Z$) of the jumps in Edixhoven's filtration $\mathscr{F}^\bullet\mathscr{G}_k,$ each weighted with its multiplicity.
\end{enumerate}
\end{definition}
The crucial property of the order function is the following (cf. \cite[Proposition 8.2.2.2]{HN} and Remark \ref{leftconremark}): 
\begin{proposition}\label{jumpsrecurrenceprop}
Let $D$ be a smooth connected wound algebraic group over $K.$ Let $j_{n_1} < ... < j_{n_r}$ be the \rm pairwise distinct \it jumps of $D$ with respective multiplicities $m_1,...,m_r.$ Assume that the jumps of $D$ are \rm rational \it and let $e(D)$ be the smallest natural number such that $e(D) j_{n_i}\in \Z$ for all $i.$ Then there exists a natural number $N$ such that for all $d\geq N$ prime to $p,$ and for all $q\in \mathbf{N}$ such that $d + qe(D)$ is prime to $p,$ we have
$$\ord_D(d+qe(D))=\ord_D(d) + qe(D)c_{\mathrm{tame}}(D).$$
\end{proposition}
\begin{proof}
By Proposition \ref{jumpspropV}, $\mathrm{ord}_D(d)$ is equal to the sum of the $d$-jumps of $D$ (each counted with multiplicity). We put
$$N:=\Big(\min_{1\leq \ell \leq r-1}|j_{n_{\ell+1}}-j_{n_\ell}|\Big)^{-1}.$$
Now choose some $d>N$ prime to $p$ and let $\mathscr{D}(d)$ be the Néron lft-model of $D(d):=D\times_KK(d)$ over $\Og_{K(d)}$ (with $\mathscr{D}:=\mathscr{D}(1)$). Then, for each $l\in \{0,...,d-1\},$ at most one jump of $D$ is contained in $[\frac{l}{d}, \frac{l+1}{d}].$ Given some $i=1,...,r,$ we must distinguish two cases: 
\begin{enumerate}
\item we have $dj_{n_i}\not\in \Z.$ Then we have $\mathscr{F}^{<j_{n_i}}\mathscr{D}_k = F^{\lfloor dj_{n_i}\rfloor}_d\mathscr{D}_k$ and $\mathscr{F}^{>j_{n_i}}\mathscr{D}_k = F^{\lfloor dj_{n_i}\rfloor+1}_d\mathscr{G}_k;$ in particular, $\lfloor dj_{n_i}\rfloor$ is a $d$-jump of $G$ with multiplicity $m_i,$ or
\item $dj_{n_i}\in \Z.$ In this case, we put $f^+_i := \dim \mathscr{F}^{<j_{n_i}}\mathscr{D}_k,$ $f^0_i:= \dim \mathscr{F}^{j_{n_i}}\mathscr{D}_k,$ and $f^-_i := \dim\mathscr{F}^{>j_{n_i}}\mathscr{D}_k.$ Then $dj_{n_i}-1$ is a $d$-jump of multiplicity $f^+_i-f^0_i,$ and $dj_{n_i}$ is a $d$-jump of multiplicity $f^0_i-f^-_i.$
\end{enumerate}
In particular, we obtain 
\begin{align*}\mathrm{ord}_D(d) &= \sum_{\scriptstyle i=1 \atop\scriptstyle dj_{n_i}\not\in \Z}^r m_i \lfloor dj_{n_i}\rfloor +  \sum_{\scriptstyle i=1\atop \scriptstyle dj_{n_i} \in \Z}^r \Big((f^+_i-f ^ 0_i) (dj_{n_i}-1) + (f^0_i - f^-_i)dj_{n_i}\Big)\\
&=\sum_{i=1}^r m_i \lfloor dj_{n_i}\rfloor - \sum_{\scriptstyle i=1\atop\scriptstyle dj_{n_i} \in \Z}^r (f^+_i - f^0_i).
\end{align*}
Now observe that for $q\in \N$ as in the proposition and $i\in \{1,...,r\},$ we have $dj_{n_i}\in \Z$ if and only if $(d+qe(D))j_{n_i}\in \Z.$ In particular, we have
\begin{align*}
\mathrm{ord}_G(d+qe(D)) - \mathrm{ord}_D(d) &= \sum_{i=1}^r m_i \lfloor (d + qe(D))j_{n_i}\rfloor  - \sum_{i=1}^r m_i \lfloor dj_{n_i}\rfloor\\
&= qe(D)\sum_{i=1}^r m_i j_{n_i} \\&= qe(D)c_{\mathrm{tame}}(D).
\end{align*}
\end{proof}

\begin{remark}\label{rmk:anysequence}As a perhaps unexpected consequence of the previous proof (and with the same notation), one sees that for $d$ coprime to $p$ and large enough, the multiset $J_d(D)$ can be read off from the dimension function $\phi_D$ of Remark \ref{leftconremark}: more precisely, setting \[(s_i,s_i')=\begin{cases}(0,m_i) & \text{if } dj_{n_i}\notin\mathbf{Z} \\ (f^+_i-f^0_i,f^0_i-f^-_i) &\text{if } dj_{n_i}\in\mathbf{Z} 
\end{cases},\] we proved that \[J_d(D)=\{(\lfloor dj_{n_1}\rfloor-1;s_1),(\lfloor dj_{n_1}\rfloor;s_1'),\ldots,(\lfloor dj_{n_r}\rfloor-1;s_r),(\lfloor dj_{n_r}\rfloor;s_r')\}\] for $d> N$. 

It is also easy to see that the elements of the previous multiset are listed in strictly increasing order: notice first that the condition $j_{n_{i+1}}-j_{n_i}>1/d$ implies that $\lfloor dj_{n_i}\rfloor<\lfloor dj_{n_{i+1}}\rfloor$. Moreover, by definition $s_{i+1}$ is positive only if $j_{i+1}/d$ is an integer; if so, the inequality $\lfloor dj_{n_i}\rfloor\geq \lfloor dj_{n_{i+1}}\rfloor-1$ would contradict the choice of $d$ (one can also observe that, independently of whether $j_{i+1}/d$ is an integer, the inequality $ j_{n_i} < j_{n_{i+1}}$ implies that $\lfloor dj_{n_i}\rfloor<\lfloor dj_{n_{i+1}}\rfloor-1$ if $d>2N$).
%abbreviate $\lfloor dj_{n_i}\rfloor$ to $a_i$; the condition $j_{i+1}-j_i<1/d$ implies that $a_i<a_{i+1}$. Moreover, as shown in the previous proof the inequality $s_i>0$ only occurs if $a_{i+1}=j_{i+1}/d$; if so, an inequality $a_i\geq a_{i+1}-1$ would imply $dj_i\geq a_i\geq a_{i+1}-1=dj_{i+1}-1,$ contradicting the choice of $d$ (one can also observe that $a_i<a_{i+1}$. 
As a consequence, we obtain the following generalisation of Proposition \ref{jumpslimitprop}: for any sequence $(d_\ell)_\ell$ of integers coprime to $p$ tending to infinity, \[j_i(D)=\lim\limits_{\ell\to\infty}\frac{j_{i,\ell}(D)}{d_\ell}.\]
\end{remark}

In order to establish the rationality of $Z_G(x),$ we shall moreover need information about $[\mathscr{G}(d)_k^{\mathrm{qc}}]\in K_0(\mathrm{Var}_k),$ which includes information about both the number of irreducible components of $\mathscr{G}(d)_k^{\mathrm{qc}},$ as well as the scheme structure on $\mathscr{G}(d)_k^0.$ If $G$ is unirational over $K,$ it turns out that these invariants satisfy a similar recurrence relation. 

\begin{lemma}\label{Phireslem}
Let $K\subseteq F$ be a finite separable extension and let $G_F$ be a smooth connected wound algebraic group over $F$ with Néron lft-model $\mathscr{G}_F$ over $\Og_F.$ Let $\Phi$ be the group of connected components of $\mathscr{G}_{F,k}.$ Then the group of connected components of the special fibre of the Néron lft-model $\Res_{\Og_F/\Og_K} \mathscr{G}_F$ of $\Res_{F/K}G_F$ is canonically isomorphic to $\Phi.$ 
\end{lemma}

\begin{proof}
Let $\iota\colon \Spec k \to \Spec \Og_F$ be the canonical closed immersion. We have an exact sequence $0\to \mathscr{G}_F^0 \to \mathscr{G}_F\to \iota_\ast \Phi \to 0$ of smooth group schemes over $\Og_F.$ Note that $(\iota_\ast \Phi) \times_{\Og_F} (\Og_F\otimes_{\Og_K} k)$ is represented by the constant group scheme $\Phi$ over $\Og_F\otimes_{\Og_K} k.$ Taking Weil restrictions along the inclusion $k\to \Og_F\otimes_{\Og_K} k$ and using \cite[Proposition A.5.9]{CGP}, all we must verify is that $\Res_{(\Og_F\otimes_{\Og_K} k) / k} \Phi$ is represented by the constant group scheme $\Phi.$ But this follows because the inclusion induces a universal homeomorphism on spectra since the extension $K\subseteq F$ is totally ramified. 
\end{proof}

\begin{proposition}\label{K0recurrenceprop}
Let $G$ be a wound unirational algebraic group over $K,$ which fits into an exact sequence $0 \to T \to G \to U \to 0$ with a torus $T$ and $U$ wound unipotent. Let $K\subseteq L$ be a finite Galois extension such that $T$ becomes split and $U$ becomes Galois-split over $L.$ For each $d$ prime to $p,$ let $\mathscr{G}(d)$ be the Néron lft-model of $G(d):=G\times_KK(d)$ over $\Og_{K(d)}.$ Then, if $d$ is prime to $p$ and $q\in \N$ is such that $d+q[L:K]$ is also prime to $p,$ we have
$$[\mathscr{G}(d)_k^{\mathrm{qc}}] = [\mathscr{G}(d+q[L:K])_k^{\mathrm{qc}}]$$ in $K_0(\mathrm{Var}_k).$
\end{proposition}
\begin{proof}
Let $\Gamma:=\Gal(L/K).$ Denote by $K\subseteq K(\delta) \subseteq L$ the maximal tamely ramified subextension of $K \subseteq L$ and by $\Gamma'=\Gal(L/K(\delta))$ the (unique) $p$-Sylow subgroup of $\Gamma.$ Then we can write $\Gamma=\Gamma'\rtimes \mu_{\delta};$ in particular, for each $m\mid \delta,$ there exists a unique subgroup $\Gamma_m\subseteq \Gamma$ of index $m.$ 

Let $\Phi_G(d)$ be the group of connected components of $\mathscr{G}(d)_k$ for $d$ prime to $p;$ we define $\Phi_T(d)$ and $\Phi_U(d)$ analogously. Moreover, put $d':=\mathrm{gcd}(d,\delta).$ By \cite[Proposition 3.7]{HNII} one has an equality $[\mathscr{G}^{\mathrm{qc}}_k]=\#(\Phi_G)_{\mathrm{tors}} [\mathscr{G}^0_k]$ in $K_0(\mathrm{Var}_k)$, which reduces us to showing that $[\mathscr{G}(d)^{\mathrm{0}}_k] = [\mathscr{G}(d+q[L:K])^{\mathrm{0}}_k]$ and $\#\Phi_G(d)_{\mathrm{tors}} = \#\Phi_G(d+q[L:K])_{\mathrm{tors}}$ for all $d,q$ as in the Proposition. Let $g=\dim G$ and let $t_d$ be the dimension of the maximal split subtorus of $G\times_KK(d).$ Then $$[\mathscr{G}(d)^{0}_k] = (\mathbf{L}-1)^{t_d} \mathbf{L}^{g-t_d}$$ for all $d$ by Lemma \ref{toricranklem} and \cite[Lemma 3.1]{N}. Let $X^\ast(T)$ be the character module of $T$ and let $H$ be the $p$-primary étale group scheme dual to $U,$ viewed as a continuous representation of $\Omega:=\Gal(K\sep/K).$ By \cite[Proposition 4.4]{N}, $t_d$ only depends on the induced representation of $\Gamma_{d'}$ on $X^\ast(T).$ Since $d'=(d+q[L:K])',$ the claimed recurrence for $[\mathscr{G}(d)_k^0]$ follows. 

In order to analyse the groups of connected components, consider the distinguished triangle
\begin{align*}\R\Hom_k(\boldsymbol{\Gamma}(K,U), \Z) \to \R\Hom_k(\boldsymbol{\Gamma}(K,G), \Z) &\to \R\Hom_k(\boldsymbol{\Gamma}(K,T), \Z)\\&\to \R\Hom_k(\boldsymbol{\Gamma}(K,U), \Z)[1],\end{align*} 
 which induces a commutative diagram
\scriptsize$$\begin{tikzcd}
\Hom_{\Z}(\Phi_T, \Z) \arrow[r]\arrow[d] & \Ext^1_{\Z}(\Phi_U, \Z) \arrow[r]\arrow[d] & \Ext^1_{\Z}(\Phi_G, \Z) \arrow[r]\arrow[d] & \Ext^1_{\Z}(\Phi_T, \Z)\arrow[r]\arrow[d] & \Ext^2_k(\boldsymbol{\Gamma}(K, U), \Z)\arrow[d]\\
\Hom_{\Z}(\Phi_T(d), \Z) \arrow[r] & \Ext^1_{\Z}(\Phi_U(d), \Z) \arrow[r]& \Ext^1_{\Z}(\Phi_G(d), \Z) \arrow[r] & \Ext^1_{\Z}(\Phi_T(d), \Z) \arrow[r] & \Ext^2_k(\boldsymbol{\Gamma}(K(d), U(d)), \Z)
\end{tikzcd}$$\normalsize
with exact rows. The vertical arrows are induced by the canonical maps $\Res_{K(d)/K} G(d) \to G$ (and similarly for $T$ and $U$). Here we use Lemma \ref{Phireslem} and that the kernel of the canonical map $\boldsymbol{\Gamma}(K, G) \to \Phi_G$ is connected (and similarly for $T$ and $U;$ cf. the argument after the proof of \cite[Theorem 8.1]{Suz}).
As above, it suffices to show that $\#\Phi_G(d)_{\mathrm{tors}} = \# \Phi_G(d')_{\mathrm{tors}}$ for all $d$ prime to $p.$ We may replace $K$ by $K(d')$ and hence suppose that $d$ is moreover prime to $[L:K].$ We shall prove that the middle vertical arrow is then an isomorphism. We put $\Omega_d:=\Gal(K\sep/K(d)).$ By duality (\cite[Theorem 8.1]{Suz} and Corollary \ref{dualitycor} above), we can instead consider the canonical maps
$$H^ i (\Omega, X^\ast(T)) \to H^ i (\Omega_d, X^\ast(T)) \hspace{.2 in} \text{and} \hspace{.2 in} H^ i (\Omega, H) \to H^ i (\Omega_d, H).$$ We shall prove that the maps on the left are isomorphisms for $i=0,1,$ as well as that the map on the right is an isomorphism for $i=0$ and injective for $i=1.$ 

The claim for $i=0$ follows immediately from the fact that the subgroups of invariants of $X^\ast(T)$ and $H$ only depend on the induced representation of $\Gal(L/K)$ and that the induced map $\Gal(L(d)/K(d))\to \Gamma$ is an isomorphism because $d$ is prime to $[L:K].$ The homomorphism on the left is an isomorphism for $i=1$ for the same reason because $H^1(\Gal(L/K), X^\ast(T)) \to H^1(\Omega, X^\ast(T))$ is an isomorphism \cite[Proof of Lemma 3.4]{N}. For the map on the right and $i=1,$ we consider the (continuous) inflation-restriction sequence
$$H^1(\mu_d, H^{\Omega_d}) \to H^1(\Omega, H) \to H^1(\Omega_d,H).$$
The first term vanishes because the orders of $\mu_d$ and $H^{\Omega_d}$ are coprime. This implies that the right-most vertical arrow in the diagram is injective. We have already seen that the first, second, and fourth arrows (from the left) are isomorphisms, so the middle one is an isomorphism as well by the lemma of five homomorphisms. 
\end{proof}

Now we have assembled all the tools necessary to deduce the \it rationality \rm of $Z_G(x)$ if $G$ is a unirational wound algebraic group or a potentially totally degenerate Abelian variety:

\begin{theorem}\label{Zetaratthm}
Let $G$ be a unirational wound algebraic group over $K$ or an Abelian variety with potentially totally multiplicative reduction. Then $Z_G(x)$ is an element of the subring 
$$K_0(\mathrm{Var}_k)\Big[x, \frac{1}{1-\mathbf{L}^ a x^b} \colon (a,b)\in \N_0\times \mathbf{N}, \frac{a}{b} = c_{\mathrm{tame}}(G)\Big]\subseteq K_0(\mathrm{Var}_k)[\![x]\!].$$
In particular, $Z_G(x)$ is a rational function. 
\end{theorem} 
\begin{proof}
(Cf. \cite[Proof of Theorem 8.3.1.2]{HN}) We first treat the case of unirational $G.$ Choose $N \gg 0$ as in Proposition \ref{jumpsrecurrenceprop}. It suffices to show that the series $\sum_{p\nmid d, d> N} [\mathscr{G}(d)^{\mathrm{qc}}_k]\mathbf{L}^{\mathrm{ord}(d)} x^d$ is contained in the subring from the Theorem. Suppose $G$ fits into an exact sequence $0\to T\to G\to U \to 0$ with $T$ a torus and $U$ unipotent. Let $K\subseteq L$ be a finite Galois extension of $K$ such that $U$ and $T$ are (Galois-)split over $L,$ and let $e'(G):=\mathrm{lcm}(p, e(G), [L:K]),$ where $e(G)$ is the smallest positive integer such that $e(G)j\in \Z$ for all jumps $j$ of $G$ (which exists by Theorem \ref{jumpsratthmGEN}). For each $\alpha\in \{N+1,..., N+e'(G)\}$ prime to $p,$ we define an auxiliary series
$$Z^{(\alpha)}_G(x):= \sum_{\scriptstyle{d\equiv \alpha\!\!\!\!\mod e'(G)} \atop\scriptstyle{d> N}} [\mathscr{G}(d)^{\mathrm{qc}}_k]\mathbf{L}^{\mathrm{ord}_G(d)} x^d;$$
it is then sufficient to show that each $Z^{(\alpha)}_G(x)$ is contained in the subring. However, by Propositions \ref{jumpsrecurrenceprop} and \ref{K0recurrenceprop}, we have
\begin{align*}
Z_G^{(\alpha)}(x) &= \sum_{\lambda=0} ^\infty [\mathscr{G}(\alpha + \lambda e'(G))^{\mathrm{qc}}_k]\mathbf{L}^{\mathrm{ord}_G(\alpha + \lambda e'(G))} x^{\alpha + \lambda e'(G)}\\
&=[\mathscr{G}(\alpha)^{\mathrm{qc}}_k]\mathbf{L}^{\ord_G(\alpha)}x^{\alpha}\sum_{\lambda=0} ^\infty \mathbf{L}^{\lambda e'(G) c_{\mathrm{tame}}(G)} x^{\lambda e'(G)} \\
&=[\mathscr{G}(\alpha)^{\mathrm{qc}}_k]\mathbf{L}^{\ord_G(\alpha)}x^{\alpha}\ \cdot \frac{1}{1-\mathbf{L}^{e'(G) c_{\mathrm{tame}}(G)}x^{e'(G)}}.
\end{align*}
Now let $G$ be an Abelian variety with potentially totally multiplicative reduction. Let $T$ be an algebraic torus over $K$ uniformising $G$ \cite[Section 3.3.7]{HN}. For each $d$ prime to $p,$ let $\mathscr{G}(d)$ and $\mathscr{T}(d)$ be the Néron lft-models of $G\times_KK(d)$ and $T\times_KK(d)$ over $\Og_{K(d)},$ respectively. As in the proof of \cite[Proposition 6.2.3.2]{HN}, we obtain a canonical étale morphism $\mathscr{T}(d)^0_k \to \mathscr{G}(d)^0_k$ for all such $d;$ hence $[\mathscr{T}(d)^0_k]=[\mathscr{G}(d)^0_k]$ in $K_0(\mathrm{Var}_k)$ and, moreover, $\ord_G(d) = \ord_T(d)$ for all $d$ prime to $p$ (as well as $c_{\mathrm{tame}}(G) = c_{\mathrm{tame}}(T)$). Let $K\subseteq L$ be a finite separable extension of $K$ splitting $T$ and let $K\subseteq K(\delta)\subseteq L$ be the maximal tamely ramified subextension of $K\subseteq L.$ By Theorem \ref{Abelianjumpsratthm}, we can define $e'(G)$ exactly as in the unirational case. Once again, it suffices to show that, for all $\alpha\in \{N+1,..., N+e'(G)\}$ prime to $p,$ the auxiliary series $Z_G^{(\alpha)}(x)$ is contained in the subring. Let $\Phi_G(d)$ denote the group of connected components of $\mathscr{G}(d)_k.$ As in the unirational case, we compute
$$Z^{(\alpha)}_G(x) = [\mathscr{G}(\alpha)^0_k]\mathbf{L}^{\ord_G(\alpha)} x^{\alpha} \sum_{\lambda=0}^\infty \#\Phi_G(\alpha+\lambda e'(G)) \mathbf{L}^{\lambda e'(G) c_{\mathrm{tame}}(G)}x^{\lambda e'(G)}.$$ Substituting $y:=\mathbf{L}^{e'(G)c_{\mathrm{tame}}(G)}x^{e'(G)},$ it suffices to show that the series $\sum_{\lambda=0}^{\infty} \#\Phi_G(\alpha+\lambda e'(G)) y^{\lambda} \in \Z[\![y]\!]$ is contained in $\Z[y, (1-y)^{-1}],$ which works (essentially) as in the proof of \cite[Theorem 6.5]{HNComp}. Indeed, put $d':=\mathrm{gcd}(d,\delta)$ for $d$ prime to $p.$ For any $\alpha\in \{N+1,..., N+e'(G)\}$ prime to $p,$ we have $$\#\Phi_G(\alpha + \lambda e'(G)) = \Big(\frac{\alpha}{\alpha'} + \frac{e'(G)}{\alpha'} \lambda\Big)^{t(G\times_KK(\alpha'))} \#\Phi_{G}(\alpha'),$$ where $t(G\times_KK(d))$ denotes the dimension of the maximal torus of $\mathscr{G}(d)^0_k$ \cite[Theorem 5.7]{HNComp}. The series $\sum_\lambda (\frac{\alpha}{\alpha'} + \frac{e'(G)} {\alpha'} \lambda)^{t(G\times_KK(\alpha'))} y^\lambda$ is contained in $\Z[y, (1-y)^{-1}]$ by \cite[Lemma 6.2]{HNComp}, so the claim follows. 
\end{proof}

\subsection{Additivity of the tame base change conductor}
The tame base change conductor remains a mysterious invariant about which very little is known. For a smooth connected wound algebraic group $G$ over $K$, we have $c_{\mathrm{tame}}(G)=0$ if and only if the identity component of the Néron lft-model of $G$ commutes with tame base change. For an algebraic torus $T,$ Halle-Nicaise have shown the remarkable formula $c_{\mathrm{tame}}(T) = \frac{1}{2}u(T),$ where $u(T)$ is the unipotent rank of the special fibre of the identity component of the Néron lft-model of $T$ \cite[Proposition 7.1.1.3]{HN}. Note that this formula implies that the tame base change conductor is isogeny invariant for algebraic tori (which is not the case generally for the jumps). The proof relies on deep results due to Chai-Yu-de Shalit relating Chai's base change conductor of $T$ to the Artin conductor of its rational character module $X^{\ast}(T)\otimes_{\Z} \Q.$ The same formula as for tori can be deduced for potentially totally degenerate Abelian varieties using rigid uniformisation \cite[Proposition 6.2.3.2]{HN}. For an elliptic curve $E$ over $K,$ the tame base change conductor only depends on the Kodaira symbol of the reduction of $E$ and the explicit value is known in each case \cite[Table 6.1 on p. 100]{HN}.

In this paragraph, we shall show that $c_{\mathrm{tame}}(-)$ satisfies an analogue of Chai's conjecture:

\begin{theorem}\label{ctameadditivethm}
Let $0\to G\to B \to A \to 0$ be an exact sequence of smooth connected wound algebraic groups over $K$ with $G$ \rm unirational. \it Then, for each $d$ prime to $p,$ we have $\ord_B(d) = \ord_G(d) + \ord_A(d).$ In particular, 
$$c_{\mathrm{tame}}(B) = c_{\mathrm{tame}}(G) + c_{\mathrm{tame}}(A).$$
\end{theorem} 
\begin{proof}
For $d$ prime to $p,$ let $\mathscr{G}(d),$ $\mathscr{B}(d),$ and $\mathscr{A}(d)$ be the Néron lft-models of $G\times_KK(d),$ $B\times_KK(d),$ and $A\times_KK(d),$ respectively. Put $\mathscr{G}:=\mathscr{G}(1)$ (and similarly for $\mathscr{B}$ and $\mathscr{A}$).  Denote by $\mathscr{G}^\bullet$ the complex $0 \to \mathscr{G}^0 \to \mathscr{B}^0 \to \mathscr{A}^0 \to 0$ (with $\mathscr{G}^0$ in degree 1) and define $\mathscr{G}(d)^\bullet$ analogously. Denote by $\chi_{\mathrm{points}}(\mathscr{G}^\bullet)$ the quantity from \cite[Definition 2.6]{OS}. Finally, let $\Lie \mathscr{G}^\bullet$ be the complex $0 \to \Lie\mathscr{G} \to \Lie\mathscr{B} \to \Lie\mathscr{A} \to 0$ of $\Og_K$-modules and define $\Lie \mathscr{G}(d)^\bullet$ similarly. Now let $d$ be prime to $p.$ Then we have a term-wise exact sequence of complexes
\begin{align}0 \to \Lie\mathscr{G}^\bullet \otimes_{\Og_K}\Og_{K(d)} \to \Lie \mathscr{G}(d)^\bullet \to Q^\bullet \to 0,\label{Complex}\end{align} where $Q^\bullet$ is the obvious cokernel. For a bounded complex $A^\bullet$ of finitely generated $\Og_K$-modules which is generically exact, we let $\chi(A^\bullet)$ be the cohomological Euler characteristic of $A^\bullet$ \cite{OS}, i. e. the quantity $\sum_{i\in\Z} (-1)^i \ell_{\Og_K}(H^i(A^\bullet)).$ It follows exactly as in the second argument in the proof of \cite[Theorem 2.11]{OS} that $\chi_{\mathrm{points}}(\mathscr{G}^\bullet) = 0 = \chi_{\mathrm{points}}(\mathscr{G}(d)^\bullet)$ (replacing the reference to \cite[Lemma 4.3]{Chai} by a reference to Proposition \ref{cohomologyvanishingprop} above) and hence $\chi(\Lie\mathscr{G}^\bullet \otimes_{\Og_K}\Og_{K(d)}) = 0 = \chi(\Lie \mathscr{G}(d)^\bullet)$ by \cite[Theorem 2.7]{OS}. Using the exact sequence (\ref{Complex}), we deduce that $\chi(Q^\bullet) = 0.$ If $B^\bullet$ is a bounded complex of finitely generated \it torsion \rm $\Og_K$-modules, one easily shows by induction that 
$$\chi(B^\bullet) = \sum_{i\in \Z} (-1)^i\ell_{\Og_K}(B^i).$$
In particular, we deduce from the exact sequence (\ref{Complex}) that
$$0 = \chi(Q^\bullet) = -\ord_G(d)+\ord_B(d)-\ord_A(d),$$ using the definition of the order function. If $(d_\ell)_{\ell}$ is a grid sequence, we have $c_{\mathrm{tame}}(G) = \lim_{\ell\to \infty}\frac{1}{d_{\ell}}\ord_G(d_\ell)$ (and similarly for $B$ and $A$) by Proposition \ref{jumpspropV} (ii), so the claim follows. 
\end{proof}

\begin{remark}\label{ctamerem} There is some reason to believe that the formula $c_{\mathrm{tame}}(G) = \frac{1}{2}u(G)$ might hold for all unirational algebraic groups $G,$ but this is far from clear. By the additivity property just observed together with \cite[Proposition 7.1.1.3]{HN}, it would be enough to establish this claim for unirational \it unipotent \rm wound connected $G,$ in which case it reduces to $c_{\mathrm{tame}}(G) = \frac{1}{2}\dim G.$ This can be explicitly checked for the algebraic groups $\mathbf{G}_{\mathrm{m}}^{(1/p^n)}/\Gm$ for $n\in \N,$ but beyond that, very little is known. Halle-Nicaise's proof in the case of tori cannot be generalised directly because the definition of Chai's base change conductor \it would not make sense \rm for unipotent $G.$ \end{remark}

\subsection{A question of Oesterlé}
In this subsection, we let $K$ be any field of characteristic $p>0$ such that $[K:K^p]=p.$ The description of unirational wound unipotent algebraic groups over $K$ from Proposition \ref{classprop} applies without change under these more general assumptions on $K.$ In the special case where $K$ is the function field of a smooth proper geometrically integral curve $C$ over a finite field, Oesterlé \cite[Notes sur le chapitre VI on p. 80]{Oes} asked the following questions: Suppose $G$ is a smooth connected wound unipotent algebraic group over $K$ such that $G(K)$ is infinite.
\begin{enumerate}[(i)]
\item Does $G$ contain a \it unirational \rm (necessarily wound unipotent) subgroup?
\item Does $G$ contain a subgroup of the form $T^{(1/p)}/T$ for some algebraic torus $T$ over $K$? 
\end{enumerate}
Question (i) was answered affirmatively by Suzuki and the first-named author in \cite[p. 3]{OSII} (it had already been observed by Oesterlé that an affirmative answer would follow from \cite[Conjecture 1.1 (2)]{OSII}), and independently by Rosengarten (unpublished) using different methods. It will turn out that the answer to question (ii) is affirmative only \it up to finite separable extension of $K$ and relative perfection; \rm without allowing finite separable extensions, one obtains a slightly weaker resolution property after relative perfection. We shall need a version of Cartier duality for relative perfections of algebraic tori, for which some technical preparation is necessary. Mainly, we need to understand how certain Ext groups change upon relative perfection.

Now we define
$$N_j := \ker(\mathbf{G}_{\mathrm{a}}^{(1/p^{j+1})} \to \Ga) \hspace{.3 in} \text{and} \hspace{.3 in} M_j:= \ker(\mathbf{G}_{\mathrm{m}}^{(1/p^{j+1})} \to \Gm)$$ for $j\in \N.$ The group $N_0$ has already been studied in \cite[Section 9]{OSII}.

\begin{lemma}\label{ExtlemIII} The following claims hold:
\begin{enumerate}[(i)]
\item The groups $M_0$ and $N_0$ are reduced and annihilated by $p.$ In particular, $\Hom_K(M_j, \Gm) =0= \Hom_K(N_j, \Gm)$ for all $j.$ 
\item We have $\Ext^1_K(M_j, \Gm)=0=\Ext^1_K(N_j, \Gm)$ for all $j.$
\end{enumerate}
\end{lemma}
\begin{proof}
(i): We have $N_0 = \boldsymbol{\alpha}_p^{(1/p)}$ and $M_0 = \boldsymbol{\mu}_p^{(1/p)}.$ To see that those algebraic groups are reduced, we observe that $M_0\cong N_0$ \it as schemes; \rm hence both are described by the equation $F(x_0,...,x_{p-1}):=x_0^p + tx_1^p + ... + t^{p-1}x_{p-1}^p=0$ inside $\mathbf{A}_K^p,$ where $t$ is a $p$-basis of $K$ (i. e. an element not contained in $K^p$). The topological space underlying $N_0$ is homeomorphic to an affine space (hence irreducible); this can be seen by base change to $K(t^{1/p}).$ If $F$ were reducible, its factorisation into irreducible factors would therefore have to be of the form $F=\lambda P^n$ for some irreducible homogeneous polynomial $P\in K[x_0,...,x_{p-1}]$ and $\lambda\in K.$ Since $p$ is prime, either $n=1$ or $n=p.$ However, the latter case is impossible because then $P$ would have to be a \it linear \rm polynomial; hence $M_0(K\sep)\cong N_0(K\sep)$ would be infinite. Since $M_0$ and $N_0$ are $p$-torsion, any $K$-homomorphism $M_0\to \Gm$ (resp. $N_0 \to \Gm$) has to factor through $\boldsymbol{\mu}_p,$ and since $M_0$ and $N_0$ are reduced, the homomorphism vanishes. Now consider the exact sequence $0 \to \Gm \to \mathbf{G}_{\mathrm{m}}^{(1/p^j)} \to \mathbf{G}_{\mathrm{m}}^{(1/p^j)}/\Gm \to 0,$ the last non-trivial term of which is unipotent. Using the snake lemma, induction on $j,$ and \cite[Proposition 9.2]{OSII}, we find that $N_j$ and $M_j$ admit finite filtrations with successive quotients isomorphic to $N_0$ or $M_0,$ so the claim follows.\\
(ii) It suffices to show that $\Ext^1_K(M_0, \Gm)=0=\Ext^1_K(N_0, \Gm).$ The second equality is \cite[Proposition 10.1]{OSII}. For the first equality, note that we have an exact sequence 
$$0 \to \boldsymbol{\mu}_p \to M_0 \to \mathbf{G}_{\mathrm{m}}^{(1/p)}/\Gm \to 0,$$ which induces an exact sequence
$$0 \to \Hom_K(\boldsymbol{\mu}_p, \Gm) \to \Ext^1_K(\mathbf{G}_{\mathrm{m}}^{(1/p)}/\Gm, \Gm) \to \Ext^1_K(M_0, \Gm) \to \Ext^1_K(\boldsymbol{\mu}_p, \Gm).$$
Since we have 
\begin{align*}
\Ext^1_K(\mathbf{G}_{\mathrm{m}}^{(1/p)}/\Gm, \Gm) &= \Ext^1_{K_{\mathrm{RP}}}((\mathbf{G}_{\mathrm{m}}^{(1/p)}/\Gm)^{\mathrm{RP}}, \mathbf{G}_{\mathrm{m}}^{\mathrm{RP}}) \\
&= \Hom_{K_{\mathrm{RP}}}((\mathbf{G}_{\mathrm{m}}^{(1/p)}/\Gm)^{\mathrm{RP}}, \nu_{\infty}(1)_K^{\mathrm{RP}}) = \Z/p\Z
\end{align*}
(using \cite[(2.2) and Proposition 10.3]{OSII}), the second map in the exact sequence must be an isomorphism; hence $\Ext^1_K(M_0, \Gm) \subseteq \Ext^1_K(\boldsymbol{\mu}_p, \Gm).$ By Cartier duality, every element of $\Ext^1_K(\boldsymbol{\mu}_p, \Gm)$ becomes trivial after a finite separable extension of $K;$ therefore the same is true for elements of $\Ext^1_K(M_0, \Gm).$ However, because $\Hom_K(M_0, \Gm)=0$ by (i), this implies that $\Ext^1_K(M_0, \Gm)=0$ by Galois descent.
\end{proof}
Recall that $\rho_\ast \colon \mathrm{Ab}(\mathrm{Sch}/K) \to \mathrm{Ab}(K_{\RP})$ is induced by the forgetful functor.
\begin{proposition}
Let $T_1, T_2$ be algebraic tori over $K.$ Then the canonical map 
$$\rho_\ast \sHom_K(T_1, T_2) \to \sHom_{K_{\mathrm{RP}}}(T_1^{\mathrm{RP}}, T_2^{\mathrm{RP}})$$ is an isomorphism. In particular, the Cartier dual of $T_1$ (which is already contained in $K_{\mathrm{RP}}$ since it is étale over $K$) represents $\sHom_{K_{\mathrm{RP}}}(T_1^{\mathrm{RP}}, \mathbf{G}_{\mathrm{m}}^{\mathrm{RP}}).$ \label{RPCartierprop}
\end{proposition}
\begin{proof}
Since both $T_i$ are split by an étale cover of $\Spec K,$ we may assume that $T_i$ are split for $i=1,2.$ Then the Proposition reduces to the claim that the canonical map $\rho_\ast\sEnd_K(\Gm) \to \sEnd_{K_{\mathrm{RP}}}(\mathbf{G}_{\mathrm{m}}^{\mathrm{RP}})$ is an isomorphism. Let $X$ be an affine relatively perfect scheme over $K.$ By the universal property of relative perfection and \cite[Tag 01ZC]{Stacks}, we have $$\Hom_X(\mathbf{G}_{\mathrm{m}}^{\mathrm{RP}}, \mathbf{G}_{\mathrm{m}}^{\mathrm{RP}}) = \Hom_X(\mathbf{G}_{\mathrm{m}}^{\mathrm{RP}}, \Gm ) = \varinjlim \Hom_X(\mathbf{G}_{\mathrm{m}}^{(1/p^n)}, \Gm).$$ All that remains to be shown is that $\Hom_X(M_n, \Gm)=0$ for all $n\in \N.$ We already know that $M_n$ is a repeated extension of $M_0$ and $N_0$ (cf. the proof of Lemma \ref{ExtlemIII} (ii)). Hence it suffices to show that $\Hom_X(N_0, \Gm) = 0 = \Hom_X(M_0, \Gm).$ Let $R:=\Gamma(X,\Og_X).$ The morphism 
$$\bigsqcup_{\mathfrak{p}\subseteq R} \Spec R_{\mathfrak{p}} \to X$$ (with $\mathfrak{p}\subseteq R$ running through \it minimal \rm prime ideals) is scheme-theoretically dominant\footnote{Some care must be taken here since this morphism need not be quasi-compact. A morphism $f\colon X\to Y$ of \it reduced \rm schemes is dominant if and only if it is scheme-theoretically dominant, and this happens if and only if the map $\Og_Y \to f_\ast\Og_X$ is injective.} because $X$ is reduced \cite[Tag 00EW]{Stacks}. Observe that $N_0\times_KX$ is flat and relatively perfect over the reduced scheme $N_0,$ and hence reduced. Since $N_0$ is flat and of finite presentation over $K$ (hence universally open), the morphism $\bigsqcup_{\mathfrak{p}} N_0\times_K R_{\mathfrak{p}} \to N_0\times_KX$ is scheme-theoretically dominant as well \cite[Tag 0H8F]{Stacks}. Since $R_{\mathfrak{p}}$ is a relatively perfect field extension of $K,$ we have reduced the claim to showing that $\Hom_K(N_0, \Gm) = 0 = \Hom_K(M_0, \Gm),$ which is Lemma \ref{ExtlemIII} (i). 
\end{proof}

\begin{proposition}
Let $G$ be a unirational wound unipotent algebraic group over $K.$ Then the following hold:
\begin{enumerate}[(i)]
\item There exists an infinitesimal group scheme $\boldsymbol{\mu}$ of multiplicative type over $K$ and an isomorphism $\R^1\rho_\ast\boldsymbol{\mu} \cong G^{\mathrm{RP}},$ and 
\item there exist algebraic tori $T_1$ and $T_2$ over $K$ together with an exact sequence 
$$0 \to T_1^{\mathrm{RP}} \to T_2^{\mathrm{RP}} \to G^{\mathrm{RP}} \to 0.$$
\end{enumerate}
\end{proposition}
\begin{proof}
Let $H$ be the $p$-primary étale group scheme representing $\sHom_{K_{\mathrm{RP}}}(G^{\mathrm{RP}}, \nu_\infty(1)_K^{\mathrm{RP}})$ and let $\boldsymbol{\mu}$ be the Cartier dual of $H.$ Then $\boldsymbol{\mu}$ is of multiplicative type and infinitesimal by construction. Now consider Bégueri's \it résolution standard \rm $0 \to \boldsymbol{\mu} \to T_1 \to T_2 \to 0,$ where $T_1:=\Res_{H/K}\Gm$ \cite[Proposition 2.2.1]{Beg}. Since $\boldsymbol{\mu}$ is totally non-smooth, we have $\boldsymbol{\mu}^{\mathrm{RP}}=0$ by \cite[Proposition 8.4]{BS}; hence we obtain an exact sequence 
$$0 \to T_1^{\mathrm{RP}} \to T_2^{\mathrm{RP}} \to \R^1\rho_\ast\boldsymbol{\mu} \to 0$$
(surjectivity of the second-to-last map follows from Lemma \ref{RPcohomlem}). Proposition \ref{RPCartierprop} then shows that $\R^1\rho_\ast\boldsymbol{\mu}$ becomes isomorphic to a finite direct sum of $K$-group schemes of the form $(\mathbf{G}_{\mathrm{m}}^{(1/p^n)}/\Gm)^{\mathrm{RP}}$ over a finite separable extension of $K.$ In particular, $\sExt^1_{K_{\mathrm{RP}}}(\R^1\rho_\ast\boldsymbol{\mu}, \mathbf{G}_{\mathrm{m}}^{\mathrm{RP}}) = \sHom_{K_{\mathrm{RP}}}(\R^1\rho_\ast\boldsymbol{\mu},\nu_{\infty}(1)_K^{\mathrm{RP}})$ is étale over $K$ by Proposition \ref{classprop}. Applying the functor $\sHom_{K_{\mathrm{RP}}}(-,\mathbf{G}_{\mathrm{m}}^{\mathrm{RP}})$ then yields an exact sequence
$$0 \to \sHom_{K_{\mathrm{RP}}}(T_2^{\mathrm{RP}},\mathbf{G}_{\mathrm{m}}^{\mathrm{RP}}) \to \sHom_{K_{\mathrm{RP}}}(T_1^{\mathrm{RP}},\mathbf{G}_{\mathrm{m}}^{\mathrm{RP}}) \to \sExt^1_{K_{\mathrm{RP}}}(\R^1\rho_\ast\boldsymbol{\mu},\mathbf{G}_{\mathrm{m}}^{\mathrm{RP}}).$$
The last map is surjective because because 
$$\Ext^1_{F_{\mathrm{RP}}}(\mathbf{G}_{\mathrm{m}}^{\mathrm{RP}}, \mathbf{G}_{\mathrm{m}}^{\mathrm{RP}}) = \varinjlim \Ext^1_{F_{\mathrm{RP}}}(\mathbf{G}_{\mathrm{m}}^{(1/p^j)}, \Gm) = 0$$
for all separable algebraic extensions $F$ of $K$ by \cite[Proposition 8.10]{BS} as well as Lemma \ref{ExtlemIII} (ii). By Proposition \ref{classprop}, this implies that $G\cong \R^1 \rho_\ast \boldsymbol{\mu},$ which shows both claims.
\end{proof}

In particular, up to relative perfection, every unirational wound unipotent algebraic group admits a two-step resolution by algebraic tori. However, this does not in general guarantee the existence of subgroups of the form predicetd by Question (ii). Part (i) of the following result has been independently observed by Rosengarten (unpublished), using the structure theory of unipotent groups from \cite{Ros, RosII}.

\begin{proposition} 
Let $K$ be a field of characteristic $p>0$ such that $[K:K^p]=p.$ 
\begin{enumerate}[(i)]
\item There exists a unirational unipotent algebraic group $G$ over $K$ which does not contain an algebraic group of the form $T^{(1/p)}/T$ for any algebraic torus $T\not=0.$
\item Suppose that $\mathrm{char}\, K \geq 5$ and that $K$ contains an element $x$ such that the polynomial $t^{p-1}-x\in K[t]$ is irreducible (this is always satisfied if $K$ admits a discrete valuation). Then there exists a unirational unipotent algebraic group $G$ over $K$ such that $G^{\mathrm{RP}}$ does not contain a subgroup scheme of the form $T^{\mathrm{RP}}\otimes_{\Z}\F_p$ for any algebraic torus $T\not=0.$
\end{enumerate}
\end{proposition}
\begin{proof}
(i) Put $G:=(\mathbf{G}_{\mathrm{m}}^{(1/p)}/\Gm)^{(1/p)}.$ This group is unirational because it has the same relative perfection as $\mathbf{G}_{\mathrm{m}}^{(1/p)}/\Gm,$ and unirationality only depends upon the relative perfection by Proposition \ref{classprop}. We have an exact sequence $0 \to N_0^{p-1} \to G \to \mathbf{G}_{\mathrm{m}}^{(1/p)}/\Gm \to 0$ by \cite[Proposition 9.2]{OSII}. Were $G$ to contain $T^{(1/p)}/T$ for a torus $T,$ it would have to contain $\mathbf{G}_{\mathrm{m}}^{(1/p)}/\Gm$ after a finite separable extension. Since the endomorphism ring of $\mathbf{G}_{\mathrm{m}}^{(1/p)}/\Gm$ is a field (again by Proposition \ref{classprop}), either the sequence splits (which is impossible since $N_0$ is not smooth), or the induced endomorphism of $\mathbf{G}_{\mathrm{m}}^{(1/p)}/\Gm$ vanishes. In the latter case, the closed immersion $\mathbf{G}_{\mathrm{m}}^{(1/p)}/\Gm \to G$ factors through $N_0^{p-1},$ which is impossible because $N_0(K\sep)=0.$ \\
(ii) Choose $x\in K$ as in (ii). Since $K$ contains a primitive $(p-1)$-st root of unity, we have $\Gal(K(x^{1/(p-1)})/K) = \F_p^\times.$ Let $H$ be the étale $K$-group scheme corresponding to a continuous character $\Gal(K\sep/K)\to \F_p^\times$ which does \it not \rm factor through $\Z^ \times = \{\pm 1\} \subseteq \F_p^{\times},$ and let $G$ be the corresponding Galois twist of $\mathbf{G}_{\mathrm{m}}^{(1/p)}/\Gm.$ 
 Both $G^{\mathrm{RP}}$ and $T^{\mathrm{RP}}\otimes_{\Z} \F_p$ become isomorphic to direct sums of $\mathbf{G}_{\mathrm{m}}^{\mathrm{RP}}\otimes_{\Z} \F_p$ over some finite separable extension of $K,$ so Proposition \ref{dualityprop} shows that any injective homomorphism $T^{\mathrm{RP}}\otimes_{\Z} \F_p \to G^{\mathrm{RP}}$ is an isomorphism, contradicting our choice of $G.$
\end{proof}

\section{Jumps of some algebraic groups in the case of imperfect residue field}

 %\begin{notation}Recall some notation from the introduction: $\Og_K$ denotes a strictly henselian discrete valuation ring with uniformizer $\pi$, residue field $k$ of characteristic exponent $p$, and field of fractions $K$. All field extensions of $K$ will be considered inside some fixed separable closure $K^{sep}$ of $K$. The ring of integers of a valued field $F$ will be denoted by $\Og_F$, its residue field by $k_F$ (with the usually exception of fields with a primed symbol). 
 
This section is devoted to the study of the jumps (regarded now as \textit{real numbers}) of some algebraic $K$-groups in the case when the residue field $k$ of $K$ is possibly {imperfect}. We let $\Og_K$ be a strictly Henselian discrete valuation ring with field of fractions $K$ and arbitrary residue field $k.$

In the first two subsections we prove the rationality of the jumps of $K$-tori $T$ which are direct factors of $K$-rational varieties (Definition \ref{def:Kratvar}). For these, we shall obtain an analogue of Theorem \ref{jumpsratthm} (i), with the degree of a splitting extension $L/K$ replaced by its ramification index; the inertia index instead controls the possible multiplicities.

The main obstacle one encounters is that, as the methods of the previous sections cannot be applied if $k\neq \bar k$, one cannot use an arbitrary smooth, surjective morphism $f$ with connected kernel from a product $P$ of induced tori to $T$ to study the jumps of $T$. Instead, we seek to apply Lemma \ref{universallyexactlem}, and in order to do so, one needs a short exact sequence \[0\to F\to P\to T\to 0\] of tori, with $P$ as before, whose base-change to $K(d)$ induces an exact sequence of Néron lft-models for sufficiently many integers $d$ coprime to $p$. Proposition \ref{Chaiexactn} shows that this happens if $F$ is an \textit{invertible torus}, and a theorem by Colliot-Thélène and Sansuc then implies that the so-called \textit{flasque resolution} of $T$ meets our needs, if $T$ is a direct factor of a $K$-rational variety.
One is thus reduced to the case of induced tori along some extension $M/K$ (which we treat first); a similar argument reduces this to the case when $M/K$ is Galois, which hinges on the computational Proposition \ref{prop:coker}. 

By a similar computation, in the last subsection we show that the jumps of the wound unipotent group $\boldsymbol{\nu}_1(r)_K$ are rational, and the denominators of the non-zero ones are equal to $p$.

\subsection{Jumps of induced tori}

Following \cite[Tag 09E4]{Stacks}, we call an extension of discrete valuation fields \textit{weakly unramified} if its ramification index is one.

In the following, all the extensions of $K$ will be regarded inside some fixed separable closure $K\sep$. 

\begin{proposition}\label{prop:coker}Let $L/K$ be a finite Galois extension of strictly henselian discretely valued fields; denote by $e$ and $f$ be the ramification and inertia indexes of $L/K$, respectively, and let $d$ be a positive integer such that $d\equiv 1 \bmod p[L:K].$ If $K(d)/K$ a totally ramified extension of $K$ of degree $d$, then the cokernel of the natural map \[\Og_{L}\otimes_{\Og_K}\Og_{K(d)}\to\Og_{L\otimes_KK(d)}\] is isomorphic to \[\bigoplus\limits_{\nu=0}^{e-1}\left(\Og_{K(d)}/\ffm_d^{\nu\frac{d-1}{e}}\right)^{\oplus f}\] as an $\Og_{K(d)}$-module. \end{proposition}
\begin{proof}
Let $\pi_{K(d)} \in \Og_{K(d)}$ be a uniformiser. Denote the ramification and inertia indices of a finite extension $F/K$ by $e_{F/K}$ and $f_{F/K}$, respectively. We shall prove the following slightly stronger \\
\\
\rm Claim: \normalfont \it There exists an $\Og_{K(d)}$-basis $(e_{ij})_{ij}$ (with $0\leq i \leq e_{L/K}-1$ and $0\leq j \leq f_{L/K}-1)$ of $\Og_{L(d)}$ such that the elements $\pi_{K(d)}^{i (d-1)/e_{L/K}} e_{ij}$ are contained in $\Og_L$ and form an $\Og_K$-basis of $\Og_L.$\rm \\
\\
The Proposition then follows immediately from this claim. We shall assume $L\not=K$ and first treat the case where $K\subseteq L$ is weakly unramified or totally ramified (in which case it holds even without the Galois condition). In the first case, note that the canonical map $\Og_L\otimes_{\Og_K}\Og_{K(d)} \to \Og_{L(d)}$ is an isomorphism (since the extension $K\subseteq K(d)$ is totally ramified). Hence we can choose any $\Og_K$-basis of $\Og_L.$ In the latter case, choose a uniformiser $\pi_L$ of $\Og_L.$ A simple calculation involving valuations shows that the element $\pi_{L(d)}:=\pi_{K(d)}^{-(d-1)/e_{L/K}}\pi_L$ is a uniformiser of $\Og_{L(d)}.$ Since we have $\Og_L=\Og_K[\pi_L]$ and $\Og_{L(d)}=\Og_{K(d)}[\pi_{L(d)}],$ the claim follows in this case.

Now choose a subextension $K\subseteq F \subseteq L$ with $K\not=F$ such that $K\subseteq F$ is either totally ramified or weakly unramified. This is certainly possible if $K\subseteq L$ contains a non-trivial tamely ramified subextension. Otherwise, $\Gal(L/K)$ is a $p$-group and hence contains a subgroup of index $p,$ whose fixed field $F$ is then either weakly unramified or totally ramified.

By induction on the degree of the extension, we may choose, for a given uniformiser $\pi_{F(d)}\in \Og_{F(d)},$ an $\Og_{F(d)}$-basis $(e'_{ij})_{ij}$ of $\Og_{L(d)}$ (with $0\leq i \leq e_{L/F}-1$ and $0\leq j \leq f_{L/F}-1$) such that the elements $\pi_{F(d)}^{i(d-1)/{e_{L/F}}}e'_{ij}$ are contained in $\Og_{L}$ and form an $\Og_{F}$-basis of $\Og_L.$ Once again, we distinguish two cases:
\begin{enumerate}
\item The extension $K\subseteq F$ is weakly unramified.  Then we put $\pi_{F(d)} := \pi_{K(d)}.$ Any $\Og_K$-basis $(e''_\lambda)_\lambda$ of $\Og_F$ (with $0\leq \lambda \leq f_{F/K}-1$) is automatically also an $\Og_{K(d)}$-basis of $\Og_{F(d)}.$ In particular, $(e'_{ij}e''_\lambda)_{ij\lambda}$ is an $\Og_{K(d)}$-basis of $\Og_{L(d)}$ such that the elements $\pi_{K(d)}^{i(d-1)/{e_{L/K}}}e'_{ij}e''_{\lambda}$ form an $\Og_K$-basis of $\Og_L.$
\item The extension $K\subseteq F$ is totally ramified. Choose a uniformiser $\pi_F\in \Og_F$ and put $\pi_{F(d)}:=\pi_{K(d)}^{-(d-1)/e_{F/K}}\pi_F,$ which is a uniformiser of $\Og_{F(d)}.$ Choose a unit $\epsilon\in \Og_{F(d)}^\times$ such that $\pi_{F(d)}^{e_{F/K}} = \epsilon \pi_{K(d)}.$ We know that the elements $\pi_{K(d)}^{\mu(d-1)/{e_{F/K}}}\pi_{F(d)}^\mu$ (with $0\leq \mu\leq e_{F/K}-1$) are contained in $\Og_F$ and form an $\Og_K$-basis of $\Og_F,$ and that the elements $\pi_{F(d)}^\mu$ form an $\Og_{K(d)}$-basis of $\Og_{F(d)}.$ It follows that the elements
$$\pi_{K(d)}^{e_{L/F}\mu(d-1)/e_{L/K} + i(d-1)/e_{L/K}} \pi_{F(d)}^\mu\epsilon^{i(d-1)/e_{L/K}}e'_{ij}$$ are contained in $\Og_L$ and form an $\Og_K$-basis of $\Og_L.$ Since the elements $\pi_{F(d)}^\mu\epsilon^{i(d-1)/e_{L/K}}e'_{ij}$ clearly form an $\Og_{K(d)}$-basis of $\Og_{L(d)},$ the claim now follows from the fact that the map
\begin{align*}\{(i,\mu) \colon 0\leq i \leq e_{L/F}-1, 0 \leq \mu \leq e_{F/K}-1\} \to \{\lambda \colon 0 \leq \lambda \leq e_{L/K}-1\}.\end{align*}
given by $(i,\mu) \mapsto e_{L/F}\mu + i$ is bijective.%\qedhere
\end{enumerate}
\end{proof}

\begin{proposition}\label{thm:jumps}With the same notation as in Proposition \ref{prop:coker}, the $d$-jumps of the induced torus $\indtorus{L}{K}$ are \[0,\frac{d-1}{e},\ldots,(e-1)\frac{d-1}{e},\] each with multiplicity $f$; its jumps are $0,\dfrac{1}{e},\ldots,\dfrac{e-1}{e}$, each with multiplicity $f$.
\end{proposition}
\begin{proof}
    Let $G=\indtorus{L}{K}$ and let $d$ be a positive integer greater than $e$ and such that $d\equiv1\mod p[L:K]$. By \cite[Chapter 10.1, Example 5]{BLR}, the identity component of Néron lft-model of $\mathbf{G}_{m,L}$ is just $\mathbf{G}_{m,\Og_L}$, and since Weil restriction is compatible with the formation of Néron lft-models \cite[Chapter 7.6, Proposition 6]{BLR} and the formation Lie algebras of a group scheme commutes with Weil restriction, we can identify $\Lie(\mathscr{G}\times_{\Og_K}\Og_{K(d)})$ with $\Og_L\otimes_{\Og_K}\Og_{K(d)}$ and $\Lie\mathscr{G}(d)$ with $\Og_{L\otimes_KK(d)}$ as $\Og_{K(d)}$-modules. Under these identifications, the map \[\Lie h(G,d)\colon\Lie(\mathscr{G}\times_{\Og_K}\Og_{K(d)})\to\Lie\mathscr{G}(d)\] corresponds to % the unique map $\Og_L\otimes_{\Og_K}\Og_{K(d)}\to\Og_{L\otimes_KK(d)}$ making the diagram of $\Og_{K(d)}$-modules 
    %\[\begin{tikzcd}         \Og_L\otimes_{\Og_K}\Og_{K(d)}\arrow{rr}\arrow{dr} && \Og_{L\otimes_KK(d)}\arrow[hook]{dl}\\ & L\otimes_KK(d) &
     %\end{tikzcd},\] 
    %where the diagonal arrows are the obvious ones, commute, i.e. 
    the one in the statement of Proposition \ref{prop:coker}. Therefore, the $d$-jumps of $G$ are $0, \frac{d-1}e,\ldots,\frac{(e-1)(d-1)}e$, each with multiplicity $f$.
    
    By Proposition \ref{jumpslimitprop} we can compute the jumps of $G$ using any grid sequence $(d_\ell)_\ell$ such that $d_\ell\equiv 1\mod p[L:K]$ for all $\ell;$ picking any such sequence, we see that
    \[\lim\limits_{\ell\to\infty}\frac{i(d_\ell-1)}{d_\ell e}=\frac ie\] is a jump of $G$ for $i=0,\ldots,e-1$ (each occurring with multiplicity $f$) as stated.
\end{proof}

\begin{example} \label{ctameexample}
    If $K\subseteq L$ is a finite Galois extension with ramification index $e$ and inertia index $f$, one obtains the formula $$c_{\mathrm{tame}}(\Res_{L/K}\Gm) = \frac{f(e-1)}{2}.$$ This shows in particular that Theorem \ref{ctameadditivethm} no longer holds if the residue field is imperfect. Indeed, suppose that $\Og_K$ is of equal characteristic 2 with imperfect residue field. It was shown in \cite[Subsection 4.4.1]{OS} that there exists a Galois extension $K\subseteq L$ such that $\Gal(L/K)\cong \Z/2\Z \times \Z/2\Z$ and such that for any quadratic subextension $K\subseteq F \subseteq L,$ the extension $K\subseteq F$ is weakly unramified and $F\subseteq L$ is totally ramified. By Galois theory, there are precisely 3 quadratic subextensions, which we shall call $K_1, K_2$ and $F.$ Put $T_i:=(\Res_{K_i/K}\Gm)/\Gm$ for $i=1,2$ and $T:=\Res_{L/K}\Gm/\Res_{F/K}\Gm.$ Using Lemma \ref{universallyexactlem} and Proposition \ref{Chaiexactn}, we find that $c_{\mathrm{tame}}(T_i)=0$ for $i=1,2,$ and $c_{\mathrm{tame}}(T)=1.$ However, as in Example \ref{Example01} (cf. also \cite[Subsection 4.4.1]{OS}), we have an exact sequence $0 \to T_1 \to T \to T_2 \to 0,$ in which $c_{\mathrm{tame}}(-)$ is not additive. This example also shows that Halle-Nicaise's formula $c_{\mathrm{tame}}(T)=\frac{1}{2}u(T)$ does not generalise to imperfect residue fields, and moreover, since $T$ is isogenous to $T_1\times_KT_2$, that neither does isogeny invariance. 
\end{example}

A multiset $\mathcal{M}$ with underlying set $I$ will be written as $\{(a;\mu(a))\}_{a\in I},$ where $\mu(a)$ is the multiplicity of the element $a$.

\begin{corollary}\label{thm:indtorusrationality}
    Let $M/K$ be a finite separable extension of strictly henselian discrete valuation rings with Galois closure $L/K$ and let $e=e_{L/K}$, $f=f_{L/K}.$ Letting $T$ be the induced torus $T=\indtorus{M}{K}$, one has \[J_d(T)\subseteq\left\{(0;f),\left(\frac{d-1}e;f\right)\ldots,\left((e-1)\cdot\frac{d-1}e;f\right)\right\}\] for any $d\equiv1\mod p[L:K]$, $d>e$ and \[J(T)\subseteq\left\{(0;f),\left(\frac1e;f\right),\ldots,\left(\frac{e-1}e;f\right) \right\}.\] 
In particular, the jumps of $T$ are rational numbers.
\end{corollary}
\begin{proof}
    Letting $\Gamma=\Gal(L/K)$ and $\Delta=\Gal(L/M),$ by \cite[Satz 0.4.3]{Brahm} the canonical surjection of $\Z[\Gamma]$-modules $\Z[\Gamma]\to\Z[\Gamma/\Delta]$ induces, for each $d$ as in the statment, a short exact sequence of tori \[0\to\indtorus {M(d)}{K(d)}\to\indtorus {L(d)}{K(d)}\to T(d)\to0.\] Notice also that $e_{M/K}=e_{M(d)/K(d)}$ and $f_{M/K}=f_{M(d)/K(d)}$.
    
     By Proposition \ref{Chaiexactn}, the sequence of Néron lft-models induced by the previous sequence remains exact; therefore, we can apply Proposition \ref{Chaiexactn}, whereby the statement follows from Lemma \ref{universallyexactlem} and the computation in Proposition \ref{thm:jumps}.
\end{proof}

In view of the computation in {Proposition \ref{thm:jumps}}, one might wonder whether the formation of identity components of Néron lft-models commutes not only with unramified extensions, but also with {weakly} unramified ones; if it were, the jumps of a $K$-torus $T$ would remain unchanged after such an extension $K'/K$. We now present two counterexamples to this statement: {one in which $K'/K$ is linearly disjoint from the splitting extension of $T$ and one in which $K'$ coincides with it.}

\begin{example}[A non-split torus $T$ with $c_{\mathrm{tame}}(T)=0$ and an extension $K'/K$ such that $c_{\mathrm{tame}}(T_{K'})\neq 0$]\label{ex:Neronwunbc1}
Let $K$ be a mixed-characteristic complete discretely valued field with imperfect residue field $k$. Assume that $K$ has $p$-th roots of unity and it has absolute ramification greater than one; such $K$ can always be obtained by taking suitable extensions of a Cohen ring. 
Let $\lambda\in\Og_K$ be an element such that $\bar\lambda\in k\setminus k^p$ and, fixing a separable closure $K\sep$ of $K$, consider $\alpha,\beta\in K\sep$ such that $\alpha^p=\lambda$, $\beta^p=\pi-\lambda$ for a uniformiser $\pi$ of $K$. 

The two extensions $F=K(\alpha)$ and $K'=K(\beta)$ are Galois over $K$, and their compositum $F'=K(\alpha,\beta)$ is totally ramified over $K'$: indeed, let $\nu_{F'}$ be the normalised valuation on $F'.$ Then one has \[(\alpha+\beta)^p\equiv\alpha^p+\beta^p=\pi \mod p,\] and the inequality $\nu_{F'}(p)>\nu_{F'}(\pi)$ implies that $p\nu_{F'}(\alpha+\beta)=\nu_{F'}(\pi)$. \color{black} This also implies that $F\cap K'=K$, and by Galoisness, that the two extensions are linearly disjoint \cite[Corollary 2]{Cohntensor}. In particular, $F\otimes_KK' = F'.$ 

The torus $T=\indtorus{F}{K}$ is not split, but $c_{\mathrm{tame}}(T)=0$; on the other hand, $T_{K'}=\indtorus{F'}{K'},$ and since $F'/K'$ is totally ramified of degree $p$, by the previous computation its jumps are $0,\tfrac 1p,\ldots,\tfrac{p-1}p$. %This shows in particular that the formation of identity components of Néron models of tori is not compatible with weakly unramified base-change, even along extensions which are linearly disjoint from the splitting extension.
\end{example}\label{ex:Neronwunbc2}
\begin{example} Let $F/K$ be a non-trivial weakly unramified extension of (say) strictly henselian fields of finite degree equal to the residue characteristic $p$ of $K$ with purely inseparable residue field extension $k\subseteq k_F$: for an example in mixed characteristic, let $k$ be any imperfect and separably closed field, $\Og_K$ be a Cohen ring with residue field $k$, $\lambda\in\Og_K$ such that $\bar\lambda\in k\setminus k^p$ and $P=X^p-\lambda$. In equal characteristic, let $\Og_K=\powser{k}{t}$, $\lambda\in k\setminus k^p$ and $P=X^p-tX+\lambda$. Letting $F=K[X]/\langle P\rangle$, in both cases $F/K$ is weakly unramified and $\Og_F=\Og_K[x]/\langle P\rangle$ by \cite[Proposition 15]{SerreLF}. 
 
Let $G=\indtorus{F}{K}$, so that $\mathscr{G}^{\mathrm{0}}=\indtorus{\Og_F}{\Og_K}$. By \cite[VI.A, Proposition 3.3]{SGA3}, \[(\mathscr{G}\times_{\Og_K}{\Og_F})^{\mathrm{0}}=\mathscr{G}^{\mathrm{0}}\times_{\Og_K}{\Og_F},\] which is not isomorphic to the identity component $\mathscr{G}_F^0$ of the Néron lft-model of $G\times_KF$, because the maximal torus of its special fibre $\indtorus{k_F}{k}$ has dimension 1, whereas the maximal torus of $\mathscr{G}^0_{F, k_F}$ has dimension $\geq 2.$ The last claim follows using Lemma \ref{toricranklem}, because $F\otimes_KF$ is the product of at least two finite separable extensions of $F$ as an $F$-algebra.     

\end{example}

\subsection{Jumps of $K$-tori which are direct factors of $K$-rational varieties}
\sectionmark{Jumps of $K$-tori which are direct factors...}
This section is devoted to the study of jumps of a particular class of $K$-tori, namely, those which are direct factors of $K$-rational varieties (Definition \ref{def:Kratvar}): Corollary \ref{jumpsformulas} describes the jumps of such a torus $T$ in terms of a flasque resolution of $T$ (a notion which will be recalled below). 
    Let us first review some definitions and results from \cite{CTSReq} and \cite{EndoMiyata}. 

Let $G$ be a finite group (which for us will be the Galois group of a Galois splitting extension of a torus). For any $G$-module $A$ and any integer $i$, denote by $\hat H^i(G,A)$ the $i$-th Tate cohomology group \cite[§VIII.1]{SerreLF}.
\begin{definition}[{\cite[§1]{CTSReq}}]A $\Z$-free $G$-module of finite rank (over $\Z$) is called
\begin{enumerate}
    \item \textit{a permutation module} if it admits a $\Z$-basis permuted by $G$;
    \item \textit{invertible} if it is a direct factor of a permutation module;
    \item \textit{flasque} if $\hat H^{-1}(H,A)=0$ for all subgroups $H$ of $G$.
\end{enumerate}
\end{definition}

\begin{definition}Let $A$ be a $\Z$-free $G$-module of finite rank.
    A flasque resolution of $A$ is a short exact sequence \[0\to A\to P\to F\to 0\] with $P$ a permutation $G$-module and $F$ a flasque $G$-module.
\end{definition}

Any $\Z$-free $G$-module of finite rank admits a flasque resolution:
\begin{lemma}[{\cite[Lemme 3,5]{CTSReq}, \cite[Lemma 1.1]{EndoMiyata}}]
    Any $\Z$-free $G$-module of finite rank $A$ admits a flasque resolution \[0\to A\to P\to F\to 0;\] moreover, the $G$-module $F$ is unique up to addition of permutation modules. 
\end{lemma}

The previous notions can be also formulated for tori, by the duality between $K$-tori that split over a Galois extension $L/K$ with Galois group $G$ and $\Z$-free $G$-modules of finite rank. 
    It is easy to see that a torus $T$ is a permutation torus if and only if it is isomorphic to a finite product of induced tori; in particular, both the properties of being a permutation and an invertible torus are stable under base change along a finite separable extension of the base field.

\begin{definition}\label{def:Kratvar}A $K$-rational variety is a variety which is $K$-birational to a finite dimensional affine space over $K$; a $K$-torus $T$ is a direct factor of a $K$-rational variety if there exist a $K$-variety $Y$ such that $Y\times_KT$ is a $K$-rational variety.    
\end{definition}
The following result, due to Colliot-Thélène and Sansuc, provides conditions which are equivalent to the invertibility of the torus $F$; several more are listed in \cite[Proposition 7.4]{CTSTori}.
\begin{proposition}[{\cite[Proposition 7.4]{CTSTori}}]\label{prop:CTSconditions}Let $0\to F\to P\to T\to 0$ be a flasque resolution of a torus $T$ over a field $K$. The following conditions are equivalent:\begin{enumerate}
    \item $F$ is an invertible $K$-torus;
    \item $T$ is a direct factor of a $K$-rational variety;
    \item for any field extension $K\hookrightarrow K'$, $H^1(K',F_{K'})=0.$
\end{enumerate}    
\end{proposition}

{To put this into some context, recall that, by a theorem by Endo and Miyata {\cite[Theorem I]{EndoMiyata}}, all the Sylow subgroups of a finite group $G$ are cyclic if and only if any flasque $G$-module is invertible. Moreover, the Galois group of a finite Galois extension $L/K$ of discretely valued strictly henselian fields with residue field of positive characteristic $p$ is a semi-direct product of a cyclic group of order coprime to $p$ and a (normal) $p$-subgroup $H_p$ if $p>0$; therefore, if $H_p$ is cyclic, any torus split by $L/K$ satisfies the conditions of the previous proposition.}

\begin{lemma}\label{prop:Neronflasqueex}
    Let $T$ be a torus over a discretely valued field $K$. If $T$ is a direct factor of a $K$-rational variety, then the sequence of Néron lft-models \[0\to \mathscr{F}\to \mathscr{P}\to \mathscr{T}\to 0\] induced by any flasque resolution $0\to F\to P\to T\to 0$ of $T$ is exact.
\end{lemma}
\begin{proof}Let $L/K$ be a Galois splitting extension of $T$.
    By Proposition \ref{Chaiexactn}, it suffices to show that $R^1j_\ast^{\mathrm{sm}}F=0$. As $T$ is a direct factor of a rational $K$-variety, $F$ is invertible by Proposition \ref{prop:CTSconditions}, so that there exists a torus $F'$ such that $P':=F\oplus F'$ is a finite product of induced tori.
    Therefore, the identity of $R^1j_*^{\mathrm{sm}}F$ factors through $R^1j^{\mathrm{sm}}_*P'$, which vanishes by Proposition \ref{Chaiexactn}. 
\end{proof}

\begin{corollary}\label{jumpsformulas}
\begin{enumerate}
    \item Let $T$ be a $K$-torus which is a direct factor of a $K$-rational variety and  $0\to F\to P\to T\to 0$ flasque resolution; then \[J_d(P)=J_d(F)\uplus J_d(T)\] for any $d$ coprime to $p$.
    \item If $d\equiv 1\mod e,$ where $e$ is the ramification index of the minimal splitting extension of $T$, then all the $d$-jumps of $T$ are of the form $a(d-1)/e,$ with $0\leq a\leq e-1.$ In particular, the jumps of $T$ are rational numbers of the form $a/e,$ with $a$ as before.
\end{enumerate}
\end{corollary}
\begin{proof}
Since the property of being an invertible/permutation torus is stable under base-change, for any $d$ coprime to $p$ \[0\to F(d)\to P(d)\to T(d)\to 0\] is a flasque resolution of $T(d)$ and $F(d)$ is invertible. By the previous lemma, the sequence \[0\to \mathscr{F}(d)\to \mathscr{P}(d)\to \mathscr{T}(d)\to 0\] is again exact, whereby one can apply Lemma \ref{universallyexactlem} and obtain the first statement.

    The second point follows directly from this and the computation of the jumps of an induced torus in Corollary \ref{thm:indtorusrationality}.
\end{proof}

\subsection{A unipotent example}
Finally, we shall give an example of a wound unipotent algebraic group with rational jumps in the case of an imperfect residue field. Suppose that $[K:K^p]=p^r$ for some positive integer $r$ and let $\pi_K$ be a uniformiser of $\Og_K.$ Since $\Og_K$ is excellent, this implies that $[k:k^p]=p^{r-1}$ \cite[Lemma 2.1]{OvII}. Choose elements $t_1,..., t_{r-1}$ which map to a $p$-basis of $k.$ Then $t_1,..., t_{r-1}, \pi_K$ form a $p$-basis of $\Og_K.$ The choice of $p$-basis yields an isomorphism $\mathbf{G}_{\mathrm{a}}^{(1/p)} \cong \mathbf{G}_{\mathrm{a}}^{p^r}.$ Now consider the unipotent algebraic group $\boldsymbol{\nu}_1(r)_K$ which is given by the equation
$$x_{0...0} - \sum _{0 \leq i_1,..., i_r \leq p-1} t_1^{i_1}\cdot ... \cdot t_{r-1}^{i_{r-1}} \pi_K^{i_r} x_{i_1...i_r}^p = 0$$ inside $\mathbf{G}_{\mathrm{a}}^{p^r}.$ See \cite[Section 8]{OSII} for more details on this algebraic group. In particular, we recall from \cite[Proposition 8.2]{OSII} that the smooth group scheme $\boldsymbol{\nu}_1(r)_{\Og_K}^\sim$ described by the same equation over $\Og_K$ is the Néron model of $\boldsymbol{\nu}_1(r)_{K}.$ Now let $d\equiv 1 \mod p$ and choose a uniformiser $\pi_{K(d)}$ (for simplicity) such that $\pi_{K(d)}^d=\pi_K.$ The algebraic group $\boldsymbol{\nu}_1(r)_{K(d)}$ is then given by the equation above with $\pi_K$ replaced by $\pi_{K(d)}.$ We have an isomorphism $\boldsymbol{\nu}_1(r)_{K}\times_KK(d)\to \boldsymbol{\nu}_1(r)_{K(d)}$ given by the change of variables $x_{i_1...i_r}\mapsto \pi_{K(d)}^{-i_r(d-1)/p}x_{i_1...i_r}.$
\begin{proposition}
Let $d\equiv 1 \mod p.$ Then we have an exact sequence
$$0\to \Lie \boldsymbol{\nu}_1(r)_{\Og_K}^\sim \otimes_{\Og_K}\Og_{K(d)} \to \Lie\boldsymbol{\nu}_1(r)_{\Og_K(d)}^\sim \to \bigoplus_{i=1}^{p-1}\Big(\Og_{K(d)}/\langle\pi_{K(d)}^{i\frac{d-1}{p}}\rangle\Big)^{\oplus p^{r-1}} \to 0.$$ In particular, the jumps of $\boldsymbol{\nu}_1(r)_K$ are $i/p$ for $1\leq i \leq p-1$ (each with multiplicity $p^{r-1}$) and $0$ (with multiplicity $p^{r-1}-1).$
\end{proposition}
\begin{proof}
Using the canonical identification $\Lie \mathbf{G}_{\mathrm{a}}^{(1/p)} = \Og_K$ (with $\Og_K$ acting via the Frobenius), the explicit description of the Néron model recalled above shows that the elements $t_1^{i_1}\cdot ... \cdot t_{r-1}^{i_{r-1}}\pi_K^{i_r}$ (with $0\leq i_1,..., i_r \leq p-1$ and $(i_1,..., i_r) \not=(0,...,0)$) form a basis of $\Lie \boldsymbol{\nu}_1(r)_{\Og_K}^\sim.$ Note that, since the extension $\Og_K\subseteq \Og_{K(d)}$ is totally ramified, the elements $t_1,..., t_{r-1}, \pi_{K(d)}$ form a $p$-basis of $\Og_{K(d)}.$ We thus obtain a basis $e_{i_1...i_r}$ of $\Lie \boldsymbol{\nu}_1(r)_{\Og_{K(d)}}^\sim$ indexed as above such that the elements $\pi_{K(d)}^{i_r(d-1)/p}e_{i_1...i_r}$ form a basis of the image of $\Lie \boldsymbol{\nu}_1(r)_{\Og_K}^\sim \otimes_{\Og_K}\Og_{K(d)} \to \Lie\boldsymbol{\nu}_1(r)_{\Og_{K(d)}}^\sim.$ This immediately implies the first claim; the second one follows using Proposition \ref{jumpslimitprop}.
\end{proof}
\begin{remark}
    The formula for the tame base change conductor suggested (conjecturally) in Remark \ref{ctamerem} for algebraically closed $k$ fails for $\boldsymbol{\nu}_1(r)_K$ if $r>1.$ However, from the explicit description of the Néron lft-model given above we can read off that $$c_{\mathrm{tame}}(\boldsymbol{\nu}_1(r)_K)=\frac{1}{2}\rho_{\mathrm{us}}(\boldsymbol{\nu}_1(r)^{\sim}_{\Og_K}\times_{\Og_K} k)$$ where $\rho_{\mathrm{us}}(-)$ denotes the dimension of the maximal \it split \rm unipotent closed subgroup of a smooth algebraic group over $k.$ An analogous formula holds for tori induced along a finite Galois extension $K\subseteq L$. Indeed, let $e$ and $f$ be the ramification index and the inertial degree of $K\subseteq L,$ respectively, and let $A:=\Og_L\otimes_{\Og_K} k.$ By Example \ref{ctameexample}, we must show that the maximal split unipotent closed subgroup of $\Res_{A/k}\Gm$ is of dimension $f(e-1).$ Choosing a sequence $0 = \mathfrak{m}_0 \subseteq ... \subseteq \mathfrak{m}_r \subseteq A$ of ideals such that $\mathfrak{m}_r$ is the maximal ideal and $\mathfrak{m}_i^2\subseteq \mathfrak{m}_ {i-1}$ for all $i=1,...,r,$ one shows easily by induction that the canonical map $\Res_{A/k}\Gm \to \Res_{(A/\mathfrak{m}_r)/k}\Gm$ is smooth, surjective, and has split unipotent kernel, whereas $\Res_{(A/\mathfrak{m}_r)/k}\Gm$ is wound. Here one uses the fact that, for any commutative ring $R$ with unity and any ideal $I\subseteq R$ such that $I^2=0,$ we have an exact sequence $0 \to I \to R^\times \to (R/I)^\times \to 0,$ where the map $I\to R^\times$ is given by $\eta\mapsto 1+\eta.$ The dimensions of source and target of the canonical map are $ef$ and $f,$ respectively, so the claim follows. It would be very interesting to know to what extent this formula can be generalised.
    \end{remark}

The preceding remark suggests that the split unipotent radical of the special fibre of a Néron lft-model controls the base change behaviour in some sense. A slightly weaker claim can be easily shown, which we include here since it may be of general interest: 

\begin{proposition}
    Let $\Og_K$ be an excellent strictly Henselian discrete valuation ring with field of fractions $K$ and residue field $\kappa.$ Let $G$ be a smooth connected wound algebraic group over $K$ with Néron lft-model $\mathscr{G}\to \Spec \Og_K.$ Then the following claims hold: 
    \begin{enumerate}[(i)]
        \item Let $\mathscr{G}'$ be a smooth separated model of $G$ over $\Og_K$ with connected fibres such that $\rho_{\mathrm{us}}(\mathscr{G}'_{\kappa})=0.$ Then the canonical map $\mathscr{G}'\to \mathscr{G}^0$ is an isomorphism.
        \item  Let $\Og_K\subseteq \Og_L$ be a flat extension of excellent strictly Henselian discrete valuation rings (with no restriction on the ramification index), and assume that the induced extensions $K\subseteq L$ (resp. $\kappa \subseteq \kappa_L$) of fields of fractions (resp. residue fields) are \rm separable \it (not necessarily finite). If $\rho_{\mathrm{us}}(\mathscr{G}^0_{\kappa})=0,$ then the identity component of the Néron lft-model of $G$ commutes with base change along $\Og_K\subseteq \Og_L.$
    \end{enumerate}
\end{proposition}
\begin{proof}
    (i) By \cite[Corollary 2.3]{LLR}, the morphism $\mathscr{G}' \to \mathscr{G}^0$ is the composition $\mathscr{G}'= \mathscr{G}_0 \to \mathscr{G}_1 \to ... \to \mathscr{G}_r = \mathscr{G}^0$ of finitely many dilatations \cite[Chapter 3.2]{BLR} in smooth centres; moreover, the centres are easily seen to be connected because so is $\mathscr{G}'_{\kappa}.$ We may assume that $r$ is minimal with this property and would like to show that $r=0.$ Assume for the sake of a contradiction that $r\geq 1.$ If the centre of the dilatation $\mathscr{G}' \to \mathscr{G}_1$ were equal to $\mathscr{G}_{1,\kappa},$ then $\mathscr{G}'\to \mathscr{G}_1$ would be an isomorphism, contradicting the minimality of $r.$ Since $\mathscr{G}_{1, \kappa}$ is connected (which follows by induction from the fact that the centres are connected), this implies that the centre of $\mathscr{G}'\to \mathscr{G}_1$ has positive codimension. But then the construction of the dilatation immediately shows that the kernel of $\mathscr{G}'_{ \kappa}\to \mathscr{G}_{1, \kappa}$ is isomorphic (as a scheme) to a positive-dimensional affine space, so it contains a copy of $\Ga$ by \cite[Theorem B.2.5]{CGP}, contradicting our assumption on $\mathscr{G}'_{\kappa}.$\\
    (ii) The separability conditions ensure that both $G_L$ and $\mathscr{G}_{\kappa_L}$ are wound \cite[Theorem B.3.4]{CGP}; in particular, $G_L$ admits a Néron lft-model $\mathscr{G}_L$ over $\Og_L.$ Now (ii) follows by applying (i) to the canonical map $\mathscr{G}^0\times_{\Og_K}\Og_L \to \mathscr{G}_L^0.$  
\end{proof}
This proposition is, of course, well known if $\mathscr{G}'$ is a semiabelian scheme over $\Og_K,$ hence in particular if $\kappa$ is algebraically closed.

\noindent\textsc{Mathematisches Institut der Heinrich-Heine-Universität Düsseldorf, Universitätsstr. 1, 40225 Düsseldorf, Germany} \\
\it E-mail address: \rm \texttt{otto.overkamp@uni-duesseldorf.de}\\
\\
\textsc{Mathematisches Institut der Heinrich-Heine-Universität Düsseldorf, Universitätsstr. 1, 40225 Düsseldorf, Germany} \\
\it E-mail address: \rm \texttt{ismaele.vanni@hhu.de}\\
\\

  %  There is an obvious notion of multisubset; for a multiset $\mathcal{M}$, we denote by $\mathscr{P}(\mathcal{M})$ the set of its multisubsets.

    %The sum of a finite collection $\{\mathcal{M}_i\}_{i\in I}$ of multisets is the multiset \[\biguplus\limits_{i\in I}\mathcal{M}_i:=\left(\bigcup\limits_{i\in I}|\mathcal{M}_i|,\mu\right)\] with, for any $x$ in the support, $\mu(x)=\sum\limits_{i\in I}\mu_i(x)$, where $\mu_i\colon\bigcup\limits_{i\in I}|\mathcal{M}_i|\to\N$ extends by $0$ the multiplicity function of $\mathcal{M}_i$. Remark that the support of $\uplus_i\mathcal{M}_i$ is \textit{not} the disjoint union of the supports $|\mathcal{M}_i|$: so, some element $x$ in the sum of two multisets $\mathcal{M}_1$ and $\mathcal{M}_2$ might belong to both, and in this case its multiplicity in the multiset sum is the sum of the multiplicities.


\begin{thebibliography}{10}
%\bibitem{AS}
%Abbes, A., Saito, T. 
%\textit{Ramification of local fields with imperfect residue fields}. Amer. J. Math., Vol. 124, No. 5, pp. 879-920, 2002. 

\bibitem{Achet}
Achet, R.
\textit{Unirational algebraic groups}. Preprint; available at \url{https://hal.science/hal-02358528/document}.

\bibitem{Beg}
Bégueri, L.
\textit{Dualité sur un corps local à corps résiduel
algébriquement clos}. Mémoire de la Société mathématique de France, no. 4, 1980.
\bibitem{BourCommAlg}
N. Bourbaki, \textit{Algebra II: Chapters 4–7}. Springer, 1990. Translated from the French by P. M. Cohn and J. Howie.

\bibitem{Brahm}
Brahm, B. 
\textit{Néron-Modelle algebraischer Tori}. PhD thesis, Westf\"alische Wilhelms-Universit\"at M\"unster, 2003.

\bibitem{BS}
Bertapelle, A., Suzuki, T.
\textit{The relatively perfect Greenberg transform and cycle class maps}. Manuscripta Math. 175(1-2), pp. 365-407, 2024.

\bibitem{BLR}
Bosch, S., L\"utkebohmert, W., Raynaud, M.
\textit{Néron models}. Ergeb. Math. Grenzgeb., Springer-Verlag, Berlin, Heidelberg, 1990.

\bibitem{Chai}
Chai, C.-L.
\textit{Néron models for semiabelian varieties: congruence and change of base field}. Asian J. Math. 4(4), pp. 715–736, 2000.

\bibitem{CY}
Chai, C.-L., Yu, J.-K. (Appendix by E. de Shailt)
\textit{Congruences of Néron models and the Artin conductor}. Ann. of Math. 154, pp. 347-382, 2001.

\bibitem{Cohntensor}
Cohn, M. 
\textit{On the decomposition of a field as a tensor product.}
Glasgow Mathematical Journal, pp. 141–145, 1979

\bibitem{CTSReq}
Colliot-Thélène, J.-L., Sansuc, J.-J.
\textit{La R-équivalence sur les tores}. Ann. scient. Éc. Norm. Sup, pp. 175-230, 1977.

%\bibitem{CTS}
%Colliot-Thélène, J.-L., Sansuc, J.-J.
%\textit{Principal Homogeneous Spaces under Flasque Tori: Applications}. J. Algebra 106, pp. 148-205, 1987.

\bibitem{CTSTori} Colliot-Thélène, J.-L., Sansuc, J.-J. \textit{Principal homogeneous spaces under flasque tori: Applications}. J. Algebra 106, no. 1, 148–205, 1987.

\bibitem{CGP}
Conrad, B., Gabber, O., Prasad, G.
\textit{Pseudo-reductive groups}. 2nd edition, New Mathematical Monographs 26, Cambridge University Press, 2015.

\bibitem{DG}
Demazure, M., Gabriel, P.
\textit{Groupes algébriques}. Masson \& Cie, éditeur, Paris; North-Holland Publishing Company, Amsterdam, 1970.

\bibitem{SGA3}
Demazure, M., Grothendieck, A. (eds.) 
\textit{ Schémas en Groupes. Séminaire de Géométrie Algébrique du Bois Marie 1962–1964} (SGA 3). Augmented and corrected 2008–2011 re-edition of the original by P. Gille and P. Polo. \url{http://www.math.jussieu.fr/~polo/SGA3}. 
\bibitem{Edixhoven}
Edixhoven, B.
\textit{Néron models and tame ramification}. Compositio Math., tome 81, no 3, pp. 291-306, 1992. 

\bibitem{EndoMiyata}
Endo, S., Miyata, T.
\textit{On a classification of the function fields of algebraic tori}. 
Nagoya Math. J. Vol. 56, pp. 85-104, {1975}. With an erratum: \textit{Corrigenda On a classification of the function fields of algebraic tori}  Nagoya Math. J., Vol. {79}, pp. {187–190}, 1980.

\bibitem{EHN}
Eriksson, D., Halle, L. H., Nicaise, J.
\textit{A logarithmic interpretation of Edixhoven’s jumps for Jacobians}. 
Advances in Mathematics 279, 532–574, 2015

\bibitem{Fulton}
Fulton, W.
\textit{Intersection Theory}. 2nd edition, Springer-Verlag, 1998.
\bibitem{Hallgroups}
  Hall Jr., M.
  \textit{The Theory of Groups}. Macmillan, {1959}.

\bibitem{HNComp}
Halle, L.H., Nicaise, J. 
\textit{The Néron component series of an abelian variety.} Math. Ann. 348, pp. 749–778, 2010.

\bibitem{HNII}
Halle, L. H., Nicaise, J. 
\textit{Motivic zeta functions of abelian varieties, and the monodromy conjecture}. Adv. Math. 227, pp. 610–653, 2011. 

\bibitem{HN}
Halle, L. H., Nicaise, J. 
\textit{Néron models and Base Change}. Lecture Notes in Math. 2156, Springer-Verlag, 2016.

\bibitem{Jac}
Jacobson, N. 
\textit{Commutative restricted Lie algebras}. In \it Nathan Jacobson Collected Mathematical Papers. \rm Vol. 2, Contemporary Mathematicians. Birkhäuser, Boston, 1989.

\bibitem{Kato}
Kato, K.
\textit{Duality Theories for the $p$-Primary Etale Cohomology. I}. Algebraic and Topological Theories - to the memory of Dr. Takehiko MIYATA, pp. 127-148, 1985. 

\bibitem{KMW}
Maex, M., Kaya, E., Waeterschoot, A.
\textit{Jumps of Jacobians via orthogonal canonical forms}. Proc. Amer. Math. Soc. 153, pp. 947-961, 2025.

\bibitem{LorLiu} Liu, Q., Lorenzini, D. 
\textit{Special fibers of Néron models and wild ramification.} J. reine angew. Math, pp. 179-222, 2001.

\bibitem{LLR}
Liu, Q., Lorenzini, D., Raynaud, M.
\textit{Néron models, Lie algebras, and reduction of curves of genus one}. Invent. math. 157, pp. 455-518, 2004. 

\bibitem{Milne}
Milne, J. S.
\textit{Étale Cohomology}. Princeton Mathematical Series, Vol. 33, Princeton University Press, Princeton, 1980.
%\bibitem{Neukirch}
%Neukirch, J.         \textit{Algebraic Number Theory}.         Vol. 322. Springer Science and Business Media, 2013. Translated from the German by N. Schappacher.

\bibitem{N}
Nicaise, J.
\textit{Motivic invariants of algebraic tori}. Proc. Amer. Math. Soc., Vol. 139, Nr. 4, pp. 1163-1174, 2011. 

\bibitem{NS}
Nicaise J., Sebag J. 
\textit{The Grothendieck ring of varieties}. In: Cluckers R., Nicaise J., Sebag J., eds. Motivic Integration and Its Interactions with Model Theory and Non-Archimedean Geometry. London Mathematical Society Lecture Note Series. Cambridge University Press, pp. 145-188, 2011.

\bibitem{Oes}
Oesterlé, J.
\textit{Nombres de Tamagawa et groupes unipotents en charactéristique $p$}. Invent. Math. 78(1), pp. 13-88, 1984. 

\bibitem{OvI} Overkamp, O.
\textit{Jumps and Motivic Invariants for semiabelian Jacobians}. Int. Math. Res. Not., Vol. 2019, Issue 20, pp. 6437-6479, 2019. 

\bibitem{OvIII} Overkamp, O.
\textit{On Jacobians of geometrically reduced curves and their Néron models}. Trans. Amer. Math. Soc., Vol. 377, Nr. 8, pp. 5863-5903, 2024.

\bibitem{OvII} Overkamp, O.
\textit{Chai's conjecture for semiabelian Jacobians}. Compositio Math., Vol. 161, Issue 1, pp. 120 - 147, 2025.

\bibitem{OS} Overkamp, O., Suzuki, T. 
\textit{Chai's conjectures on base change conductors}. To appear in J. Algebraic Geometry. 

\bibitem{OSII} Overkamp, O., Suzuki, T. 
\textit{Existence of global Néron models beyond semi-abelian varieties}. Preprint, 2023. 

\bibitem{Pink} Pink, R.
\textit{Finite group schemes}. Preprint; available at \url{https://people.math.ethz.ch/~pinkri/ftp/FGS/CompleteNotes.pdf}

\bibitem{RayI}
Raynaud, M.
\textit{Passage au quotient par une relation d’équivalence plate}. In \it Proceedings of a conference
on Local Fields \rm (Driebergen, 1966), Springer-Verlag, pp. 79–85, 1967.  

\bibitem{Ros} 
Rosengarten, Z.
\textit{Permawound Unipotent Groups}. Transformation groups, 2024.

\bibitem{RosII} 
Rosengarten, Z.
\textit{Rigidity and unirational groups}. Preprint; available at \url{https://arxiv.org/pdf/2307.04649}.

\bibitem{SerreLF}
Serre, J.-P.
\textit{Local Fields}. Graduate Texts in Mathematics, Springer, 1979. Translated from the French by M. J. Greenberg.

\bibitem{Stacks}
Authors of the Stacks project.
\textit{Stacks project}. Columbia University.

\bibitem{Suz}
Suzuki, T.
\textit{Grothendieck's pairing on Néron component groups: Galois descent from the semistable case}. Kyoto J. Math., Vol. 60, No. 2, pp. 593-716, 2020.

\bibitem{SuzII}
Suzuki, T.
\textit{Duality for local fields and sheaves on the category of fields}. Kyoto J. Math., Vol. 62, No. 4, pp. 789–864, 2022.

\bibitem{SuzIII}
Suzuki, T.
\textit{Class field theory, Hasse principles and Picard-Brauer duality for two-dimensional local rings}. Algebraic Geometry 11(4), pp. 460-505, 2024.

\bibitem{SuzIV}
Suzuki, T.
\textit{An improvement of the duality formalism of the rational étale site}. In \it Algebraic number theory and related topics 2018, \rm RIMS Kôkyûroku Bessatsu, B86,
pp. 287–330; Res. Inst. Math. Sci. (RIMS), Kyoto, 2021.

\bibitem{Stix}
Stix, J. 
\textit{A course on finite flat group schemes and
p-divisible groups}. Preprint, 2009. Available at \url{https://www.uni-frankfurt.de/115677822/stix_finflat_grpschemes.pdf}.
\bibitem{Vannithesis}
Vanni, I.
\textit{Greenberg functors and jumps of tori}. PhD thesis, Scuola Dottorale Vito Volterra, 2025.
%\bibitem{ZarSam}
%Zariski, O., Samuel, P. \textit{Commutative Algebra}. Vol. I, Springer, 1958.
\end{thebibliography}
\end{document}